\documentclass{amsart}
\usepackage{amsmath,amsthm,amssymb,latexsym,color}
\usepackage{hyperref}
\usepackage{lipsum}
\usepackage{todonotes}
\usepackage{multirow}

\usepackage{xcolor}
\date{}
\usepackage{pgfplots}
\usepackage{tikz}
\usetikzlibrary{arrows,matrix,positioning}
\usetikzlibrary{decorations.pathreplacing,decorations.pathmorphing,calc}
\usepackage{graphicx}
\newtheorem{theor}{Theorem}[section]
\newtheorem*{theorem*}{Theorem}

\newtheorem{lemma}[theor]{Lemma}
\newtheorem{cor}[theor]{Corollary}
\newtheorem{defi}[theor]{Definition}
\newtheorem{prop}[theor]{Proposition}

\theoremstyle{definition}
\newtheorem{rem}[theor]{Remark}

\theoremstyle{plain}

\newcommand{\N}{\mathbb{N}}

\newcommand{\R}{\mathbb{R}}

\def\Prob{{\mathbb P}}

\newcommand{\supp}{{\rm supp\,}}

\newcommand{\length}[1]{{\rm len}(#1)}

\def\col{{\rm col}}

\def\row{{\rm row}}
\def\dist{{\rm dist}}

\def\tuple{{\mathcal T}}
\def\mult{{\rm mult\,}}
\def\disc{{\rm disc\,}}
\def\Adj{{\rm Adj}}

\def\ione{{\bf i1}}
\def\itwo{{\bf i2}}
\def\jone{{\bf j1}}
\def\jtwo{{\bf j2}}
\def\mdata{{\mathcal D_m}}

\def\gconst{{\mathfrak n}}
\def\cconst{{\mathfrak m}}
\def\zconst{{\mathfrak z}}

\definecolor{oooo}{HTML}{ED7D31}
\newcommand\chng[1]{{#1}}      

\title[Sharp Poincar\'e and log-Sobolev inequalities for the switch chain]{Sharp Poincar\'e and log-Sobolev inequalities for the switch chain on regular bipartite graphs}
\author{
Konstantin Tikhomirov
\and
Pierre Youssef
}
\address{
\medskip
\noindent
Konstantin Tikhomirov,
School~of Math.,
GeorgiaTech,
686 Cherry street,
Atlanta, GA 30332.
\texttt{\small
e-mail:   ktikhomirov6@gatech.edu}
}
\address{
\medskip
\noindent
Pierre Youssef, 
Mathematics, Division of Science, New York University Abu Dhabi, UAE.
\texttt{\small
e-mail:  yp27@nyu.edu}}

\def\R{{\mathbb R}}
\def\N{{\mathbb N}}

\def\trel{{t_{\rm rel}}}

\newcommand{\Perfmatch}{\mathcal{C}_{n,d}}
\newcommand{\Expset}{\mathcal{R}_{n,d}}
\newcommand{\Ugly}{\mathcal{U}_{n,d}}
\newcommand{\Permset}{\mathcal{S}_{nd}}

\def\BipGSet{\Omega^B}

\def\MultBipGSet{{\mathcal M}^B}
\def\ConfBipGSet{\Omega^{B\,C}}
\def\SNeigh{{\mathcal S\mathcal N}}
\def\categ{{\rm Cat}_{n,d}}

\def\tmix{{t_{\rm mix}}}

\def\Prob{{\mathbb P}}
\def\Exp{{\mathbb E}}

\def\Dir{{\mathcal E}}
\def\cf{\mathcal{Q}}
\newcommand{\Ent}{{\rm Ent}}
\newcommand{\Var}{{\rm Var}}
\newcommand{\Psimple}{\pi_{u}}
\newcommand{\Pconfig}{\pi_{BC}}
\newcommand{\Pperm}{\pi}

\def\supp{{\rm supp\, }}

\def\neigh{{\mathcal N}}
\def\dist{{\rm dist}}
\def\cov{{\rm Cov}}

\def\antiexp{{\mathcal Z}}

\tikzset{
	My Style/.style={
        circle,
        draw,
        fill          = black!50,
        inner sep     = 0pt,
        minimum width =4 pt
    }   
}
\newcommand{\eyepic}{
\begin{tikzpicture}[thick,scale=1.9,->,
                   shorten >=2pt+0.5*\pgflinewidth,
                   shorten <=2pt+0.5*\pgflinewidth,
                        ]
\path[draw] 
       node [My Style] at (.2,0) {}  
       node [My Style] at (1.2,1) {} 
       node [My Style] at (.2,1) {}  
       node [My Style] at (1.2,0) {} ;
        \node at (0,0) {$i'$};
        \node at (0,1) {$i$};
        \node at (1.4,0) {$j'$};
        \node at (1.4,1) {$j$};

\begin{scope}	[arrows={-angle 60}]
    \draw (.2,0) -- (1.2,0) ; 
    \draw (.2,1) -- (1.2,1) ;
\end{scope}

\begin{scope}   [arrows={-angle 60}, dashed]   
     \draw (.2,0) -- (1.2,1)  ; 
     \draw (.2,1) -- (1.2,0)  ;
\end{scope}

\end{tikzpicture}
}

\begin{document}

\maketitle

\begin{abstract}
Consider the switch chain on the set of $d$-regular bipartite graphs on $n$ vertices with
$3\leq d\leq n^{c}$, for a small universal constant $c>0$.   
We prove that the chain satisfies a Poincar\'e inequality with a constant of order $O(nd)$; 
moreover, when $d$ is fixed, we establish a log-Sobolev inequality for the chain with a constant of order $O_d(n\log n)$.
We show that both results are optimal.
The Poincar\'e inequality implies 
that in the regime $3\leq d\leq n^c$ the mixing time of the switch chain is at most $O\big((nd)^2 \log(nd)\big)$,
improving on the previously known bound $O\big((nd)^{13} \log(nd)\big)$ due to Kanan, Tetali and Vempala \cite{KTV} and $O\big(n^7d^{18} \log(nd)\big)$ obtained by Dyer et al. \cite{Dyer-arxiv}. 
The log-Sobolev inequality that we establish for constant $d$
implies a bound $O(n\log^2 n)$ on the mixing time of the chain which, up to the $\log n$ factor, 
captures a conjectured optimal bound.
Our proof strategy relies on building,
for any fixed function on the set of $d$-regular bipartite simple
graphs, an appropriate extension to a function on the set of multigraphs given by the configuration model.
We then establish a comparison procedure with the well studied random transposition model in order to obtain the corresponding functional 
inequalities. While our method falls into a rich class of 
comparison techniques for Markov chains on different state 
spaces, the crucial feature of the method --- dealing with chains with a large distortion between their stationary measures
--- is a novel addition to the theory.
\end{abstract}
\tableofcontents

\section{Introduction}
Regular graphs or more generally graphs with predefined degree sequences have been popular in applications such as
network analysis, and the active study of these models over past decades
has spawned a large amount of research literature. 
Besides their practical importance, the study of those graphs is interesting from a purely theoretical viewpoint. 
One of the basic problems is sampling uniformly at random from the set of graphs with a predefined degree sequence. 
A conventional method for obtaining an exact uniform sampler is through the use of the \textit{configuration model} \cite{BC,B} (see also \cite[Section~11.1]{Frieze-book}). 
However, a serious drawback in this approach is that the configuration model tends to create multiple edges and the probability of it being simple 
decays very fast as the degree grows (see for example \cite{W-survey}). 
A number of research papers has appeared with algorithms intended to sample regular graphs uniformly, either exactly or approximately.
We refer, in particular, to \cite{MW, W, F, JS0, JS, BKSaberi, SW, KV, GaoW, JansonUG}. 

A general method to sample random elements from some
set of objects is via rapidly mixing Markov chains. 
In the context of graphs with predefined degree sequences, a popular Markov chain --- {\it the switch chain} --- has been extensively studied \cite{KTV,CDG,Greenhill,Greenhill2,MES,EKMS,EMMS,BHS,AK,unified}. It relies on a local operation called {\it the simple switching} which can be described as follows: given a graph $G$ with a predefined degree sequence, take two
non-incident edges $i\to_G j$ and $i'\to_G j'$, and replace them by $i'\to j$ and $i\to j'$
whenever this doesn't introduce multiple edges (see Figure~\ref{fig: switch}).
Note that the simple switching keeps the degree sequence of the graph invariant.
The simple switching was introduced (for general graphs) by Senior \cite{Senior} (in that paper, it was called ``transfusion''); in
the context of regular graphs it was first applied by McKay \cite{Mckay}, and since then has proved to be very useful in problems requiring
certain information about the structure of a typical regular graph, essentially reduced to estimating cardinalities of some subsets of regular graphs.
As just one of such examples, we would like to mention a line of research dealing with the limiting spectral distribution
of random directed regular graphs \cite{Cook-inv,Cook-circ,BCZ,LLTTY:15,LLTTY-rank,LLTTY-sing,LLTTY-ker,LLTTY-circ}.

\begin{figure}[t]
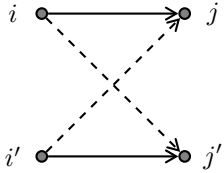


\begin{center}

\eyepic

\end{center}
\caption{The simple switching.}\label{fig: switch}
\end{figure}
Markov chains based on switchings have been introduced by Besag and Clifford \cite{BClifford} for bipartite graphs, Diaconis and Sturmfels \cite{DS3} for contingency tables, and Rao et al. \cite{RJB} for directed graphs.  

Despite the enormous amount of study of the switch chain on various models of graphs, the mixing time is still to be determined exactly. The known polynomial bounds are very far from the truth as we will discuss later on.
One of our motivations was to initiate a line of research
aiming at reaching the optimal mixing time estimates.  
Our focus in this paper will be on the switch chain on bipartite regular graphs.
We leave the study of the uniform undirected $d$--regular model to future works.

For any $n\in \N$ and any $2\leq d\leq n$, let $\BipGSet_n(d)$
be the set of all bipartite $d$--regular graphs without multiple edges on $[n^{(\ell)}]\sqcup[n^{(r)}]$ equipped with the uniform measure $\Psimple$. 
The switch chain is defined as follows: for every graph $G\in \BipGSet_n(d)$, we pick two edges of $G$ independently uniformly at random
(there are $nd(nd-1)$ choices for the ordered pair of the edges).
If the simple switching operation on the edges is admissible (i.e.\ the edges are not incident and the switching
does not introduce multiedges)
then the switching defines the transition to another graph. Otherwise, if the switching
is not admissible, we stay at the same graph $G$. For any two distinct graphs $G_1,G_2$, the transition probability from $G_1$ to $G_2$ takes one of the two values
$0$ or $\big(nd(nd-1)/2\big)^{-1}$.
Accordingly, the Markov generator of the switch chain is given by
$$
Q_u(G_1,G_2):=\begin{cases}
-\frac{|\neigh(G_1)|}{nd(nd-1)/2},&\mbox{if $G_1=G_2$};\\
\big(nd(nd-1)/2\big)^{-1},&\mbox{if $G_2\in \neigh(G_1)$};\\
0,&\mbox{otherwise.}
\end{cases}
$$
Here, $\neigh(G_1)$ denotes the set of all graphs in $\BipGSet_n(d)$ which can be obtained
from $G_1$ by the simple switching operation.  It is easy to see that the chain defined in this way has uniform stationary distribution and that it is reversible i.e. $\Psimple(G_1)Q_u(G_1,G_2)= \Psimple(G_2)Q_u(G_2,G_1)$. The {\it mixing time $\tmix(\varepsilon)$} is formally defined as
$$
\tmix(Q_u,\varepsilon)=\max_{G\in \BipGSet_n(d)} \min\{t\geq 0:\, \Vert P_t(G,\cdot)-\Psimple\Vert_{TV}\leq \varepsilon\},
$$
where $P_t=e^{tQ_u}$ refers to the underlying Markov semi-group and $\Vert\cdot\Vert_{TV}$ denotes the total variation distance. 

In an influential work, Kannan, Tetali and Vempala \cite{KTV} studied the mixing time of the switch chain on regular bipartite graphs. They showed that the conductance of the chain is at least of order $(nd)^{-6}$, which combined with the method of Jerrum and Sinclair \cite{JS-conductance}  implied that for any $\varepsilon\in (0,1)$
\begin{equation}\label{eq: mixing KTV}
\tmix(Q_u,\varepsilon)\leq  C(nd)^{12}\big( \log \vert \BipGSet_n(d)\vert +\log \varepsilon^{-1}\big),
\end{equation}
for some universal constant $C$. When $d$ is small enough, Dyer et al. \cite{Dyer-arxiv} improved on this bound by obtaining 
\begin{equation}\label{eq: mixing Dyer}
\tmix(Q_u,\varepsilon)\leq  Cn^6d^{17}\big( nd\log (nd) +\log \varepsilon^{-1}\big),
\end{equation}
for some universal constant $C$. 
The work of \cite{KTV} was followed by a number of results on the mixing time of 
the switch chain for several models.
To name a few, the switch chain was studied for regular undirected graphs \cite{CDG}, regular directed graphs \cite{Greenhill}, half-regular bipartite graphs \cite{MES}, irregular graphs and digraphs \cite{GS}. We refer to \cite{unified} for a recent unified approach to these results and a complete account of the references. In most of these works, a multicommodity flow argument \cite{Sinclair} was used to estimate the mixing time. As we will see below, our results based on establishing functional inequalities will imply a major improvement, under some growth condition on $d$, of the estimate of \cite{KTV} for bipartite graphs. We believe that our approach can be extended to cover other regular models. 

Functional inequalities have proved to be powerful tools to obtain bounds on the mixing time of Markov chains. 
However, those inequalities are important and interesting on their own right
in view of their close relation to the concentration of measure phenomenon (see \cite{ledoux}).
Our aim in this paper is to derive optimal Poincar\'e and log-Sobolev inequalities for the switch chain on regular bipartite graphs. 
In the context of Markov chains, these inequalities aim at comparing a Dirichlet form associated with the chain 
to the variance or entropy associated with its stationary measure. Given a probability measure $\mu$ on a finite state space $S$ and a reversible Markov generator $Q$, the associated 
Dirichlet form is given by 
$$
\Dir_{\mu}(f,f):=-\langle Qf,f\rangle_{\mu}= -\Exp_{\mu}[fQf]=\frac{1}{2}\sum_{x, y\in S}Q(x,y)(f(x)-f(y))^2\mu(x),
$$
where $Qf(x)=\sum_{y\in S} Q(x,y)f(y)$ 
for any function $f:S\to \R$, and $\Exp_\mu$ refers (here and in the rest of the paper) to the integral with respect to the measure $\mu$. We say that 
$(S, \mu, Q)$ satisfies a {\it Poincar\'e inequality with constant $\alpha$} if for any function $f:\, S\to \R$ 
$$
{\rm Var}_\mu(f):=\frac12 \sum_{x,y\in S} \mu(x)\mu(y)(f(x)-f(y))^2 \leq \alpha\, \Dir_{\mu}(f,f).
$$
Similarly, $(S, \mu, Q)$ satisfies a  {\it Logarithmic Sobolev inequality (LSI) with constant $\alpha$} if for any function $f:\, S\to (0,\infty)$ 
$$
{\rm Ent}_\mu(f^2):= \Exp_\mu\big[f^2(\log f^2- \log \Exp_\mu f^2)\big] \leq \alpha\, \Dir_{\mu}(f,f). 
$$
These functional inequalities allow to bound 
the average global variations of a function (the left hand side of the inequalities) 
by its average local variations, where the notion of ``local'' is dictated by the Markov generator. 
We will refer to the best value of $\alpha$ in the above inequalities as the Poincar\'e and Log-Sobolev constant\footnote{In the literature, it is the inverse of $\alpha$ which is sometimes referred to as the Poincar\'e and Log-Sobolev constant.}, respectively. 
It is a classical fact (see, for example, \cite{book-peres}) that a Poincar\'e inequality provides, in some cases, a control on the relaxation time of a reversible Markov chain as the latter is defined as the inverse of absolute
spectral gap of the chain, while the Poincar\'e constant coincides with the (inverse) spectral gap. Moreover, we have 
\begin{equation}\label{eq: rel-mix}
\big(\trel(Q)-1\big)\log \big(\frac{1}{2\varepsilon}\big)\leq \tmix(Q,\varepsilon)\leq \log \big(\frac{1}{2\varepsilon \mu_{\min}}\big)\, \trel(Q),
\end{equation}  
and 
\begin{equation}\label{eq: mix-logsob}
\tmix(Q,\varepsilon)\leq \frac14 \alpha_{LS}(Q)\big( \log\log \frac{1}{\mu_{\min}}+\log\frac{1}{2\varepsilon^2}\big),
\end{equation}
where $\mu_{\min}=\min_{x\in S} \mu(x)$, $\trel(Q)$ denotes the relaxation time and $\alpha_{LS}(Q)$ the log-Sobolev constant. We refer to \cite{book-peres} for more on the relaxation time and \cite{DS1} for the relation between the logarithmic Sobolev inequality and the mixing time. 

\subsection{Main results}The main result of this paper is a sharp Poincar\'e inequality for the switch chain. 
\begin{theor}[Poincar\'e inequality for the switch chain]\label{th: poincare}
There exist positive universal constants $c, C$ such that the following holds. 
Let $n\in \N$ and $3\leq d\leq n^{c}$. Then $\big(\BipGSet_n(d),\Psimple, Q_u\big)$ satisfies a Poincar\'e inequality with constant $Cnd$. 
In other words, for any $f:\, \BipGSet_n(d)\to \R^+$, we have 
$$
\Var_{\Psimple}(f)\leq Cnd\, \Dir_{\Psimple}(f,f). 
$$
\end{theor}

As we will show in Section~\ref{sec: preliminaries}, the above estimate is sharp. 
The constant $c$ in the above statement can be taken to be $1/1143$; we haven't tried to optimize its value.
We refer to the next subsection for a discussion of the restrictions on $d$ in our argument.
In view of the above, Theorem~\ref{th: poincare} (see also Remark~\ref{rem: lambda-min} below) asserts that
the relaxation time of the switch chain on regular bipartite graphs is of order $nd$. 
Moreover, applying \eqref{eq: rel-mix}, we deduce the following bound on the mixing time. 

\begin{cor}
There exist positive universal constants $c, C$ such that the following holds. Let $n\in \N$ and $3\leq d\leq n^{c}$. 
Then the relaxation and mixing time of the switch chain satisfy 
$$
\trel(Q_u)\leq Cnd\quad \text{and}\quad \tmix(Q_u,\varepsilon)\leq Cnd \big( \log \vert \BipGSet_n(d)\vert+\log(2\varepsilon^{-1})\big). 
$$
\end{cor}
The above bound on the mixing time improves considerably on the bounds stated in \eqref{eq: mixing KTV} and \eqref{eq: mixing Dyer}. 
Moreover, combined with known enumeration estimates for the number of regular bipartite graphs (see for example \cite{W-survey, W-ICM}), it implies that for any $\varepsilon\in (0,1)$  
$$
\tmix(Q_u,\varepsilon)\leq Cnd \big( nd\log (nd)+\log(2\varepsilon^{-1})\big),
$$
for some universal constant $C>0$. Our approach also allows us to derive a log-Sobolev inequality when $d$ is a constant independent of $n$, which yields an improvement on the above bound for the mixing time. 

\begin{theor}[Log-Sobolev inequality for the switch chain]\label{th: log-sob}
Let $d\geq 3$ be a fixed integer. 
Then $\big(\BipGSet_n(d),\Psimple, Q_u\big)$ satisfies a log-Sobolev inequality with constant $C_dn \log n$, where $C_d>0$ may only depend on $d$. In other words, for any $f:\, \BipGSet_n(d)\to \R^+$, we have 
$$
{\rm Ent}_{\Psimple}(f^2)\leq C_d n\log n\, \Dir_{\Psimple}(f,f). 
$$
\end{theor}

We show in Section~\ref{sec: preliminaries} that the above estimate on the log-Sobolev constant is optimal. 
In fact, we implement a comparison procedure between the switch chain and the random transposition model on $\{1,\ldots, nd\}$, which implies
that the Poincar\'e and the log-Sobolev constants coincide (up to constant multiples) for the two models. 
In view of \eqref{eq: mix-logsob}, the above statement implies that for any $\varepsilon\in (0,1)$ 
$$
\tmix(Q_u,\varepsilon)\leq C_d n\log n \big( \log n+\log\frac{1}{2\varepsilon^2}\big),
$$
for some constant $C_d$ depending only on $d$. This estimate matches the numerical upper bound on the mixing time stated in \cite{BHS}\footnote{The simulations in \cite{BHS} covered also more general degree sequences.}. 
However, we believe, as was mentioned in \cite{CDG} (see a remark after Theorem~1 there), 
that the correct mixing time is of order $n\log n$ when $d$ is constant. Thus, Theorem~\ref{th: log-sob} captures, 
up to a logarithmic factor, the predicted mixing time for the switch chain on regular bipartite graphs. 
We expect that the tools developed in this paper can be extended to treat the switch chain on other regular models of graphs 
and considerably improve the mixing time estimates available in the literature \cite{CDG, Greenhill, MES, GS, unified}. 
We plan to pursue this program in the near future. 
Finally, let us note that in this paper, we focused on the Poincar\'e and the log-Sobolev inequalities and have not discussed 
the {\it modified log-Sobolev inequality} which is also known to imply estimates on the mixing time (see \cite{bobkov-tetali}). 
In particular, establishing the optimal modified log-Sobolev inequality would remove the extra logarithmic factor from our mixing time estimate. 
As proving this inequality will require additional effort, we chose to leave it for future work to keep the current paper of a reasonable size. 

We should note that besides the implications of Theorem~\ref{th: poincare} and Theorem~\ref{th: log-sob} 
on the mixing time of the switch chain, those functional inequalities
are interesting on their own right as they also imply concentration inequalities on the space of graphs. 
Indeed, one can derive from Theorem~\ref{th: poincare} an exponential concentration inequality for any Lipschitz function on $\BipGSet_n(d)$, and in particular, 
for the edge count and other graph statistics.  Similarly, Theorem~\ref{th: log-sob} implies a corresponding sub-Gaussian concentration inequality. 
We refer the interested reader to \cite[Chapters~3 and 5]{ledoux} for more information on how these concentration inequalities follow from the Poincar\'e and the log-Sobolev inequality.

\subsection{Strategy of the proof} 
At a high-level, our proof is an implementation of a ``double'' comparison procedure that can be described as follows. 
We consider the switch chain on $\ConfBipGSet_n(d)$, the set of all $d$--regular bipartite 
multigraphs on $[n^{(\ell)}]\sqcup[n^{(r)}]$ equipped with 
the probability measure $\Pconfig$ induced by the configuration model 
(see Section~\ref{sec: preliminaries} for the exact definition).  
As the first (simple) step, we establish the corresponding functional inequalities on $\ConfBipGSet_n(d)$ by comparing the model with 
the so-called {\it random transposition model} on the set of permutations of $\{1,\ldots,nd\}$. This comparison is rather straightforward and is 
carried in Section~\ref{sec: preliminaries}. The second, and main, comparison step is for 
the switch chain on $(\ConfBipGSet_n(d), \Pconfig)$ (whose generator will be denoted by $Q_c$) and the switch chain on $(\BipGSet_n(d), \Psimple)$.
It will require considerable effort and novel ideas. 

The random transposition model is a well studied Markov chain. The mixing time and relaxation time were established for this chain \cite{D,DS2}. 
Moreover, a corresponding log-Sobolev inequality for the random transposition model was also 
derived \cite{DS1,Lee-Yau} (see also \cite{G} for the modified log-Sobolev inequality). 
In view of our comparison procedure, it is not surprising that the Poincar\'e and the log-Sobolev constants we obtain 
in Theorems~\ref{th: poincare} and \ref{th: log-sob} match the corresponding ones from the random transposition model. 

Comparison techniques for Markov chains are a set of tools originally developed by Diaconis and Saloff-Coste 
\cite{DS-Comparison, DS-Comparison2}, which have been extensively used since then to estimate 
the relaxation and mixing times of Markov chains. In its essence, those methods aim at transferring 
knowledge of statistics of a known Markov chain (such as the relaxation time) to another Markov 
chain of interest. The main idea behind the methods of \cite{DS-Comparison, DS-Comparison2} is that 
when the two stationary measures are ``comparable'', it is enough to provide a comparison of the corresponding 
Dirichlet forms of the two chains. The canonical path or flow method then aims precisely at providing such 
a comparison between the Dirichlet forms when the two chains share the same state space. 
We refer to \cite{saloff,RT,Dyer} for an extensive review of those techniques. 

However, those methods have been typically used for the case when the two chains share the same state space 
and the ratio between their stationary measures is a well controlled constant. 
We are only aware of two works \cite{DS-Comparison3,Feder,CDGH} where the assumption on having the same 
state space is relaxed as to having one of the spaces embedded or included in the other. However, in both 
these works, the two stationary measures are in a certain sense comparable. 

In our setting, the probability measures $\Pconfig$ and $\Psimple$ differ significantly as $d$ 
grows with $n$. Indeed, it is known (see \cite[Theorem~6.2]{Janson}) that the probability that 
the configuration model produces a simple graph is asymptotically equivalent to $e^{\frac{-(d-1)^2}{2}}$. 
This discrepancy between $\Pconfig$ and $\Psimple$ makes any comparison procedure based on the 
results of \cite{DS-Comparison, DS-Comparison2} inefficient as it produces an extra factor of 
$e^{\frac{(d-1)^2}{2}}$. Thus, it is only when $d$ is constant (which is the case of Theorem~\ref{th: log-sob}), 
when those techniques could be useful in our setting (see Section~\ref{sec: log-sob}). 
Proving Theorem~\ref{th: poincare} for growing $d$ requires 
us to compare not only the Dirichlet forms, but also variances of functions on the two probability spaces. 

The above discussion leads us to the problem of building, for a given real function $f$ on $(\BipGSet_n(d), \Psimple,Q_u)$, 
an appropriate extension $\tilde f:\, (\ConfBipGSet_n(d),\Pconfig,Q_c)\to \R$ such that the Dirichlet 
forms and variances are in some sense comparable. This would allow us to transfer a Poincar\'e inequality on 
$(\ConfBipGSet_n(d),\Pconfig,Q_c)$ to the one on $(\BipGSet_n(d), \Psimple,Q_u)$.
More specifically, if for any given function $f$ we are able to construct another function $\tilde f$ on $\ConfBipGSet_n(d)$
such that
\begin{equation}\label{eq: 2087604987245}
\Var_{\Psimple}(f)\leq c\,\Var_{\Pconfig}(\tilde f)\quad\mbox{ and }\quad \Dir_{\Pconfig}(\tilde f,\tilde f)\leq C \Dir_{\Psimple}(f,f),
\end{equation}
then we immediately get a Poincar\'e inequality on $\BipGSet_n(d)$ using known results for the random transposition model.

Note that, in a sense, the extension $\tilde f$ must simulteneously have a ``large enough'' variance and a ``small enough''
value of the Dirichlet form.
Since the Dirichlet form measures the local variations of a function, it is natural to define the 
extension in such a way that it varies little locally.
It is a well known fact that the smallest possible value of $\Dir_{\Pconfig}(\tilde f,\tilde f)$ for a function $\tilde f$
on $\ConfBipGSet_n(d)$ which coincides with $f$ on $\BipGSet_n(d)$, is achieved for the {\it harmonic extension}
of $f$, i.e.\ under the assumption that $\tilde f$ is harmonic on $\ConfBipGSet_n(d)\setminus \BipGSet_n(d)$,
with $\BipGSet_n(d)$ viewed as the boundary of the domain for $\tilde f$
(we note here that, in particular, the harmonic extension was used  
by Aldous (see \cite[Theorem~13.20]{book-peres}) to compare the spectral gap of a Markov chain 
with an induced chain).
In probabilistic terms, the harmonic extension $f_H$ of $f$ to the space $\ConfBipGSet_n(d)$ is given by 
$$
f_H(x)= \Exp[f(X_{T_{\BipGSet_n(d)}})\mid X_0=x],
$$
where $(X_t)_{t\geq 0}$ denotes the switch chain on $\ConfBipGSet_n(d)$ and $T_{\BipGSet_n(d)}$ denotes the first time 
the chain hits $\BipGSet_n(d)$. While the harmonic extension minimizes the Dirichlet form over all possible 
extensions, it remains difficult to analyse as it requires a deep understanding of the underlying space in 
order to capture the hitting times essential to its definition.
A much more serious problem is that the harmonic extension of a function in general does not satisfy the leftmost relation
in \eqref{eq: 2087604987245} when the degree $d$ grows with $n$. We would like to give a heuristic argument here
without providing a rigorous proof.

Assume that $d\to\infty$ with $n$ and that $d=o(n)$.
For every subset $\mathcal T$ of $\BipGSet_n(d)$, let $f_{\mathcal T}$ be a function on $\BipGSet_n(d)$
which equals one for $G\in \mathcal T$ and equals zero otherwise. It is 
not difficult to see that for a vast majority of choices of $\mathcal T$ (with respect to the uniform counting measure on
$2^{\BipGSet_n(d)}$), the variance of $f$, $\Var_{\Psimple}(f)=1/4+o(1)$.
Now, let $f_{H,\mathcal T}$ be the harmonic extension of $f_H$.
Then for each graph $G'\in \ConfBipGSet_n(d)$,
$$
f_{H,\mathcal T}(G')=\sum_{G\in \BipGSet_n(d)}w(G,G')\,f_{\mathcal T}(G),
$$
where $w(G,G')$ are non-negative weights not depending on $\mathcal T$ and summing up to one for each fixed $G'$.
The actual values of the weights are not important for us; the only observation we need is that
$\max\limits_{G'}\max\limits_{G}w(G,G')=o(1)$ whenever $G'\notin \BipGSet_n(d)$.
This, together with standard concentration inequalities for linear combinations of independent Bernoulli variables,
implies that for any fixed $G'\in \ConfBipGSet_n(d)\setminus \BipGSet_n(d)$ and for $1-o(1)$ fraction of $2^{\BipGSet_n(d)}$,
$f_{H,\mathcal T}(G')=1/2+o(1)$. A simple double counting argument then gives that
$\Var_{\Pconfig}(f_{H,\mathcal T})=o(1)$ for a typical $\mathcal T$.
Thus, the harmonic extension (in the regime $d\to\infty$) does not satisfy \eqref{eq: 2087604987245}. 

Although the harmonic functions are not suitable for our purposes, they provide
a good illustration of a desired property: the {\it averaging} behaviour of the extension, when
the value at a given multigraph is defined as some average over simple graphs.
Instead of launching a random walk from an element of $G\in\ConfBipGSet_n(d)$ until it reaches $\BipGSet_n(d)$, we exhibit 
a special type of tractable (defined in an explicit and simple manner)
walks which we refer to as {\it the simple paths} (see Definition~\ref{def: simple path}), whereas 
a set of simple graphs reached via the simple paths from $G$
is referred to as the {\it $s$-neighborhood} of $G$ (see Definition~\ref{def: s-neighborhood}).
The value of the function extension at $G$ will then be essentially determined by the average of the original function $f$
over the $s$--neighborhood of $G$. However, this still leaves a problem with the variance:
a function extension defined as such an average will not satisfy the leftmost inequality in \eqref{eq: 2087604987245} in general.
For that, we need to introduce {\it controlled fluctuations} in the definition of the extension, which will be large
enough to get a satisfactory estimate for the variance, but not too large in order not to destroy the required bound
for the Dirichlet form. Those fluctuations are essentially the standard deviation of $f$ restricted to a given $s$--neighborhood.

Let us be a little more specific at this stage while still avoiding technical details which would only overload
the presentation. Suppose that for every multigraph $G\in \ConfBipGSet_n(d)\setminus \BipGSet_n(d)$
we defined its $s$--neighborhood $\SNeigh(G)$ --- a collection of ``nearby'' simple graphs, with a crucial property
that for multigraphs which are at a close distance to each other, their $s$--neighborhoods are also
close in an appropriate sense (Section~\ref{sec: bijection} will make this precise).
We then set
$$
h(G):=\frac{1}{\vert \SNeigh(G)\vert} \sum_{\tilde G\in \SNeigh(G)} f(\tilde G),\quad
w(G):= \frac{1}{|\SNeigh(G)|}
\sum_{\tilde G\in \SNeigh(G)}(f(\tilde G)-h(G))^2,
$$
and define the extension
$$
\tilde f(G)\approx h(G)+\xi_G\,\sqrt{w(G)},
$$
where the weights $\xi_G$ are to be discussed below.
Let us emphasize that the actual definition of $\tilde f$ is more complicated (see Definition~\ref{def: extension}),
hence the ``$\approx$'' sign above.

Without a proper choice of the coefficients $\xi_G$, the above definition is still unsatisfactory.
Indeed, to capture the right bound for the variance, we would need to make sure that $\xi_{G_1}$ and $\xi_{G_2}$
differ significantly when $G_1$ and $G_2$ are far enough from each other (this assertion should hold at least in the average,
for a large fraction of couples\footnote{Here and in the rest of the paper, the term ``couple'' refers to an ordered pair.} $(G_1,G_2)$). On the other hand, for graphs $G_1,G_2$ which are close (for example,
adjacent), the values of $\xi_{G_1}$ and $\xi_{G_2}$ should ideally be close to each other as well 
because otherwise the Dirichlet form of $\tilde f$ may blow up. This produces
complicated restrictions on $\xi_G$.
Our approach consists in defining $\xi_G$ using randomness. In fact, $\{\xi_G\}_{G\in \ConfBipGSet_n(d)}$
shall be a specially constructed {\it centered Gaussian field} on $\ConfBipGSet_n(d)$ whose covariance
structure will guarantee all the required properties with a non-zero probability.
Thus, the function extension $\tilde f$ which we will be working on is randomized: in a sense, we are
dealing with an uncountable collection of functions which, as it turns out, contains a function satisfying
\eqref{eq: 2087604987245}. The Gaussian field and its properties will be discussed in Section~\ref{sec: construction}.

The above description is considerably simplified. In fact, as a first step we
reduce ourselves to the study of a subset of multigraphs which have edges of multiplicities
one and two only, no incident multiedges, and with no more than a prescribed number of multiedges.
With our restrictions on $d$, those multigraphs hold the main weight of $\Pconfig$, and the remaining ``rare'' multigraphs 
can be handled differently with a special trick.
We refer to Definition~\ref{def: category} for a precise definition of the ``standard'' multigraphs in our analysis.
Another complication to the proof is connected with the fact that certain (small) number of simple graphs,
the ones with a very weak expansion property, cannot be reached by our simple path construction,
and as a result either do not belong to any $s$--neighborhood at all or are contained in very few of them.
This destroys our counting argument, and so those graphs require a special treatment.
Only those considerations, together with the Gaussian field construction, allow us to obtain a working
definition of the function extension. We refer to sections Section~\ref{sec: neighborhood}
and Section~\ref{sec: construction} for all further details.

In Section~\ref{sec: variance}, we prove the leftmost inequality \eqref{eq: 2087604987245} (which
holds with a non-zero constant probability with respect to the randomness of $\tilde f$).

To complete the proof, we compare the Dirichlet forms in Section~\ref{sec: Dirichlet}
to obtain the rightmost relation in \eqref{eq: 2087604987245}. To this aim, similarly to the technique 
of \cite{DS-Comparison, DS-Comparison2} using {\it flows},
we create, for every pair of adjacent (or equal) multigraphs, a special collection of paths
on the set of simple graphs $\BipGSet_n(d)$.
The choice of those paths is determined by our construction of the function extension
and, disregarding certain rare cases, is essentially a mapping between the respective $s$--neighborhoods
(or within the same $s$--neighborhood when the multigraphs are equal).
The Dirichlet form associated with $\tilde f$ is then bounded above
by a weighted sum of terms of the form $\sum_P\sum_{t=1}^{\length{P}}(f(P[t])-f(P[t-1]))^2$, where $P$
are paths on $\BipGSet_n(d)$.
The success or failure of this procedure then crucially depends on the structure of the paths $P$
and on values of the weights.
Roughly, we need to make sure that the cumulative weight of any given pair of adjacent simple graphs in the sum
is not too large, which corresponds to {\it controlling the congestion of a flow}.

In Section~\ref{sec: bijection}, we show that for most pairs of adjacent multigraphs, which we call 
{\em perfect pairs} (see Definition~\ref{def: perfect pair}), there is a bijective matching between their $s$--neighborhoods
such that all the matched graphs are at distance one from each other. These perfect pairs turn out to be the main 
contributors to the Dirichlet form and are dealt with in a relatively simple manner.
To work with the remaining pairs of multigraphs, we build
special paths which we call {\it connections}.
We will deal with three types of connections: between simple graphs within
the $s$--neighborhoods of two non-perfect adjacent multigraphs,
between two simple graphs within the same $s$--neighborhood, and
between a simple graph and the $s$--neighborhood of an adjacent multigraph.
As has already been mentioned, the main difficulty is to construct the connections in such a way that
they do not overuse any given edge in $\BipGSet_n(d)$ since otherwise our estimate for the Dirichlet form will blow up.
The actual argument is lengthy and involved; we refer to Section~\ref{sec: connections} for details.


As a concluding remark, we would like to comment on the conditions on $d$ which we obtain in our proof.
The value of the constant $c>0$ in the condition $d\leq n^c$ that we impose in our result on the Poincar\'e constant,
can definitely be improved with more careful computations in Section~\ref{sec: connections}. 
At the same time, there are more fundamental obstacles appearing when considering a relatively large $d$.
Specifically, in the case $d\gg n^{1/3}$ a typical multigraph (drawn according to $\Pconfig$) will contain
a multiedge of multiplicity three or greater.
Furthermore, when $d\gg \sqrt{n}$ the number of multiedges contained in a typical multigraph
becomes comparable to or exceeds the number of vertices.
This destroys our argument which relies heavily on the fact that {\it perfect pairs} of multigraphs
contribute the main weight to the Dirichlet form.

Treatment of the log--Sobolev inequality for constant $d$ is simpler (because of the restriction on $d$) and to a large extent repeats
the approach for the Poincar\'e inequality. We refer directly to Section~\ref{sec: log-sob} for further details.

\medskip 
\noindent
 \textbf{Acknowledgment:} The authors would like to thank Catherine Greenhill for bringing to their attention the paper \cite{Dyer-arxiv}. 
 This project was initiated when the first named author visited the second named author 
 at New York University Abu Dhabi. Both authors would like to thank the institution for excellent working conditions. 
 The first named author was partially supported by the Sloan Fellowship.

\section{Preliminaries}\label{sec: preliminaries}
Throughout the paper, we will make use of several parameters which we list below
\begin{equation}\label{eq: parameters}
r_0=40,\quad \zconst=\lfloor n/d\log^2 n\rfloor,\quad \gconst=\lceil n/d\rceil,\quad \text{and}\quad \cconst=d^2\lfloor \log n\rfloor. 
\end{equation}
We suppose that $n$ is large enough and that 
\begin{equation}\label{eq: assumption-d}
3\leq d\leq n^{1/1143}. 
\end{equation}

\chng{Throughout the paper, we will use ``$\subset$'' to denote set inclusion which is not necessarily strict.}

Recall that we denote by $\ConfBipGSet_n(d)$ the set of all $d$--regular bipartite multigraphs on $[n^{(\ell)}]\sqcup[n^{(r)}]$,
and by $\Pconfig$ the probability measure on $\ConfBipGSet_n(d)$ induced by the configuration model.  
For any graph $G'\in \ConfBipGSet_n(d)$, we denote by $\Adj(G')$ its adjacency matrix, with
$$
\Adj(G')_{ij}=\mult_{G'}(i,j),\quad (i,j)\in[n]\times[n],
$$
where $\mult_{G'}(i,j)$ is the {\it multiplicity} of the edge $(i,j)$ in $G'$ (whenever $G'$ does not contain an edge
$(i,j)$, we will assume that its multiplicity in $G'$ is zero). Edges whose multiplicity is greater than one,
will be called {\it multiedges} in this paper.

A simple switching operation can be uniquely identified by a 
quadruple $\langle i_1,i_2,j_1,j_2\rangle$, which determines the two edges $(i_1,j_1)$ and $(i_2,j_2)$ to be switched.
More formally, a simple switching operation 
$\phi=\langle i_1,i_2,j_1,j_2\rangle$, $i_1,i_2\in [n^{\ell}]$, $j_1,j_2\in [n^{(r)}]$, on the set of multigraphs
$\ConfBipGSet_n(d)$ is defined as follows: 
for any graph $G'$
with $\mult_{G'}(i_1,j_1)>0$ and $\mult_{G'}(i_2,j_2)>0$, the graph $\phi(G')$ is uniquely identified by the conditions
\begin{itemize}
\item $\mult_{\phi(G')}(i_1,j_1)=\mult_{G'}(i_1,j_1)-1$;
\item $\mult_{\phi(G')}(i_2,j_2)=\mult_{G'}(i_2,j_2)-1$;
\item $\mult_{\phi(G')}(i_1,j_2)=\mult_{G'}(i_1,j_2)+1$;
\item $\mult_{\phi(G')}(i_2,j_1)=\mult_{G'}(i_2,j_1)+1$;
\item The multiplicities of all edges outside of the set $\{(i_1,j_1),(i_1,j_2),(i_2,j_1),(i_2,j_2)\}$ are the same for $G'$
and $\phi(G')$.
\end{itemize}
The domain of $\phi$ is the set of all multigraphs $G'$ with $\mult_{G'}(i_1,j_1)>0$ and $\mult_{G'}(i_2,j_2)>0$.
In what follows, we will often use the ``quadruple'' notation $\langle i_1,i_2,j_1,j_2\rangle$ for the switching.

\medskip

The restriction of $\phi$ to the domain $\big\{G'\in \BipGSet_n(d):\;\mult_{G'}(i_k,j_k)=1,\;\mult_{G'}(i_k,j_{3-k})=0,\;k=1,2\big\}$
will be called a simple switching operation on the set of simple graphs $\BipGSet_n(d)$.

\subsection{The switch chain on the configuration model}
Observe that $(\ConfBipGSet_n(d), \Pconfig)$ is generated from $(\Permset, \Pperm)$, the set of permutations of $nd$ elements  
equipped with the uniform measure $\Pperm$, once one takes into account the invariance of the graph under
permutation within multiedges and permutations within the ``buckets'' of the configuration model.
More precisely, for any $G\in \ConfBipGSet_n(d)$, we have 
$$
\Pconfig (G)= \frac{(d!)^{2n}}{(nd)! \prod_{1\leq i,j\leq n} \mult_{G}(i,j)!}.
$$
The following estimate, which holds for large enough $n$, is taken from \cite[Theorem~6.2]{Janson}
and will be often used:
\begin{equation}\label{eq: prob-simple}
\frac12 e^{-\frac{(d-1)^2}{2}}\leq \Pconfig\big(\BipGSet_n(d)\big) \leq 2e^{-\frac{(d-1)^2}{2}}. 
\end{equation}
Note that the above relation can be equivalently rewritten as
\begin{equation}\label{eq: size-simple}
\frac12 e^{-\frac{(d-1)^2}{2}} \frac{(nd)!}{(d!)^{2n}}\leq \vert \BipGSet_n(d)\vert \leq 2e^{-\frac{(d-1)^2}{2}}\frac{(nd)!}{(d!)^{2n}}. 
\end{equation}

We will always assume a graph structure on $\ConfBipGSet_n(d)$,
with two elements (graphs) of $\ConfBipGSet_n(d)$ connected by an edge whenever there is a simple switching
operation on edges that are not incident (regardless whether it creates or destroys
multiedges) transforming one graph to the other. Given a graph $G$ in $\ConfBipGSet_n(d)$,
we denote by $\neigh(G)$ the set its neighbors $\ConfBipGSet_n(d)$.

We obtain {\it the switch chain} on $\ConfBipGSet_n(d)$ via 
its generator $Q_c$ defined for any $G_1, G_2\in \ConfBipGSet_n(d)$ by
$$
Q_c(G_1,G_2):=\begin{cases}
 \frac{\mult_{G_1}(i,j)\,\mult_{G_1}(i',j') }{nd(nd-1)/2},&\mbox{if $G_2\in \neigh(G_1)$ is obtained from $G_1$}\\
 &\mbox{ by the switching $\langle i,i',j,j'\rangle$};\\
-\sum_{G'\in \neigh(G_1)} Q_c(G_1,G'),&\mbox{if $G_1=G_2$};\\
&\\
0,&\mbox{otherwise.}
\end{cases}
$$
In the next lemma, we verify that the chain is reversible and aperiodic. 
\begin{lemma}
The Markov generator $Q_c$ defined above is reversible with respect to $\Pconfig$ and aperiodic provided $d\geq 2$.
\end{lemma}
\begin{proof}
Let $G_1, G_2\in \ConfBipGSet_n(d)$ be two adjacent multigraphs and let $\Phi= \langle i,i',j,j'\rangle$ be the switching operation such that 
$G_2=\Phi(G_1)$. Note that the multiplicities $\mult_{G_1}(r,s)=\mult_{G_2}(r,s)$ for any $(r,s)\not\in\{(i,j), (i,j'), (i',j), (i',j')\}$;
$$
\mult_{G_1}(i,j)= \mult_{G_2}(i,j)+1,\quad \mult_{G_1}(i',j')= \mult_{G_2}(i',j')+1,
$$
and 
$$
\mult_{G_1}(i,j')= \mult_{G_2}(i,j')-1,\quad \mult_{G_1}(i',j)= \mult_{G_2}(i',j)-1. 
$$
Using this, it is easy to check that $\Pconfig(G_1) Q_c(G_1,G_2)= \Pconfig(G_2) Q_c(G_2,G_1)$ and deduce the reversibility. 

To prove that $Q_c$ is aperiodic, it is enough to show that for any $G\in  \ConfBipGSet_n(d)$, we have 
$$
\sum_{G'\in \ConfBipGSet_n(d)\setminus \{G\}} Q_c(G,G')<1. 
$$
We have
$$
\sum_{G'\in \ConfBipGSet_n(d)\setminus \{G\}} Q_c(G,G')
=\frac12 \sum_{\underset{\text{not incident}}{e\neq e'\in E_G}}\frac{\mult_{G}(e)\,\mult_{G}(e') }{nd(nd-1)/2},
$$
where we denoted by $E_G$ the set of (multi)edges of $G$, with multiplicities not counted
(i.e.\ an edge enters the set $E_G$ once even if its multiplicity is greater than one).
Given $e\in E_G$, we have by the regularity of $G$ that 
$$
\sum_{\underset{e' \text{ incident or equal to }e}{e'\in E_G}} \mult_G(e')\geq d.
$$
Therefore, using that $\sum_{e\in E_G} \mult_G(e)=nd$, we deduce
$$
\sum_{G'\in \ConfBipGSet_n(d)\setminus \{G\}} Q_c(G,G')
\leq\frac{nd(nd-d)}{nd(nd-1)}= \frac{nd-d}{nd-1}<1,
$$
whenever $d\geq 2$. 
\end{proof}

\begin{rem}\label{rem: lambda-min}
Note that the calculation above shows that for every $G\in \BipGSet_n(d)$, the probability that there is a self-loop at $G$ (for the switch chain on both $\ConfBipGSet_n(d)$ and $\BipGSet_n(d)$) is at least $1-\frac{nd-d}{nd-1}>\frac{1}{2n}$. It follows that the smallest eigenvalue of the switch chain is bounded below by $-1+\frac{1}{n}$ (see for instance \cite[Page 702]{DS-Comparison2}). In view of this and the lower bound on the Poincar\'e constant (see Proposition~\ref{prop: lower bound poincare}), the relaxation time of the switch chain is equivalent (up to a constant factor) to the Poincar\'e constant. 
\end{rem}

The transposition chain on the set of permutations $(\Permset, \pi)$ has been widely studied. Its Markov generator $Q$ is defined for any $\sigma,\sigma'\in \Permset$ by
$$
Q(\sigma,\sigma'):=\begin{cases}
 \frac{2}{(nd)^2},&\mbox{if $\sigma$ and $\sigma'$ differ by a transposition}\\
-\frac{nd-1}{nd} ,&\mbox{if $\sigma=\sigma'$};\\
0,&\mbox{otherwise.}
\end{cases}
$$
It is known \cite{DS2, DS1,Lee-Yau} that the Poincar\'e constant of the random transposition model is of order $nd$ while the log-Sobolev constant is of order $nd\log (nd)$. 
The next proposition encapsulates the first comparison step transferring the knowledge of the Poincar\'e and log-Sobolev constant 
of the random transposition model  to the switch chain on $\ConfBipGSet_{n}(d)$.  

\begin{prop}[Poincar\'e and log-Sobolev constant for the switch chain on the configuration model]\label{prop: poincare-logsob-config}
The Poincar\'e constant of $(\ConfBipGSet_n(d), \Pconfig, Q_c)$ is of order $nd$ while the log-Sobolev constant is of order $nd\log (nd)$. 
\end{prop}

\begin{proof}
Let $\psi: \Permset\to  \ConfBipGSet_n(d)$ be the many-to-one mapping of the configuration model. 
Let $f: \ConfBipGSet_n(d)\to \R_+$. We consider the function $\tilde f=f\circ \psi$ on $\Permset$. We will verify that 
$$
\Var_{\Pperm}(\tilde f)= \Var_{\Pconfig}(f),\quad \Ent_{\Pperm}(\tilde f)= \Ent_{\Pconfig}(f),\quad
 \text{and } \Dir_{\Pperm}(\tilde f,\tilde f)=  \frac{nd-1}{nd}\Dir_{\Pconfig}(f,f).
$$
This will allow to deduce the statement of the proposition using the Poincar\'e and log-Sobolev constants of the transposition chain on $\Permset$. 
First note that given $G\in \ConfBipGSet_n(d)$, the set $\psi^{-1}(\{G\})$ has cardinality  $\frac{(d!)^{2n}}{\prod_{1\leq i,j\leq n} \mult_{G}(i,j)!}$. With this estimate in hand, we can write 
\begin{align*}
\Exp_{\Pperm}[\tilde f] &= \frac{1}{(nd)!}\sum_{\sigma \in \Permset} \tilde f(\sigma)
= \frac{1}{(nd)!} \sum_{G\in \ConfBipGSet_n(d)} \sum_{\sigma\in \psi^{-1}(\{G\})} \tilde f(\sigma) \\
&= \frac{1}{(nd)!} \sum_{G\in \ConfBipGSet_n(d)} f(G)\cdot \vert \psi^{-1}(\{G\})\vert= \Exp_{\Pconfig} [f]. 
\end{align*}
Similarly, we have $\Exp_{\Pperm}[{\tilde f}^2] = \Exp_{\Pconfig} [f^2]$ and $\Exp_{\Pperm}[{\tilde f}\log \tilde f] = \Exp_{\Pconfig} [f\log f]$ which proves that 
$\Var_{\Pperm}(\tilde f)= \Var_{\Pconfig}(f)$ and  $\Ent_{\Pperm}(\tilde f)= \Ent_{\Pconfig}(f)$. 

Finally, to compare the Dirichlet forms, we write 
$$
\Dir_{\Pperm}(\tilde f,\tilde f)= \frac{1}{(nd)!\, (nd)^2}  \sum_{\sigma \sim \sigma'\in \Permset}  \big(\tilde f(\sigma)- \tilde f(\sigma')\big)^2. 
$$
For every pair of adjacent permutations $\sigma, \sigma' \in \Permset$, we either have $\psi(\sigma)=\psi(\sigma')$ or $\psi(\sigma)$ and $\psi(\sigma')$ 
are adjacent with respect to the switch graph. Moreover, if $\psi(\sigma)=\psi(\sigma')$, we have $\tilde f(\sigma)=\tilde f(\sigma')$. 
Thus, we can write 
\begin{align*}
\Dir_{\Pperm}(\tilde f,\tilde f)&= \frac{1}{(nd)!\, (nd)^2}  \sum_{G\sim G' \in  \ConfBipGSet_n(d)} \sum_{\underset{(\sigma,\sigma') \in \psi^{-1}(\{G\})\times\psi^{-1}(\{G'\})}{\sigma\sim \sigma'}}  \big(\tilde f(\sigma)- \tilde f(\sigma')\big)^2\\
&=  \frac{1}{(nd)!\, (nd)^2}  \sum_{G\sim G' \in  \ConfBipGSet_n(d)} \big( f(G)-  f(G')\big)^2 \cdot\vert \Gamma(G,G') \vert,
\end{align*}
where we denoted $\Gamma(G,G'):=\{(\sigma,\sigma') \in \psi^{-1}(\{G\})\times\psi^{-1}(\{G'\}):\, \sigma\sim \sigma'\}$. 
Fix $G\sim G'$ and let $\langle i,i',j,j'\rangle$ be the switching operation used to transform $G$ into $G'$. Now given $\sigma \in \psi^{-1}(\{G\})$, it is not difficult to see that there are $\mult_{G}(i,j)\,\mult_{G}(i',j')$ choices for $\sigma'\in \psi^{-1}(\{G'\})$ such that $\sigma\sim \sigma'$. Therefore, we deduce that 
$$
\vert \Gamma(G,G')\vert = \mult_{G}(i,j)\cdot\mult_{G}(i',j')\cdot \vert  \psi^{-1}(\{G\})\vert= \frac{(d!)^{2n}\mult_{G}(i,j)\cdot\mult_{G}(i',j')}{\prod_{1\leq i,j\leq n} \mult_{G}(i,j)!}. 
$$
Replacing this identity above, we get that $\Dir_{\Pperm}(\tilde f,\tilde f)=\frac{nd-1}{nd}\Dir_{\Pconfig}(f,f)$ and finish the proof. 
\end{proof}

\subsection{Lower bounds on the Poincar\'e and log-Sobolev constants}

We verify in this subsection that the estimates in Theorems~\ref{th: poincare} and \ref{th: log-sob} are optimal. 
To this aim, we will provide matching lower bounds for the Poincar\'e and log-Sobolev constants of the switch chain on $\BipGSet_n(d)$. 

\begin{prop}[Lower bounds for Poincar\'e and log-Sobolev constants]\label{prop: lower bound poincare}
Let $2\leq d\leq \frac{n}{2}$. 
The Poincar\'e and log-Sobolev constants of $(\BipGSet_n(d), \Psimple, Q_u)$ are at least $nd/4$ and $\frac12
nd\log(\frac{n}{d})$, respectively. 
\end{prop}
\begin{proof}
Denote by $\alpha$ and $\beta$ the Poincar\'e and the log-Sobolev constants, respectively. 
By definition, we have 
$$
\alpha=\sup \frac{\Var_{\Psimple}(f)}{\Dir_{\Psimple}(f,f)}\quad \text{ and } \quad \beta=\sup \frac{\Ent_{\Psimple}(f^2)}{\Dir_{\Psimple}(f,f)},
$$
where the supremum is taken over all functions $f:\, \BipGSet_n(d)\to \R$. 
To obtain the required lower bounds, we shall use a test function.  
Define $f:\, \BipGSet_n(d)\to \R$ by 
$$
f(G)=\begin{cases}
1& \text{ if $(1,1)$ is an edge in $G$};\\
0& \text{otherwise}. 
\end{cases}
$$
Note that any $G\in  \BipGSet_n(d)$ with $f(G)=1$ has at most $(nd-1)$ adjacent graphs $G'$ satisfying $f(G')=0$. 
Indeed, any such $G'$ is obtained from $G$ by a switching involving the edge $(1,1)$, and it remains to choose the second edge participating in the switching. 
It follows from $d$-regularity that 
$$
\vert\{G\in \BipGSet_n(d):\, f(G)=1\}\vert =\frac{d}{n}\vert \BipGSet_n(d)\vert.
$$
Using this, we now calculate 
$$
\Dir_{\Psimple}(f,f)=  \sum_{\underset{f(G)=1}{G\in\BipGSet_n(d)}} \sum_{\underset{G'\sim G,\, f(G')=0}{G'\in  \BipGSet_n(d)}} \Psimple(G)Q_u(G,G')
\leq \frac{2}{n^2}.
$$
On the other hand, we have $\Exp_{\Psimple} f= \frac{d}{n}$. Therefore we have 
$$
\Var_{\Psimple}(f)= \frac{d}{n}\Big(1-\frac{d}{n}\Big)\quad \text{and}\quad \Ent_{\Psimple}(f^2)= -\frac{d}{n}\log\Big(\frac{d}{n}\Big). 
$$
Putting these estimates together, we deduce that 
$$
\alpha\geq \frac12 nd \Big(1-\frac{d}{n}\Big)\quad \text{ and }\quad \beta\geq \frac12 nd \log\Big(\frac{n}{d}\Big). 
$$
The result follows from the assumption on $d$. 
\end{proof}

\section{Simple neighborhoods of multigraphs}\label{sec: neighborhood}

The goal of this section is to associate to a given multigraph $G\in \ConfBipGSet_n(d)$ a collection of simple graphs connected to it. This will allow us 
to naturally extend a function on simple graphs to one on $\ConfBipGSet_n(d)$ by making use of such collections.

Given a path $P$ of length $\length{P}$ on $\ConfBipGSet_n(d)$ or on $\BipGSet_n(d)$ starting at a graph $G'$,
we say that the sequence of switchings $(\phi_1,\dots,\phi_{\length{P}})$ {\it generates $P$} if 
for every $0\leq t<\length{P}$, $\phi_{t+1}$ is the simple switching operation transforming $P[t]$ to $P[t+1]$.
Note that the ordering may be important in general.

\begin{defi}[Category $k$ multigraphs]\label{def: category}
Given an integer $0\leq k\leq nd/2$, we define $\categ(k)$ as the set of multigraphs $G\in \ConfBipGSet_n(d)$ which satisfy the following: 
\begin{itemize}
\item $G$ has exactly $k$ multiedges of multiplicity $2$; 
\item None of those multiple edges are incident to one another; 
\item $G$ has no edges of multiplicity three or greater. 
\end{itemize}
\end{defi}
In the terminology of matrices, the above conditions mean that the adjacency matrix $\Adj(G)$ of $G$ has all its entries smaller or equal to $2$, among which exactly $k$ entries equal to $2$, and no row or column contains more than one entry equal to $2$. When $G\in \categ(k)$, we will refer to $k$ as {\it the category number of $G$}. Note that with this definition, $\BipGSet_n(d)$ consists of graphs which are of category $0$. 
Given two integers $a,b$, we also denote $\categ([a,b]):=\cup_{k=a}^{b} \categ(k)$. 

\medskip

Let $\cconst$ be as in \eqref{eq: parameters} and define 
$$
\Ugly(\cconst)=\ConfBipGSet_n(d)\setminus \categ([0,\cconst]).
$$

Let us record the following useful observations. 
\begin{lemma}\label{lem: nice-multigraph-prob and transition}
Given an integer $0\leq k\leq nd/2$ and $G\in \categ(k)$, then 
$$
\Pconfig(G)= \frac{(d!)^{2n}}{(nd)! 2^{k}}= \frac{\Pconfig\big(\BipGSet_n(d)\big)}{2^k\vert \BipGSet_n(d)\vert},
$$
and 
$$
\vert \{G'\in \Ugly(\cconst):\, G'\sim G\}\vert\leq 
\begin{cases}
4kd(d-2)^2&\mbox{ if $k\leq \cconst -2$}; \\
4kd(d-2)^2+(nd-2k)(d-1)^2& \mbox{ otherwise}.
\end{cases}
$$
\end{lemma}
\begin{proof}
The first part of the lemma follows from the definition of $\Pconfig$ and $\categ(k)$.
Now, we prove the second assertion. Let $G'\in \Ugly(\cconst)$ be adjacent to $G$, and let $\langle i,i',j,j'\rangle$ be the switching used to pass from $G$ to $G'$. Without loss of generality, we can assume that $\mult_G(i',j)\geq \mult_G(i,j')$. Then the switching necessarily satisfies one of the following:
\begin{itemize}
\item $\mult_G(i',j)=2$ (so that a multiplicity $3$ edge is created in $G'$);
\item $i$ is incident to a multiple edge and $\mult_G(i,j)=\mult_G(i,j')=1$ OR
$i'$ is incident to a multiple edge and $\mult_G(i',j)=\mult_G(i',j')=1$ (two multiedges in the same row);
\item $j$ is incident to a multiple edge and $\mult_G(i,j)=\mult_G(i',j)=1$ OR
$j'$ is incident to a multiple edge and $\mult_G(i,j')=\mult_G(i',j')=1$ (two multiedges in the same column);
\item $\mult_G(i,j)=\mult_G(i',j')=\mult_G(i,j')=1$ when $G\in \categ([\cconst-1,\cconst])$ (two multiedges are created). 
\end{itemize}
Since there are $k$ multiplicity $2$ edges in $G$, then using regularity of $G$, the number of switchings
$\langle i,i',j,j'\rangle$
satisfying the first case is at most $k(d-2)^2$.

To count the number of switching satisfying the second case, note that since $G\in \categ(k)$, then there are at exactly $k$ left vertices $i$ incident to a multiple edge. Therefore, there are at most $k(d-2)^2$ choices of indices $i,j,j'$ satisfying $\mult_G(i,j)=\mult_G(i,j')=1$, and
there are at most $d-1$ choices for the index $i'$ to complete construction of the switching. Thus, the total number can be bounded above by twice $k(d-2)^2(d-1)$.

The third case is completely identical to the second one.

Finally, to bound the number of switchings in the fourth case, note that there are at most $(nd-2k)$ choices for the edge $(i,j)$, and at most $(d-1)^2$ choices for the indices $i', j'$ to satisfy $\mult_G(i',j')=\mult_G(i,j')=1$. Putting together the above estimates, we finish the proof. 
\end{proof}

We now define the simple paths which lead from a multigraph $G'\in\categ([1,\cconst])$ to a simple graph. 
\begin{defi}[Simple paths]\label{def: simple path}
Let $1\leq k\leq \cconst$, $G'\in \categ(k)$ and let $\{(i_s,j_s)\}_{1\leq s\leq k}$ be the multiplicity $2$ edges of $G'$ ordered in increasing order of $(i_s)_{1\leq s\leq k}$.  
We will say that a path $P$ starting at $G'$ is {\it simple} if for every $0\leq t<k$, the simple switching used to obtain $P[t+1]$ from $P[t]$ is given by $\langle i_{t+1},i_{t+1}',j_{t+1},j_{t+1}'\rangle$ where $i_{t+1}'\in [n^{(\ell)}], j_{t+1}'\in [n^{(r)}]$ satisfy the following conditions:
\begin{itemize}
\item For every $0\leq t<k$, we have 
$$
i_{t+1}'\not\in \{i_s\}_{1\leq s\leq k}, \quad j_{t+1}'\not\in \{j_s\}_{1\leq s\leq k},
$$
and all $(i_s')_{1\leq s\leq k}$ (resp. $(j_s')_{1\leq s\leq k}$) are pairwise distinct. 
\item For every $0\leq t<k$, $\mult_{G'}(i_{t+1}',j_{t+1})= \mult_{G'}(i_{t+1},j_{t+1}')=0$. 
\end{itemize}
\end{defi}

It is clear from the above definition that the length of any simple path starting at $G'\in \categ(k)$, if it exists, is equal to $k$ and its
endpoint belongs to $\BipGSet_n(d)$. 
As we will see below, simple paths always exist provided $\cconst$ is bounded appropriately. 

Next, we define the $s$--neighborhood of a multigraph. 

\begin{defi}[$s$--neighborhood]\label{def: s-neighborhood}
Let $G'\in \categ([1,\cconst])$.
The set of all endpoints of simple paths starting at $G'$ will be denoted by $\SNeigh(G')$ and called {\it the $s$--neighborhood}
of the graph. 
\end{defi}

The next lemma asserts that every simple path 
is uniquely determined by its starting point and endpoint. 

\begin{lemma}[Uniqueness of a simple path]\label{l: 2958720598750}
Let $1\leq k\leq \cconst$ and $G'\in \categ(k)$ with a non-empty $s$--neighborhood. 
Then any graph $G\in \SNeigh(G')$ uniquely determines the collection of $k$ simple switching operations generating
the simple path leading from $G'$ to $G$.  
\end{lemma}
\begin{proof}
Let $1\leq k\leq \cconst$,  $G'\in \categ(k)$ and $G\in \SNeigh(G')$. 
Let $\{(i_s,j_s)\}_{1\leq s\leq k}$ be the multiplicity $2$ edges of $G'$ ordered so that the sequence
$(i_s)_{1\leq s\leq k}$ is increasing. Assume further that
$(\langle i_{t+1},i_{t+1}',j_{t+1},j_{t+1}'\rangle)_{t<k}$ and
$(\langle i_{t+1},i_{t+1}'',j_{t+1},j_{t+1}''\rangle)_{t<k}$ are two collections of simple switchings generating simple paths leading from $G'$ to $G$.
We need to show that necessarily $i_{t+1}'=i_{t+1}''$ and $j_{t+1}'=j_{t+1}''$ for all $t<k$.

Fix any $t<k$.
Note that, according to the definition of a simple path, the {\it only} switching from the first collection which operates on $i_{t+1}$--st left vertex is
$\langle i_{t+1},i_{t+1}',j_{t+1},j_{t+1}'\rangle$,
and analogous assertion is true for the second collection.
Further, $\mult_{G'}(i_{t+1},j_{t+1}')= \mult_{G'}(i_{t+1},j_{t+1}'')=0$ while $\mult_G(i_{t+1},j_{t+1}')= \mult_G(i_{t+1},j_{t+1}'')=1$.
Hence, we must have $j_{t+1}'=j_{t+1}''$.

Similarly, the {\it only} switching from the first collection which operates on $j_{t+1}$--st right vertex is
$$\langle i_{t+1},i_{t+1}',j_{t+1},j_{t+1}'\rangle,$$
and the same is true for the second collection.
Applying the above argument, we obtain $i_{t+1}'=i_{t+1}''$, and the result follows.
%
\end{proof}

We now estimate the cardinality of every $s$--neighborhood. 

\begin{lemma}[Cardinality of $s$--neighborhoods]\label{l: 2-9857-2985}
Let $1\leq k\leq \cconst$ and $G\in \categ(k)$. Then necessarily
$$
\big((n-k-2d)d\big)^k\leq |\SNeigh(G)|\leq  \big((n-k)d\big)^k. 
$$
Moreover, 
 $$
 |\SNeigh(G)|\in \Big[\frac{(nd)^k}{2}, (nd)^k\Big]. 
 $$
\end{lemma}
\begin{proof}
Let $1\leq k\leq \cconst$, $G\in \categ(k)$ and let $\{(i_s,j_s)\}_{1\leq s\leq k}$ be the ordered multiedges of $G$. 
It follows from the construction of simple paths that if $P$ and $P'$ are two simple paths starting at $G$,
then $P[t]\neq P'[t]$ implies that $P[t+1]\neq P'[t+1]$ for any $0\leq t< k$.
In view of this, our task is reduced to estimating the number of choices of $P[t+1]$ conditioned on a choice of
$P[t]$, for every $0\leq t<k$.
We only need to estimate the number of choices of the edge
$(i_{t+1}',j_{t+1}')$ on which the $t$--th switching will be operated.
Since $G\in \categ(k)$, then clearly there are at most $(nd-kd)$ choices for such an edge. The upper bound follows.

Further, consider the lower bound.
Note that there are at most $(d-1)(d-2)$ choices of $(i_{t+1}',j_{t+1}')$ such that both $(i_{t+1}',j_{t+1}')$ and $(i_{t+1},j_{t+1}')$ are edges in $P[t]$. Similarly, there are at most $(d-1)(d-2)$ choices of $(i_{t+1}',j_{t+1}')$  such that both $(i_{t+1}',j_{t+1}')$ and $(i_{t+1}',j_{t+1})$ are edges in $P[t]$. Thus, we deduce that there are at least $(nd-k d)- 2(d-1)(d-2)$ choices for the edge $(i_{t+1}',j_{t+1}')$ on which the $t$--th switching can be operated. This proves the first part of the lemma. 
The second one follows by the choice of $\cconst$ in \eqref{eq: parameters}, the assumption on $d$ in \eqref{eq: assumption-d}, and the fact that $n$ is large enough. 
\end{proof}

Note that the previous lemma implies that simple paths always exist under our assumptions on $\cconst$ and $d$.  
It will also be important to know how many $s$--neighborhoods contain a given simple graph. Unfortunately, it turns out that some simple 
graphs belong to few, if any, $s$--neighborhoods. However, this only concerns a small proportion of graphs which have bad expansion properties. 
To measure ``non-expansion'' degree of $G$ in a convenient way for us, we introduce the following quantity. 

\begin{defi}[A measure of anti-expansion]\label{def: measure of ae}
For any graph $G\in \BipGSet_n(d)$, define
\begin{align*}
\antiexp(G):=\big|\big\{
(i,j)\in[n]\times[n]:\;&\exists\; k\in[n]\setminus\{i\},\; \mbox{$(i,j)$ and $(k,j)$ are edges of $G$},\\
&|\supp\row_i(\Adj(G))\cap \supp\row_k(\Adj(G))|\geq 2
\big\}\big|.
\end{align*}
Let $\zconst$ be as in \eqref{eq: parameters} and define 
$$
\Expset(\zconst):=\{G\in \BipGSet_n(d):\, \antiexp(G)\leq \zconst\}. 
$$
\end{defi}

As it is well known, most simple regular graphs are very good expanders. In the next lemma, we verify that this also applies to 
our measure of expansion, and collect some other useful properties. 

\begin{lemma}[Properties of the anti-expansion measure]\label{lem: size-bad graphs}
The following assertions hold.
\begin{itemize}
\item[i.] For any two adjacent simple graphs $G$ and $G'$ we have
$$
\antiexp(G)\leq \antiexp(G')+2d^2.
$$
\item[ii.] Let $G\in  \BipGSet_n(d)$. Then 
$$
\vert\{G'\in  \BipGSet_n(d):\, G'\sim G\mbox{ and }\antiexp(G')< \antiexp(G)\}\vert \geq \frac12 \antiexp(G)(nd-2d^4).
$$
\item[iii.] Let $G\in  \BipGSet_n(d)$. Then
$$
\vert\{G'\in  \BipGSet_n(d):\, G'\sim G\mbox{ and }\antiexp(G')> \antiexp(G)\}\vert \leq nd^5.
$$
\item[iv.] We have $\Psimple\big(\Expset(\zconst)\big)\geq 1-\frac{4nd^7}{(\zconst +2-2d^2)(nd-2d^4)}$.
\end{itemize}
\end{lemma}
\begin{proof}\hspace{0cm}\\
\begin{itemize}
\item[i.] Let $G, G'\in \BipGSet_n(d)$, $G\sim G'$, and let $\langle i,i',j,j'\rangle$ be the switching transforming $G$ to $G'$. 
Given $\tilde i\in [n^{(\ell)}]\setminus \{i,i'\}$, note that if $(\tilde i, \tilde j)$ is an edge of $G$ such that there exists $k\in [n^{(\ell)}]\setminus \{i,i'\}$ with $(k,\tilde j)$ an edge of $G$ and 
$$
|\supp\row_{\tilde i}(\Adj(G))\cap \supp\row_k(\Adj(G))|\geq 2,
$$
then necessarily 
$$
|\supp\row_{\tilde i}(\Adj(G'))\cap \supp\row_k(\Adj(G'))|\geq 2,
$$
since $\Adj(G)$ and $\Adj(G')$ coincide on these rows. This automatically implies that 
\begin{align*}
\antiexp(G)&\leq \antiexp(G')+ \vert\{(\tilde i, \tilde j)\in E_G:\, \tilde i\in \{i,i'\}\}\vert\\
&\quad + \vert\{(\tilde i, \tilde j)\in E_G:\, \exists k \in \{i,i'\}, (k,\tilde j)\in E_G\}\vert, 
\end{align*}
where we denoted $E_G$ the set of edges of $G$. 
Note that there are $2d$ edges incident to $\{i,i'\}$, and at most $(d-1)$ edges which share the same right vertex with a given edge. Thus, we get 
$$
\antiexp(G)\leq \antiexp(G')+ 2d+ 2d(d-1),
$$ 
and finish the proof. 
\item[ii.] Let $(i,j)$ be an edge of $G$ such that for some $k\neq i$ we have $(k,j)$ is an edge of $G$ and 
$$
\big|\supp\row_i(\Adj(G))\cap \supp\row_k(\Adj(G))\big|\geq 2.
$$
Clearly, the number of such pairs $(i,j)$ is $\antiexp(G)$. 
We now pick an edge $(i',j')$ of $G$ satisfying the following: 
\begin{itemize}
\item For any $\tilde i\in[n^{(\ell)}]$ such that $(\tilde i, j)$ is an edge of $G$, we have 
$$
\supp\row_{i'}(\Adj(G))\cap \supp\row_{\tilde i}(\Adj(G))=\emptyset.
$$
\item For any $\tilde i\in[n^{(\ell)}]$ such that $(\tilde i, j')$ is an edge of $G$, we have 
$$
\supp\row_{i}(\Adj(G))\cap \supp\row_{\tilde i}(\Adj(G))=\emptyset.
$$
\end{itemize}
Note that with these assumptions, the switching $\langle i,i',j,j'\rangle$ can be performed and the resulting simple graph $G'$
satisfies $\antiexp(G')\leq \antiexp(G)-1$. 
It remains to count the number of choices of the edge $(i',j')$ as above. Note that there are at most $d(1+(d-1)^2)$ choices of $i'$ which violate the first assumption, and thus at most $d^4$ choices of edges $(i',j')$ violating the first assumption. Similarly, note that there are at most $d(d-1)+1$ indices $\tilde i$ such that the supports of $\row_{\tilde i}(\Adj(G))$ and $\row_{i}(\Adj(G))$ intersect, and thus at most $d^4$ choices of edges $(i',j')$ violating the second assumption. Putting together these estimates, we get the result. 
\item[iii.] 
Note that if $G'\sim G$ satisfies $\antiexp(G')>\antiexp(G)$, then necessarily there exists  an edge $(i,j)$ of $G$ (and $G'$) such that for any $k\in [n]\setminus \{i\}$ with $(k,j)$ an edge in $G$ we have 
$$
\supp\row_{i}(\Adj(G))\cap \supp\row_k(\Adj(G))=\{j\},
$$
while this property is violated for $G'$. 
Let $(i,j)$ be an edge of $G$ satisfying the above. There are clearly $nd-\antiexp(G)$ choices of such an edge. We will count the number of switchings which makes $(i,j)$ violate the above property. 
To this aim, choose an edge $(a,b)$ such that $b\neq j$ and $b$ is adjacent to $i$. There are at most $d^2$ choices of such an edge. 
Now choose an edge $(a',b')$ such that $a'\neq i$ and $a'$ is adjacent to $j$. There are at most $d^2$ choices for such an edge. 
It is not difficult to see that the switching $\langle a,a',b,b'\rangle$, if possible, would result in $(i,j)$ violating the above property. 
Putting together the estimates, the desired bound follows. 
\item[iv.] Let $h>\zconst$ be an integer.
We define a relation $R$ between the two sets $A_h:=\{G\in \BipGSet_n(d):\, \antiexp(G)=h\}$ and $B_h:=\{G'\in \BipGSet_n(d):\, h-2d^2\leq\antiexp(G')\leq h-1\}$ by letting $(G,G')\in R$ whenever $G\sim G'$. We will view $R$ as a multimap from $A_h$ to $B_h$. \chng{Note that, in view of
the first assertion of the lemma, the value of $\antiexp(\cdot)$ on adjacent graphs can differ by at most $2d^2$.}
Then, using the \chng{second} assertion, we get for any $G\in A_h$ 
$$
\vert R(G)\vert= \vert\{G'\sim G: \antiexp(G')<\antiexp(G)\}\vert \geq \frac{h}{2} (nd-2d^4),
$$
and for any $G'\in B_h$,
$$
\vert R^{-1}(G')\vert\leq \vert\{G'\sim G: \antiexp(G')>\antiexp(G)\}\vert \leq nd^5.
$$
Thus, we deduce that 
$$
\Psimple(A_h)\leq \frac{2nd^5}{h(nd-2d^4)}\Psimple(B_h)= \frac{2nd^5}{h(nd-2d^4)} \sum_{s=h-2d^2}^{h-1} \Psimple(A_s). 
$$
Finally, summing up over $h>\zconst$, we get
\begin{align*}
\Psimple\big(\Expset(\zconst)^c\big) &\leq \frac{2nd^5}{(nd-2d^4)}\sum_{h\geq \zconst +1} \frac{1}{h}\sum_{s=h-2d^2}^{h-1}\Psimple(A_s)\\
&= \frac{2nd^5}{(nd-2d^4)}\sum_{s\geq \zconst +1-2d^2}\Psimple(A_s) \sum_{h=s+1}^{s+2d^2} \frac{1}{h}.
\end{align*}
It remains to use that $\sum_{h=s+1}^{s+2d^2} \frac{1}{h}\leq \frac{2d^2}{ \zconst +2-2d^2}$ to finish the proof. 
\end{itemize}
\end{proof}

Equipped with the definition of $\antiexp(G)$, we can now show that any simple graph belongs to many $s$--neighborhoods provided $\antiexp(G)$ is not too large. 
\begin{lemma}[Number of $s$--neighborhoods containing a graph]\label{l: 98562098724}
Let $G\in \BipGSet_n(d)$. Then for every $1\leq k\leq \cconst$, we have
$$
\frac{\big(nd-\antiexp(G)-2 k d^3\big)^{k} }{k !} (d-1)^{2k}\leq \big|\big\{G'\in \categ(k):\; G\in \SNeigh(G')\big\}\big|\leq {nd \choose k} (d-1)^{2k}. 
$$
Moreover, 
for every $G\in \Expset(\zconst)$,
$$
\big|\big\{G'\in \categ(k):\; G\in \SNeigh(G')\big\}\big|\in \Big[\frac{(nd)^{k}}{2\,k!}(d-1)^{2k}, \frac{(nd)^{k}}{k!}(d-1)^{2k}\Big]. 
$$
\end{lemma}
\begin{proof}
Fix $G\in \BipGSet_n(d)$ and $1\leq k\leq \cconst$. We first establish the upper bound. 
Let us start with complexity of the choice of $k$ multiedges (of multiplicity $2$) over all $G'\in \categ(k)$ for which $G\in \SNeigh(G')$. 
Note that every multiedge of such a graph $G'$ must be an edge in $G$, by the definition of a simple path. 
Clearly, there are at most ${nd \choose k}$ choices of the locations $\{(i_s,j_s)\}_{s=1,\ldots, k}$ of those multiedges. 
Next, we give an upper bound for the number of multigraphs $G'\in \categ(k)$ having multiplicity $2$ edges given by $\{(i_s,j_s)\}_{s=1,\ldots, k}$ and such that $G\in \SNeigh(G')$. 
To reconstruct all possible paths leading to $G$, let $1\leq t\leq k$ and suppose that $P[t],\ldots, P[k]$ are given. Our aim is to control the number of possible realizations of $P[t-1]$, that is, the number of switchings of the form $\langle i',i_t,j_t,j'\rangle$,
such that both $(i',j_t)$ and $(i_t,j')$ are edges in $P[t]$. Clearly, there are at most $d-1$ choices for $i'$ and at most $d-1$ choices for $j'$. 
The upper bound follows. 

Now, we prove the lower bound. 
Define 
\begin{align*}
\mathcal{H}(G):=\big\{
(i,j)\in E_G:\, &\forall\; q\in[n]\setminus\{i\}\mbox{ such that $(q,j)$ is an edge in $G$},\\
&\supp\row_i(\Adj(G))\cap \supp\row_q(\Adj(G))= \{j\}\big\},
\end{align*}
where $E_G$ denotes the set of edges of $G$. 
Note that $\vert \mathcal{H}(G)\vert =nd-\antiexp(G)$. 
We will estimate the number of $k$ tuples $\{(i_s,i_s',j_s,j_s')\}_{s=1,\ldots, k}$ such that 
\begin{itemize}
\item $(i_s,j_s)\in \mathcal{H}(G)$ for any $s=1,\ldots,k$. 
\item $(i_s,j_s),\, (i_s',j_s),\, (i_s, j_s')$ are edges in $G$ but $(i_s',j_s')$ is not. 
\item All $i_s, i_s'$ (resp. $j_s, j_s'$), $s=1,\dots,k$,  are distinct. 
\item The supports of $\row_{i_s}(\Adj(G))$ (resp. $\col_{j_s}(\Adj(G))$), $s=1,\ldots, k$, are pairwise disjoint. 
\end{itemize} 
Note that with those conditions, applying the switchings $\{\langle i_s',i_s,j_s,j_s'\rangle\}_{s=1,\ldots, k}$ to $G$ will
result in a multigraph $G'\in \categ(k)$ for which 
$G\in \SNeigh(G')$. Therefore, it is enough bound the number of such tuples from below. 
Since $\vert \mathcal{H}(G)\vert =nd-\antiexp(G)$, there are $nd-\antiexp(G)$ choices for $(i_1,j_1)$. Moreover, it follows from the definition of $\mathcal{H}(G)$ 
that for any $i_1'\in \supp\col_{j_1}(\Adj(G))\setminus\{i_1\}$
and any $j_1'\in \supp\row_{i_1}(\Adj(G))\setminus\{j_1\}$, $(i_1',j_1')$ is not an edge of $G$. Thus, there are $(d-1)^2$ choices of the couple $(i_1',j_1')$. 

Now note that there are at most $d^2$ indices $i_2$ (resp. $j_2$) such that
$$\supp \row_{i_2}(\Adj(G))\cap \supp \row_{i_1}(\Adj(G))\neq \emptyset$$
\big(resp. $\supp \col_{j_2}(\Adj(G))\cap \supp \row_{j_1}(\Adj(G))\neq \emptyset$\big). Therefore, there are at least $nd-\antiexp(G)-2d^3$ choices for $(i_2,j_2)\in \mathcal{H}(G)$ such that the supports of $\row_{i_2}(\Adj(G))$ (resp. $\col_{j_2}(\Adj(G))$) and that of $\row_{i_1}(\Adj(G))$ (resp. $\col_{j_1}(\Adj(G))$) are disjoint. As above, there are $(d-1)^2$ choices of the couple $(i_2',j_2')$ such that $(i_2',j_2),\, (i_2, j_2')$ are edges in $G$. 

Let $1\leq t< k$ and suppose that $\{(i_s,i_s',j_s,j_s')\}_{s=1,\ldots, t}$  has been chosen. Note that there are at most $td^2$ indices $i_{t+1}$ (resp. $j_{t+1}$) such that $\supp \row_{i_{t+1}}(\Adj(G))\cap \supp \row_{i_s}(\Adj(G))\neq \emptyset$ \big(resp. $\supp \col_{j_{t+1}}(\Adj(G))\cap \supp \row_{j_s}(\Adj(G))\neq \emptyset$\big) for some $s=1,\ldots, t$. Therefore, there are at least $nd-\antiexp(G)-2td^3$ choices for $(i_{t+1},j_{t+1})\in \mathcal{H}(G)$ such that the supports of $\row_{i_{t+1}}(\Adj(G))$ (resp. $\col_{j_{t+1}}(\Adj(G))$) and that of $\row_{i_s}(\Adj(G))$ (resp. $\col_{j_s}(\Adj(G))$) are disjoint for any $s=1,\ldots,t$. As above, there are $(d-1)^2$ choices of the couple $(i_{t+1}',j_{t+1}')$ such that $(i_{t+1}',j_{t+1}),\, (i_{t+1}, j_{t+1}')$ are edges in $G$. 

Thus, the number of $k$ tuples $\{(i_s,i_s',j_s,j_s')\}_{s=1,\ldots, k}$ satisfying the above conditions is at least 
$$
\big(nd-\antiexp(G)\big) \big(nd-\antiexp(G)-2d^3\big) \ldots \big(nd-\antiexp(G)-2(k-1)d^3\big) (d-1)^{2k}.
$$
Finally, note that the order of the above chosen switchings $\{(i_s,i_s',j_s,j_s')\}_{s=1,\ldots, k}$ does not affect the resulting multigraph $G'\in \categ(k)$, so the product we have obtained must be divided by $k!$. This proves the first part of the Lemma. The second claim follows from the choice/assumptions on $\zconst, \cconst$ and $d$. 
\end{proof}

\section{Construction of the function extension}\label{sec: construction}

In this section, we construct, for an arbitrary real valued function $f$ on $\BipGSet_n(d)$, its {\it extension} $\tilde f$ to the space $\ConfBipGSet_n(d)$.
As we discussed in the proof overview, the extension will be constructed in such a way that both the Dirichlet form
and the variance of $\tilde f$ can be bounded appropriately in terms of the Dirichlet form and the variance of the original function $f$.
This shall be achieved by combining two strategies: averaging and {\it controlled fluctuations}. 
Our construction will be randomized, and to this aim we first introduce a special Gaussian field which will be used for the randomization. 
Recall that $\gconst$ and $\cconst$ are two global parameters used throughout the paper (see \eqref{eq: parameters}).

Take any collection of disjoint subsets $(T^\ell)_{\ell=1}^{\gconst}$ of $[n]\times[n]$
and any subset $W\subset\{0,1,\dots,nd\}$, and denote $\mathcal T:=(T^\ell)_{\ell=1}^{\gconst}\sqcup W$. 
Everywhere in this subsection, we let $W^c$ to be the complement of $W$ in $\{0,1,\dots,nd\}$.
We shall define a centered Gaussian field $(\xi_G^{\mathcal T})$ on $\ConfBipGSet_n(d)$ as follows.
Let $(g^\ell)_{\ell=1}^{\gconst}\cup (\tilde g^\ell)_{\ell=1}^{\gconst}$ be i.i.d.\ standard
Gaussian variables.
For every $G\in \ConfBipGSet_n(d)$,
we let
\begin{equation}\label{eq: gaussian-field-def}
\xi_G^{\mathcal T}:=\frac{1}{\sqrt{\gconst}}\bigg(\sum_{\ell=1}^{\gconst} \big(g^\ell\,{\bf 1}_{(\ell,\mathcal T)}
+\tilde g^\ell\,{\bf 1}'_{(\ell,\mathcal T)}\big)\bigg),
\end{equation}
where ${\bf 1}_{(\ell,\mathcal T)}={\bf 1}_{(\ell,\mathcal T)}(G)$ denotes indicator of the boolean expression
$$
\sum_{(i,j)\in T^\ell}\Adj(G)_{i j}\in W,
$$
and ${\bf 1}'_{(\ell,\mathcal T)}={\bf 1}'_{(\ell,\mathcal T)}(G)$
is the indicator of ``$\sum_{(i,j)\in T^\ell}\Adj(G)_{i j}\in W^c$''.

\medskip

Let us record the following properties. 

\begin{lemma}\label{l: 0841847-085-469746}
For any choice of $\mathcal T$, $(\xi_G^{\mathcal T})$ is a centered Gaussian field on $\ConfBipGSet_n(d)$ with $\Var\, \xi_G^{\mathcal T}=1$
for every graph $G\in \ConfBipGSet_n(d)$. 
Moreover, if $G_1, G_2\in \ConfBipGSet_n(d)$ are adjacent, then 
$$
\cov\big(
\xi_{G_1}^{\mathcal T},\xi_{G_2}^{\mathcal T}
\big)\geq 1-\frac{4}{\gconst}.
$$
\end{lemma}
\begin{proof}
The first part of the lemma follows clearly from the definition. 
To prove the second, we start by noticing that
\begin{align*}
&\gconst\,\cov\big(
\xi_{G_1}^{\mathcal T},\xi_{G_2}^{\mathcal T}
\big)=\gconst\,\Exp [\xi_{G_1}^{\mathcal T}\cdot\xi_{G_2}^{\mathcal T}]\\
&=
\Exp \,\bigg(\sum_{\ell_1=1}^{\gconst} \big(g^{\ell_1}\,{\bf 1}_{(\ell_1,{\mathcal T})}(G_1)+\tilde g^{\ell_1}
\,{\bf 1}'_{(\ell_1,{\mathcal T})}(G_1)\big)\bigg)
\bigg(\sum_{\ell_2=1}^{\gconst} \big(g^{\ell_2}\,{\bf 1}_{(\ell_2,{\mathcal T})}(G_2)+\tilde g^{\ell_2}
\,{\bf 1}'_{(\ell_2,{\mathcal T})}(G_2)\big)\bigg)\\
&= \sum_{\ell=1}^{\gconst}{\bf 1}_{(\ell,{\mathcal T})}(G_1)\,{\bf 1}_{(\ell,{\mathcal T})}(G_2)+
\sum_{\ell=1}^{\gconst}{\bf 1}'_{(\ell,{\mathcal T})}(G_1)\,{\bf 1}'_{(\ell,{\mathcal T})}(G_2),
\end{align*}
where $
{\bf 1}_{(\ell_h,{\mathcal T})}(G_h)
$ and ${\bf 1}'_{(\ell_h,{\mathcal T})}(G_h)$,  $h=1,2$, are defined after formula \eqref{eq: gaussian-field-def}. 
Now, we observe that
$\Adj(G_1)$ and $\Adj(G_2)$ differ at four entries, whence
${\bf 1}_{(\ell,{\mathcal T})}(G_1)={\bf 1}_{(\ell,{\mathcal T})}(G_2)$
and ${\bf 1}'_{(\ell,{\mathcal T})}(G_1)={\bf 1}'_{(\ell,{\mathcal T})}(G_2)$
for all but at most $4$ indices $\ell$. Thus,
$$
\gconst\,\cov\big(
\xi_{G_1}^{\mathcal T},\xi_{G_2}^{\mathcal T}
\big)\geq \gconst-4,
$$
and the result follows.
\end{proof}

The following lemma will be used to compare the variance of a function to that of its extension which we will construct at the end of this section. 

\begin{lemma}\label{l: 045602472-4098098}
Let $\mathcal M_{\cconst,\gconst}$ be the set of pairs $(G_1,G_2)\in \categ([1,\cconst])\times \categ([1,\cconst])$ such that
$$
\vert\{(i,j)\in [n]\times [n]:\, \Adj(G_1)_{ij}\neq \Adj(G_2)_{ij}\}\vert\geq \gconst.
$$
Further, let 
$(K_{G_1,G_2})_{(G_1,G_2)\in \mathcal M_{\cconst,\gconst}}$
be any sequence of real numbers indexed over the set $\mathcal M_{\cconst,\gconst}$,
and let $(L_{G_1,G_2})_{(G_1,G_2)\in \mathcal M_{\cconst,\gconst}}$ 
be any sequence of non-negative numbers.
Then there is a choice of the collection
$\mathcal T=(T^\ell)_{\ell=1}^{\gconst}\sqcup W$ such the corresponding Gaussian field
$(\xi_{G}^{\mathcal T})$ on $\ConfBipGSet_n(d)$ satisfies the following condition. 
For any $G_1,G_2\in \ConfBipGSet_n(d)$, let
$\eta_{G_1,G_2}$ be the indicator of the event that both $\xi_{G_1}^{\mathcal T} K_{G_1,G_2}$ and $-\xi_{G_2}^{\mathcal T}K_{G_1,G_2}$
are non-negative, $\xi_{G_1}^{\mathcal T}$ and $-\xi_{G_2}^{\mathcal T}$ have the same sign,
and both $\xi_{G_1}^{\mathcal T}$ and $\xi_{G_2}^{\mathcal T}$ are at least $1$ in absolute value. Then we have
\begin{align*}
\sum_{(G_1,G_2)\in \mathcal M_{\cconst,\gconst}} \eta_{G_1,G_2}L_{G_1,G_2} 
\geq c\sum_{(G_1,G_2)\in \mathcal M_{\cconst,\gconst}} L_{G_1,G_2},
\end{align*}
with probability at least $c$, for some universal constant $c>0$.
\end{lemma}
\begin{proof}
Let $\mathcal S$ be the collection of all unions $(T^\ell)_{\ell=1}^{\gconst}\sqcup W$,
where $(T^\ell)_{\ell=1}^{\gconst}$ is a partition of $[n]\times[n]$
(where we allow the subsets to be empty), and $W$ is a subset of $\{0,1,\dots,nd\}$.
Let $\mu$ be the uniform probability measure on $\mathcal S$
defined as
$$
\mu\big((T^\ell)_{\ell=1}^{\gconst}\sqcup W\big):=2^{-(nd+1)}\gconst^{-n^2},\quad (T^\ell)_{\ell=1}^{\gconst}\sqcup W\in \mathcal S.
$$
A natural interpretation of the measure is the following: let $\xi_{ij}$, $(i,j)\in[n]\times[n]$,
be a collection of i.i.d.\ random variables taking values in $\{1,2,\dots,\gconst\}$
and uniformly distributed in the set, and let $\mathcal W$ be an independent uniform random subset of $\{0,1,\dots,nd\}$. Then
$$
\big(\{(i,j)\in[n]\times[n]:\;\xi_{ij}=\ell\}\big)_{\ell=1}^{\gconst}\sqcup \mathcal W
$$
is distributed according to $\mu$.

Fix for a moment any pair of multigraphs $(G_1,G_2)\in \mathcal M_{\cconst,\gconst}$. 
Let $U$ be the set of all pairs $(i,j)\in[n]\times[n]$ with $\Adj(G_1)_{ij}\neq \Adj(G_2)_{ij}$
and recall that $|U|\geq \gconst$, according to our definition of $\mathcal M_{\cconst,\gconst}$.
Let $b_{ij}$, $(i,j)\in U$, be i.i.d.\ Bernoulli($1/\gconst$) variables.
We claim that
\begin{equation}\label{eq: 1098403948}
\Prob\Big\{\sum_{(i,j)\in U}b_{ij}(\Adj(G_1)_{ij}-\Adj(G_2)_{ij})\neq 0\Big\}\geq c_2,
\end{equation}
for a universal constant $c_2>0$. To see this, we can use a classical anti-concentration inequality of L\'evy--Kolmogorov--Rogozin
\cite{Rogozin}: the L\'evy concentration function of $b_{ij}$ satisfies $\cf(b_{ij},t):=\sup_{\lambda\in\R}
\Prob\{|b_{ij}-\lambda|\leq t\}=1-1/\gconst$ for any $t<1/2$, whence there is a universal constant $C_2\geq 1$
such that \eqref{eq: 1098403948} holds provided that $|U|\geq C_2\,\gconst$. In the situation when $\gconst\leq|U|<C_2\,\gconst$,
we can simply observe that with a constant probability exactly one of the variables in $\{b_{ij}\}_{(i,j)\in U}$
is non-zero, in which case $\big|\sum_{(i,j)\in U}b_{ij}(\Adj(G_1)_{ij}-\Adj(G_2)_{ij})\big|=2$.

Viewing $b_{ij}$ as indicators of the event that $(i,j)\in {\bf T}^{\tilde\ell}$ for a fixed $\tilde \ell$, where
the distribution of random set ${\bf T}^{\tilde\ell}$ is induced by $\mu$, we obtain: for any fixed $1\leq \tilde\ell\leq \gconst$,
we have
$$
\mu\Big\{(T^\ell)_{\ell=1}^{\gconst}\sqcup W:\;\sum_{(i,j)\in T^{\tilde\ell}}\Adj(G_1)_{ij}
\neq\sum_{(i,j)\in T^{\tilde\ell}}\Adj(G_2)_{ij}\Big\}\geq c_2.
$$
Next, observe that, given any two distinct numbers $a$ and $b$ in $\{0,1,\dots,nd\}$
and a random set ${\bf W}$ uniformly distributed on the collection of all subsets of $\{0,1,\dots,nd\}$,
we have $a\in {\bf W}$ and $b\in {\bf W}^c$ with probability $1/4$.
Thus, applying the definition of $\mu$, for every fixed $1\leq \tilde\ell\leq \gconst$ we get
$$
\mu\Big\{(T^\ell)_{\ell=1}^{\gconst}\sqcup W:\;\sum_{(i,j)\in T^{\tilde\ell}}\Adj(G_1)_{ij}\in W
\mbox{ and }\sum_{(i,j)\in T^{\tilde\ell}}\Adj(G_2)_{ij}\in W^c\Big\}\geq c_3,
$$
for a universal constant $c_3:=c_2/4$. Hence, using the reverse Markov inequality,
\begin{align*}
\mu\Big\{(T^\ell)_{\ell=1}^{\gconst}\sqcup W:\;&\sum_{(i,j)\in T^\ell}\Adj(G_1)_{ij}\in W
\mbox{ and }\sum_{(i,j)\in T^\ell}\Adj(G_2)_{ij}\in W^c\\
&\mbox{for at least $c_4\gconst$ indices $\ell$}\Big\}\geq c_4,
\end{align*}
for some universal constant $c_4>0$.

Returning to the Gaussian field, the last relation implies that for any pair of graphs $G_1,G_2$ 
from $\mathcal M_{\cconst,\gconst}$ we have
\begin{align*}
\mu\Big\{\mathcal T=(T^\ell)_{\ell=1}^{\gconst}\sqcup W:\;
|\cov\big(
\xi_{G_1}^{\mathcal T},\xi_{G_2}^{\mathcal T}
\big)|\leq 1-c_5\Big\}\geq c_4
\end{align*}
for a universal constant $c_5>0$.
Note that for any two-dimensional Gaussian $(\bar g,\hat g)$, where $\bar g$ and $\hat g$ are standard and
$|\cov(\bar g,\hat g)|\leq 1-c_5$, and for arbitrary real number $t$, we have
$$
\Prob\big\{\mbox{both $t\,\bar g$ and $-t\,\hat g$ are non-negative;
$\bar g$ and $\hat g$ have opposite signs; }|\bar g|,|\hat g|\geq 1\big\}\geq c_6
$$ 
for a universal constant $c_6>0$.
Accordingly, we have 
\begin{align*}
\mu\Big\{\mathcal T=(T^\ell)_{\ell=1}^{\gconst}\sqcup W:\;
\Prob\big\{\eta_{G_1,G_2}^{\mathcal T}=1\big\}\geq c_6\Big\}\geq c_4.
\end{align*}
The above relation gives
$$
\sum_{\mathcal T}\sum_{(G_1,G_2)\in \mathcal M_{\cconst,\gconst}}L_{G_1 G_2}\,\mu(\mathcal T)\,\Exp\,\eta_{G_1,G_2}^{\mathcal T}
\geq c_4 c_6\sum_{(G_1,G_2)\in \mathcal M_{\cconst,\gconst}}L_{G_1 G_2},
$$
whence there is $\mathcal T$ such that
$$
\sum_{(G_1,G_2)\in \mathcal M_{\cconst,\gconst}}L_{G_1 G_2}\,\Exp\,\eta_{G_1,G_2}^{\mathcal T}
\geq c_4 c_6\sum_{(G_1,G_2)\in \mathcal M_{\cconst,\gconst}}L_{G_1 G_2}.
$$
It remains to apply a reverse Markov inequality to finish the proof. 
\end{proof}

Given any function $f:\, \BipGSet_n(d)\to \R$, 
we define $h:\, \categ([1,\cconst])\to \R$ by 
\begin{equation}\label{eq: hdefinition}
h(G_1)=\frac{1}{\vert \SNeigh(G_1)\vert} \sum_{G\in \SNeigh(G_1)} f(G).
\end{equation}
As a direct corollary of Lemmas~\ref{l: 0841847-085-469746} and~\ref{l: 045602472-4098098}, we get

\begin{prop}[Gaussian field properties]\label{prop: 29870536509385}
Fix a function $f$ on $\BipGSet_n(d)$. 
There is a centered Gaussian field $\{\xi_G\}_{G\in \ConfBipGSet_n(d)}$ with the following properties:
\begin{itemize}
\item $\Var \xi_G=1$;
\item For any two adjacent graphs $G_1,G_2\in \ConfBipGSet_n(d)$,
$\cov(\xi_{G_1},\xi_{G_2})\geq 1-\frac{4}{\gconst}$;
\item Let $\mathcal M_{\cconst,\gconst}$ be the set of pairs $(G_1,G_2)\in \categ([1,\cconst])\times \categ([1,\cconst])$ such that
$$
\vert\{(i,j)\in [n]\times [n]:\, \Adj(G_1)_{ij}\neq \Adj(G_2)_{ij}\}\vert\geq \gconst.
$$
Further, for any $G_1,G_2\in \categ([1,\cconst])$, let
$\eta_{G_1,G_2}$ be the indicator of the event that both $\xi_{G_1}(h(G_1)-h(G_2))$ and $-\xi_{G_2}(h(G_1)-h(G_2))$
are non-negative, $\xi_{G_1}$ and $-\xi_{G_2}$ have the same sign,
and both $\xi_{G_1}$ and $\xi_{G_2}$ are at least $1$ in absolute value. Then we have
\begin{align*}
&\sum_{G_1,G_2\in \mathcal M_{\cconst,\gconst}} \eta_{G_1,G_2}\frac{\Pconfig(G_1)\Pconfig(G_2)}{\vert\SNeigh(G_1)\vert \vert \SNeigh(G_2)\vert}\sum_{G'\in \SNeigh(G_1), G''\in \SNeigh(G_2)} (f(G')-f(G''))^2\\
&\geq c\sum_{G_1,G_2\in \mathcal M_{\cconst,\gconst}} \frac{\Pconfig(G_1)\Pconfig(G_2)}{\vert\SNeigh(G_1)\vert \vert \SNeigh(G_2)\vert}\sum_{G'\in \SNeigh(G_1), G''\in \SNeigh(G_2)} (f(G')-f(G''))^2,
\end{align*}
with probability at least $c$, for some universal constant $c>0$.
\end{itemize}
\end{prop}

Now, we are ready to construct a randomized extension of a function $f$ on $\BipGSet_n(d)$ to the set
of multigraphs $\ConfBipGSet_n(d)$. 

\begin{defi}[Randomized extension] \label{def: extension}
Let $f:\, \BipGSet_n(d)\to \R$. 
Given any graph $G\in \ConfBipGSet_n(d)$, we define $\tilde f(G)$ according to the following rule:
\begin{itemize}

\item If $G\in \BipGSet_n(d)$ and
\begin{equation}\label{eq: condition-blow up-def}
\sum_{G',G''\in \Expset(\zconst)}\Psimple(G') \Psimple(G'') \big(f(G')-f(G'')\big)^2\geq  2^{-7}\Var_{\Psimple}(f),
\end{equation}
then we set $\tilde f(G):=f(G)$. 

\item Otherwise, if $G\in \BipGSet_n(d)$ but
\eqref{eq: condition-blow up-def} does not hold, we set
$$\tilde f(G):=\frac{1}{\sqrt{\Pconfig\big(\BipGSet_n(d)\big)}} \big(f(G)-\Exp_{\Psimple} f\big)+\Exp_{\Psimple} f$$
whenever $G\in \BipGSet_n(d)\setminus \Expset(\zconst)$, and $\tilde f(G)=f(G)$ whenever $G\in \Expset(\zconst)$; 

\item Otherwise, if $G\in \categ([1,\cconst])$, we let $w(G)$ be given by
$$
w(G):= \frac{1}{|\SNeigh(G)|}
\sum_{G'\in \SNeigh(G)}(f(G')-h(G))^2,
$$
and define
$$
\tilde f(G):=h(G)+\xi_G\,\sqrt{w(G)},
$$
where the Gaussian field $\{\xi_G\}$ is given by Proposition~\ref{prop: 29870536509385}
and $h(G)$ is defined in \eqref{eq: hdefinition};

\item Otherwise, we set $\tilde f(G):=\Exp_{\Psimple} f$.

\end{itemize}
\end{defi}
Note that the constructed function $\tilde f$ is, strictly speaking, not an extension because
its value on the simple graphs with large $\antiexp(\cdot)$ may be different from that of $f$.
However, since the number of such graphs is relatively small, we will call $\tilde f$ an extension of $f$
as a minor abuse of terminology.


In what follows, we will need the next simple observation:
\begin{lemma}\label{l: 972-987-9587}
Assume that $f$ is a function on $\BipGSet_n(d)$ of zero mean, and that condition \eqref{eq: condition-blow up-def}
does not hold. Then
\begin{equation*}
\sum_{G'\in \Expset(\zconst)}\Psimple(G')f(G')^2
\leq \frac{1}{2^6}\sum_{G'\in \BipGSet_n(d)\setminus \Expset(\zconst)}
\Psimple(G') f(G')^2.
\end{equation*}
\end{lemma}
\begin{proof}
We have
\begin{align*}
&\sum_{G'\in \Expset(\zconst)}
\Psimple(G') f(G')^2+\Psimple\big(\Expset(\zconst)\big)\Var_{\Psimple}(f)\\
&=\sum_{G'\in \Expset(\zconst),G''\in \BipGSet_n(d)}
\Psimple(G') \Psimple(G'') f(G')^2+\sum_{G'\in \Expset(\zconst),G''\in \BipGSet_n(d)}
\Psimple(G') \Psimple(G'') f(G'')^2\\
&=
\sum_{G'\in \Expset(\zconst),G''\in \BipGSet_n(d)}
\Psimple(G') \Psimple(G'') \big(f(G')-f(G'')\big)^2\\
&<  2^{-7}\Var_{\Psimple}(f)
+\sum_{G'\in \BipGSet_n(d),G''\in \BipGSet_n(d)\setminus \Expset(\zconst)}
\Psimple(G') \Psimple(G'') \big(f(G')-f(G'')\big)^2\\
&\leq 2^{-7}\Var_{\Psimple}(f)+\Psimple\big(\BipGSet_n(d)\setminus \Expset(\zconst)\big)\Var_{\Psimple}(f)+\Var_{\Psimple}(f),
\end{align*}
which implies
\begin{align*}
&\big(1-2^{-7}-2\Psimple\big(\BipGSet_n(d)\setminus \Expset(\zconst)\big)\sum_{G'\in \Expset(\zconst)}
\Psimple(G') f(G')^2\\
&\hspace{1cm}\leq
\big(2^{-7}+2\Psimple\big(\BipGSet_n(d)\setminus \Expset(\zconst)\big)\big)\sum_{G'\in \BipGSet_n(d)\setminus \Expset(\zconst)}
\Psimple(G') f(G')^2.
\end{align*}
The result now follows in view of the estimate
$\Psimple\big( \BipGSet_n(d)\setminus \Expset(\zconst)\big) \leq 2^{-9}$ given by Lemma~\ref{lem: size-bad graphs} together with the choice of parameters in \eqref{eq: parameters} and the assumption on $d$ in \eqref{eq: assumption-d}.
\end{proof}

\section{Coupling of $s$--neighborhoods}\label{sec: bijection}

The goal of this section is to construct a perfect matching between the $s$--neighborhoods of two adjacent graphs, provided they satisfy certain assumptions. 
More precisely, we introduce the following set which we refer to as ``perfect pairs''. 
\begin{defi}[Perfect pairs]\label{def: perfect pair}
Given $1\leq k\leq \cconst$, we define $\Perfmatch(k)$ as the set of adjacent graphs $(G_1,G_2)\in \categ(k)\times \categ(k)$ 
such that the switching $\Phi:=\langle i,i',j,j'\rangle$ used to obtain $G_2$ from $G_1$ satisfies the following conditions:
\begin{itemize}
\item Vertices $i,i',j,j'$ are not incident to any multiedges.
\item Vertices $i,i',j,j'$ are not adjacent to vertices incident to some multiedges.
\end{itemize}
See Figure~\ref{fig: perfpair}.
\end{defi}

\begin{figure}[h]
\caption{A visualization of a perfect pair of adjacent multigraphs $G_1$, $G_2$ (for $d=3$).
The graphs are represented by their adjacency matrices.
The elements colored in black correspond to multiedges. The elements in gray are edges incident to the multiedges.
The rows and columns colored in dark blue correspond to left and right vertices incident to multiedges.
Those colored in light blue correspond to left and right vertices adjacent to vertices incident to some multiedges.
The edges (matrix elements) participating in the simple switching connecting $G_1$ and $G_2$ must lie outside the colored subset.}
\centering
\includegraphics[width=0.60\textwidth]{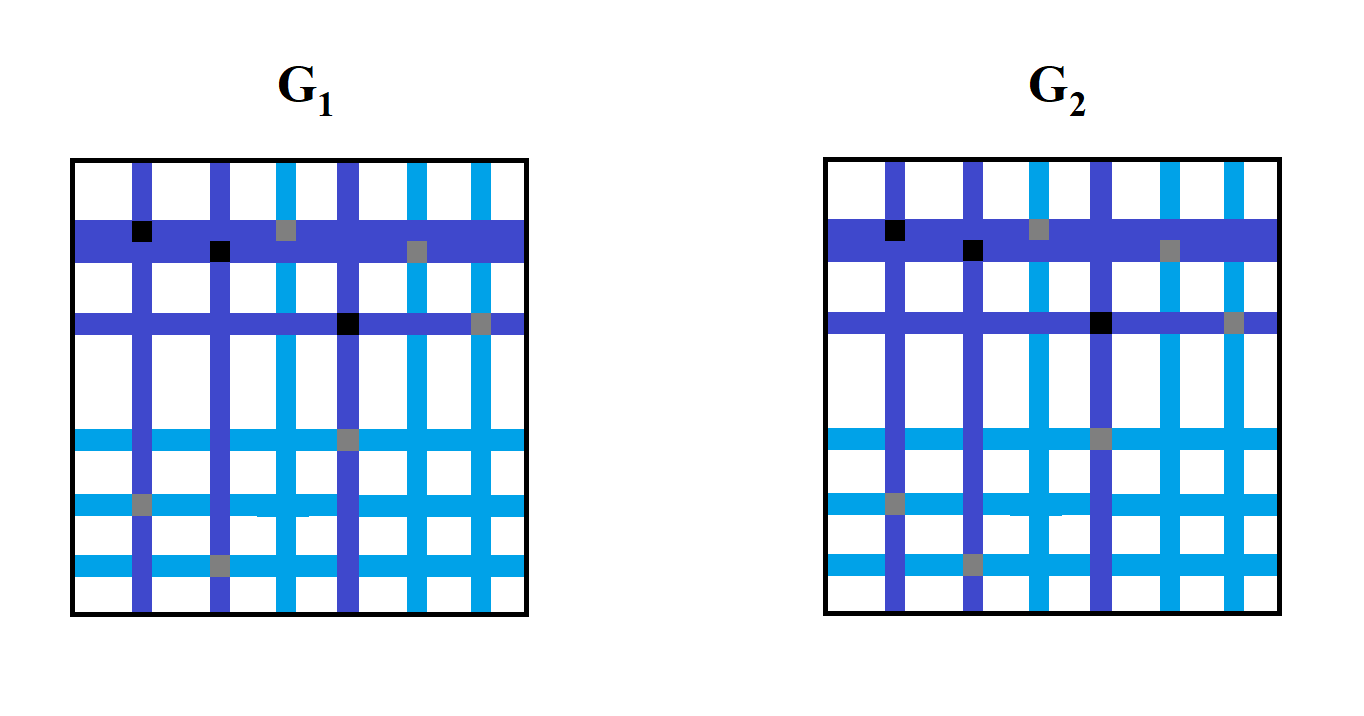}
\label{fig: perfpair}
\end{figure}

As we will see later, $\Perfmatch(k)$ forms a vast majority
in the set of couples of adjacent graphs in $\categ(k)\times \categ(k)$, and are the main ``contributors'' to the Dirichlet form.
The next proposition establishes a property crucial for us: there is a natural coupling between 
the $s$--neighborhoods of two adjacent graphs from  $\Perfmatch(k)$.

\begin{prop}[Matching of $s$--neighborhoods of perfect pairs]\label{prop: bijection}
Let $1\leq k\leq \cconst$ and $(G_1,G_2)\in \Perfmatch(k)$. 
Then there is a bijective mapping $\psi_{G_1,G_2}: \SNeigh(G_1) \to \SNeigh(G_2)$
such that
$$\mbox{$\dist(G,\psi_{G_1,G_2}(G))= 1$ for all $G\in \SNeigh(G_1)$}.$$
\end{prop}
\begin{proof}
Let $1\leq k\leq \cconst$, $(G_1,G_2)\in \Perfmatch(k)$ and $\Phi=\langle i,i',j,j'\rangle$ be the switching used to get $G_2$ from $G_1$. 
Let $G\in \SNeigh(G_1)$ and $\Phi_1,\ldots, \Phi_k$ be a sequence
of switchings such that $G= \Phi_k \circ \ldots\circ \Phi_1 (G_1)$. Recall that this collection is uniquely determined,
up to the ordering (see Lemma~\ref{l: 2958720598750}). 
Let us note that by the assumptions on the vertices of $\Phi$ and the definition of a simple path, the locations $(i,j')$ and $(i',j)$ cannot be 
used by any of the switchings $\Phi_1,\ldots, \Phi_k$
(we recall that, according to our convention, $\Adj(G_1)_{i j'}=\Adj(G_1)_{i' j}=0$). 
Indeed, if that was the case, it would necessarily imply that $i$ or $j'$ (resp.\ $i'$ or $j$) is incident to a multiedge in $G_1$.  
We will consider several cases:

\begin{enumerate}

\item None of the switchings used in $\Phi_1,\ldots, \Phi_k$ involve any of the edges $(i,j)$ or $(i',j')$. Then we set $\psi_{G_1,G_2}(G)$ to be $\Phi(G)$. 
Thus, $G'=\psi_{G_1,G_2}(G)$ is adjacent to $G$; moreover, since $\Phi$ does not operate on any common elements with $\{\Phi_1,\ldots, \Phi_k\}$,
the switchings commute, namely, we have
$$\Phi_k \circ \ldots\circ \Phi_1 (G_2)=\Phi_k \circ \ldots\circ \Phi_1 \circ\Phi (G_1)
=\Phi\circ (\Phi_k \circ \ldots\circ \Phi_1(G_1))=G',$$
whence $G'$ belongs to $\SNeigh(G_2)$. 

\item Exactly one of the two edges $(i,j)$ or $(i',j')$ is operated on by one of the switchings $\Phi_1,\ldots, \Phi_k$. 
Without loss of generality, we suppose that this edge is $(i,j)$. Let $1\leq s\leq k$ be such that the switching $\Phi_s=\langle i_s,i_s',j_s,j_s'\rangle$ 
operates on the edge $(i,j)$, where $(i_s,j_s)$ is a multiedge in $G_1$. 
Then necessarily $i=i_s'$ and $j=j_s'$. 
Note that, by the assumptions on $i,i',j,j'$, the locations $(i_s,j')$ and $(i',j_s)$ are edges neither in $G_1$ nor $G_2$, and they cannot be used by any of the switchings $\Phi_1,\ldots, \Phi_k$. 
Indeed, if $(i_s,j')$ (resp.\ $(i',j_s)$) was an edge, it would imply that $j'$ (resp. $i'$) is adjacent to the vertex $i_s$ (resp.\ $j_s$) which is incident to a multiedge. Moreover, if the location $(i_s,j')$ (resp.\ $(i',j_s)$) 
was used by any of the switchings $\Phi_1,\ldots, \Phi_k$, then it would imply that either $i_s$ (resp.\ $j_s$) is incident to more than one multiedge or $j'$ (resp. $i'$) is incident to a multiedge. 

Define $\psi_{G_1,G_2}(G)$ to be the switching $\langle i_s,i', j,j'\rangle$ applied to $G$;
set $G'=\psi_{G_1,G_2}(G)$. It is not difficult to check that 
$G'= \Phi_k'\circ\ldots\circ \Phi_1'(G_2)$ where $\Phi_s'= \langle i_s,i,j_s,j'\rangle$ and $\Phi_r'=\Phi_r$ for every $r\neq s$. Therefore $G'=\psi_{G_1,G_2}(G)\in  \SNeigh(G_2)$ and is adjacent to $G$.
See Figure~\ref{fig: bijcase2}.

\begin{figure}[h]
\caption{Construction of a bijective mapping between $s$--neighborhoods of $G_1$ and $G_2$, case (2).
The graphs are represented by their adjacency matrices. In the drawing, the switching $\Phi_1=\langle i_1,i,j_1,j\rangle$
operates on $(i,j)$. The corresponding switching for $G_2$ is $\Phi_1'= \langle i_1,i,j_1,j'\rangle$.}
\centering
\includegraphics[width=0.80\textwidth]{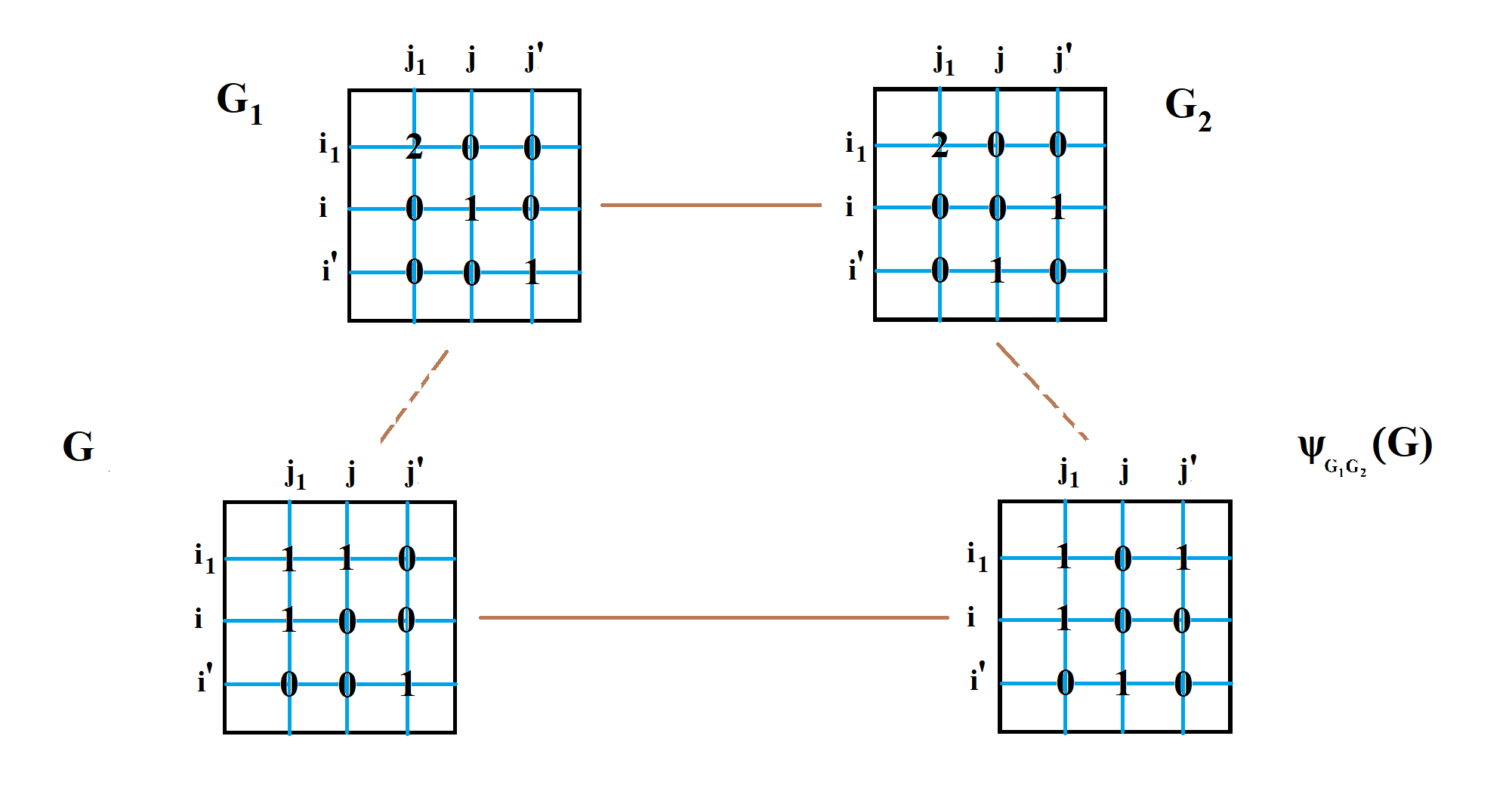}
\label{fig: bijcase2}
\end{figure}

\item Suppose that both edges $(i,j)$ and $(i',j')$ are operated on by the switchings $\Phi_1,\ldots, \Phi_k$.
This means that there exist two indices $1\leq s_1,s_2\leq k$ such that $\Phi_{s_1}=\langle i_{s_1},i_{s_1}', j_{s_1},j_{s_1}'\rangle$ (resp.\ $\Phi_{s_2}=\langle i_{s_2},i_{s_2}', j_{s_2},j_{s_2}'\rangle$) operates on $(i,j)$ (resp.\ $(i',j')$),
where $(i_{s_1},j_{s_1})$ and $(i_{s_2},j_{s_2})$ are multiedges in $G_1$.
Thus, $i=i_{s_1}'$, $j=j_{s_1}'$, $i'=i_{s_2}'$, $j'=j_{s_2}'$. 

Note that the locations $(i_{s_1},j')$, $(i',j_{s_1})$, $(i_{s_2},j)$, $(i,j_{s_2})$ are edges neither in
$G_1$ nor $G_2$ since otherwise, this would imply (in the same order) that $j'$, $i'$, $j$, $i$ are adjacent to a vertex incident to a multiedge. Moreover, these locations could not have been used in any of the switchings $\Phi_1,\ldots, \Phi_k$ since otherwise, this would imply (in the same order) that either 
$i_{s_1}$, $j_{s_1}$, $i_{s_2}$, $j_{s_2}$ are incident to more than one multiedge or that $j'$, $i'$, $j$, $i$ are incident to a multiedge. 

Define $\psi_{G_1,G_2}(G)$ to be the switching $\langle i_{s_1}, i_{s_2}, j, j'\rangle$ applied to $G$; set
$G'=\psi_{G_1,G_2}(G)$. Note that $G'= \Phi_k'\circ\ldots\circ \Phi_1'(G_2)$ where 
$\Phi_{s_1}'= \langle i_{s_1}, i, j_{s_1}, j'\rangle$ and $\Phi_{s_2}'=\langle i_{s_2}, i', j_{s_2},j\rangle$ while $\Phi_r'=\Phi_r$ for any $r\neq s_1,s_2$. 
Therefore $G'=\psi_{G_1,G_2}(G)\in  \SNeigh(G_2)$ and is adjacent to $G$.
See Figure~\ref{fig: bijcase3}.

\begin{figure}[h]
\caption{Construction of a bijective mapping between $s$--neighborhoods of $G_1$ and $G_2$, case (3).
The graphs are represented by their adjacency matrices. In the drawing, the switching $\Phi_1=\langle i_1,i,j_1,j\rangle$
operates on $(i,j)$, and the switching $\Phi_2=\langle i_{2},i', j_{2},j'\rangle$ operates on $(i',j')$.
The corresponding switchings for the graph $G_2$ are $\Phi_1'=\langle i_1, i, j_1, j'\rangle$ and
$\Phi_2'=\langle i_2, i', j_2,j\rangle$.}
\centering
\includegraphics[width=0.80\textwidth]{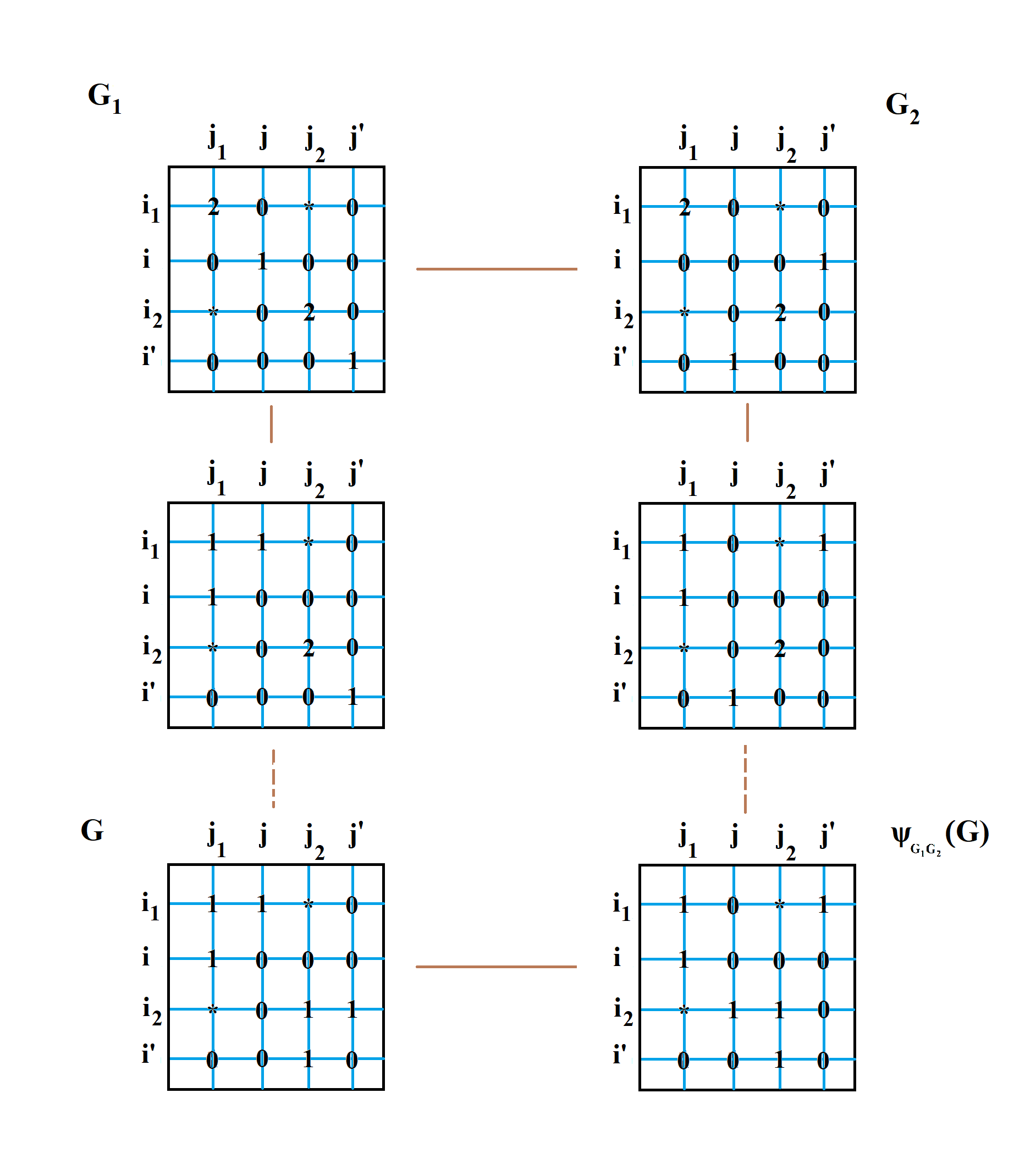}
\label{fig: bijcase3}
\end{figure}
\end{enumerate}
With the above definition, it is not difficult to check that $\psi_{G_1,G_2}\circ \psi_{G_2,G_1}$ (resp.\ $\psi_{G_2,G_1}\circ \psi_{G_1,G_2}$) is the identity operator on $\SNeigh(G_2)$ (resp.\ $\SNeigh(G_1)$). Moreover, we have $\dist(G,\psi_{G_1,G_2}(G))= 1$ for all $G\in \SNeigh(G_1)$. This finishes the proof. 
\end{proof}

\begin{rem}
In what follows, we always refer to $\psi_{G_1,G_2}$ constructed in the above proposition as {\it the} bijective mapping between $\SNeigh(G_1)$ and $\SNeigh(G_2)$.
\end{rem}

It turns out that the above procedure can be reversed in the following sense.

\begin{prop}[Matching of perfect pairs]\label{prop: matching perfect pairs}
Let $1\leq k\leq \cconst$ and $G_1\in \categ(k)$. Let $G\in \SNeigh(G_1)$ and let $G'\in \BipGSet_n(d)$ be adjacent to $G$. 
Then there exists at most one multigraph $G_2\in  \categ(k)$ such that $(G_1,G_2)\in \Perfmatch(k)$ and $\psi_{G_1,G_2}(G)=G'\in  \SNeigh(G_2)$.
\end{prop}
\begin{proof} 
Let $\Phi_1=\langle i_1,i_1',j_1,j_1'\rangle, \ldots, \Phi_k=\langle i_k,i_k',j_k,j_k'\rangle$ be the unique collection
of switchings leading from $G_1$ to $G$ (as before, we assume that
$(i_s,j_s)$ are the multiedges of $G_1$) and let $\psi=\langle a,a',b,b'\rangle$ be the switching used to pass from $G$ to $G'$.  
For any $G_2\in \categ(k)$ with $(G_1,G_2)\in \Perfmatch(k)$, denote by $\Phi_{G_1,G_2}=\langle i_{G_1,G_2},i_{G_1,G_2}',j_{G_1,G_2},j_{G_1,G_2}'\rangle$ the simple switching used to pass from $G_1$ to $G_2$. It follows from the construction of $\psi_{G_1,G_2}$ that the switchings  $\psi_{G_1,G_2}$ and $\Phi_{G_1,G_2}$ share the same right vertices. 
Moreover, the left vertices operated on by $\Phi_{G_1,G_2}$ and $\psi_{G_1,G_2}$,
coincide whenever they do not belong to $\{i_s\}_{1\leq s\leq k}$. 

Suppose there are two graphs $G_2,G_2'\in  \categ(k)$ such that $(G_1,G_2), (G_1,G_2')\in \Perfmatch(k)$ and $G'=\psi(G)=\psi_{G_1,G_2}(G)=\psi_{G_1,G_2'}(G)$. 
We will prove that $\Phi_{G_1,G_2}=\Phi_{G_1,G_2'}$ which would imply that $G_2=G_2'$. 
It follows from the above that $\Phi_{G_1,G_2}$ and $\Phi_{G_1,G_2'}$ share the same right vertices and these are $b$ and $b'$ i.e. $j_{G_1,G_2}=j_{G_1,G_2'}=b$ and $j_{G_1,G_2}'=j_{G_1,G_2'}'=b'$. 

We will consider several cases parallel to the ones in the proof of Proposition~\ref{prop: bijection}:
\begin{enumerate}
\item Neither of the two left vertices $a,a'$ belong to $\{i_s\}_{1\leq s\leq k}$.
This means that $\psi_{G_1,G_2}(G)=\psi_{G_1,G_2'}(G)$ is constructed as in the first case of the proof of Proposition~\ref{prop: bijection},
i.e.\ $i_{G_1,G_2}= i_{G_1,G_2'}=a$, $i_{G_1,G_2}'= i_{G_1,G_2'}'=a'$, and $\Phi_{G_1,G_2}=\Phi_{G_1,G_2'}=\langle a,a',b,b'\rangle$.

\item Exactly one of the left vertices $a,a'$ belongs to $\{i_s\}_{1\leq s\leq k}$. Without loss of generality, we suppose it is $a$. Let $1\leq s\leq k$ be such that $i_s=a$. Note that since $a'\not\in \{i_s\}_{1\leq s\leq k}$, then we have $i_{G_1,G_2}'= i_{G_1,G_2'}'=a'$. Necessarily, we are in the second case of Proposition~\ref{prop: bijection},
i.e.\ the switching used to pass from $G$ to $G'$ is
$$\psi=\langle i_s,i_{G_1,G_2}',j_{G_1,G_2},j_{G_1,G_2}'\rangle
=\langle i_s,i_{G_1,G_2'}',j_{G_1,G_2'},j_{G_1,G_2'}'\rangle.
$$
Note that the second case of Proposition~\ref{prop: bijection} assumes that $\Phi_s$ operates on one of the edges $(i_{G_1,G_2},j_{G_1,G_2})$
or $(i_{G_1,G_2}',j_{G_1,G_2}')=(a',b')$, i.e.\ either $(i_s',j_s')=(i_{G_1,G_2},b)$
or $(i_s',j_s')=(a',b')$.
The latter case is impossible because it would mean that the edge $(a',b')$ is removed from the graph $G$, so $\psi$ could not be applied to $G$.
Thus, 
$i_{G_1,G_2}=i_s'$.
By the same reasoning we get that $i_{G_1,G_2'}=i_s'$, 
so $\Phi_{G_1,G_2}=\Phi_{G_1,G_2'}=\langle i_s',a',b,b'\rangle$.

\item Both left vertices $a,a'$ belong to $\{i_s\}_{1\leq s\leq k}$. Let $1\leq s_1,s_2\leq k$ be such that $a=i_{s_1}$ and $a'=i_{s_2}$. Note that it is the third case of Proposition~\ref{prop: bijection} which could lead to such configuration:
$$
\psi=\langle i_{s_1},i_{s_2},j_{G_1,G_2},j_{G_1,G_2}'\rangle
=\langle i_{s_1},i_{s_2},j_{G_1,G_2'},j_{G_1,G_2'}'\rangle.
$$
Note that the third case assumes that $\Phi_{s_1}$ operates on the edge $(i_{G_1,G_2},j_{G_1,G_2})$
and $\Phi_{s_2}$ operates on $(i_{G_1,G_2}',j_{G_1,G_2}')$, i.e.\ $(i_{s_1}',j_{s_1}')=(i_{G_1,G_2},j_{G_1,G_2})$,
$(i_{s_2}',j_{s_2}')=(i_{G_1,G_2}',j_{G_1,G_2}')$, so
$\Phi_{G_1,G_2}=\langle i_{s_1}',i_{s_2}',b,b'\rangle$. Similarly, we must have $\Phi_{G_1,G_2'}=\langle i_{s_1}',i_{s_2}',b,b'\rangle$,
so 
$\Phi_{G_1,G_2}=\Phi_{G_1,G_2'}$, and the proof is complete. 
\end{enumerate}
\end{proof}

\section{Connections}\label{sec: connections}

The goal of this section is to define canonical paths
between simple graphs, which will be called {\it connections}.
As was mentioned in the proof overview, connections will be used when bounding from above the Dirichlet
form of the function extension in terms of the Dirichlet form of the original function.
Whereas the ``main weight'' of the Dirichlet form of the extension is comprised by perfect pairs
of multigraphs which admit a bijective mapping between their $s$--neighborhoods (see Section~\ref{sec: bijection}),
the imperfect pairs have to be dealt with differently.
We prefer to give a simplistic viewpoint here.
Given an imperfect pair of multigraphs $(G_1',G_2')$, we will first match all couples $(G_1,G_2)$
of simple graphs within the $s$--neighborhoods of $G_1'$ and $G_2'$ and then will bound
expressions of the form $(f(G_1)-f(G_2))^2$, which shall appear in computations, by
$$
\length{P}\,\sum_{t=1}^{\length{P}}(f(P[t])-f(P[t-1]))^2,
$$
where $P$ is a connection between $G_1$ and $G_2$ --- a path starting at $G_1$ and ending at $G_2$
that is determined not only by $G_1$ and $G_2$ but also the multigraphs $G_1',G_2'$ and thus depends on the entire $4$--tuple
$\tuple=(G_1,G_1',G_2,G_2')$.
Connections must be constructed in such a way that in the resulting weighted sum over $4$--tuples
$$
\sum\limits_{\tuple}\sum_{t=1}^{\length{P_\tuple}} \mbox{weight}(t,\tuple)\;(f(P_\tuple[t])-f(P_\tuple[t-1]))^2,
$$
no pair of adjacent simple graphs receives ``too big'' total weight,
and this is the main difficulty in properly defining them.
Connections will also be employed when mapping simple graphs within a common $s$--neighborhood,
and in the situation when a multigraph is adjacent to a simple graph which does not belong to its $s$--neighborhood.

Thus, we will deal with three types of connections:
\begin{itemize}

\item {\it Type $1$.} Paths connecting simple graphs in $s$--neighborhoods of two adjacent multigraphs.

\item {\it Type $2$.} Paths connecting two simple graphs within the $s$--neighborhood of a given multigraph.

\item {\it Type $3$.} A path connecting a simple graph within the $s$--neighborhood of a multigraph, and
a simple graph adjacent to the multigraph but not contained within its $s$--neighborhood.

\end{itemize}

Let us introduce a classification on quadruples of graphs.
Let $\tuple:=(G_1,G_1',G_2,G_2')$ be a $4$--tuple of graphs in $\ConfBipGSet_n(d)$, where
$G_1,G_2\in \BipGSet_n(d)$. We say that $\tuple$ is
\begin{itemize}

\item {\it Of type $1$}, if $G_1\in \SNeigh(G_1')$, $G_2\in \SNeigh(G_2')$, the two graphs $G_1'$ and $G_2'$ are adjacent in
$\ConfBipGSet_n(d)$; $G_1',G_2'\in \categ([1,\cconst])$, and the pair $(G_1',G_2')$ is not perfect
(see Definition~\ref{def: perfect pair}).

\item {\it Of type $2$}, if $G_1\in \SNeigh(G_1')$, $G_2\in \SNeigh(G_2')$, and $G_1'=G_2'\in \categ([1,\cconst])$.

\item {\it Of type $3$}, if $G_1=G_1'$ is adjacent to $G_2'\in\categ([1,\cconst])$, and $G_2\in \SNeigh(G_2')$, $G_1\notin \SNeigh(G_2')$.

\end{itemize}

A $4$--tuple of any of the above three types will be called {\it admissible}.
We would like to emphasize that the graphs $G_1'$ and $G_2'$ for a type $1$ tuple may be of different categories
(this may happen if the switching operation between the graphs operates on multiedges).
Clearly, in the cases of type $2$ and $3$, $3$--tuples could have been used without any loss of information,
but the above convention will allow us to treat all three types in a uniform way whenever possible.
The following simple lemma follows directly from the definitions:
\begin{lemma}
Let $\tuple:=(G_1,G_1',G_2,G_2')$ be of type $3$. Then necessarily $G_2'$ has category number $2$,
with two multiedges of multiplicity two each, and the simple switching used to produce $G_1=G_1'$ from $G_2'$,
operates on those multiedges.
\end{lemma}
\begin{proof}
We refer to Figure~\ref{fig: type3} for illustration.
\end{proof}

\begin{figure}[h]
\caption{The structure of a type $3$ tuple $\tuple=(G_1,G_1',G_2,G_2')$.
The graphs are represented by their adjacency matrices.}
\centering
\includegraphics[width=0.95\textwidth]{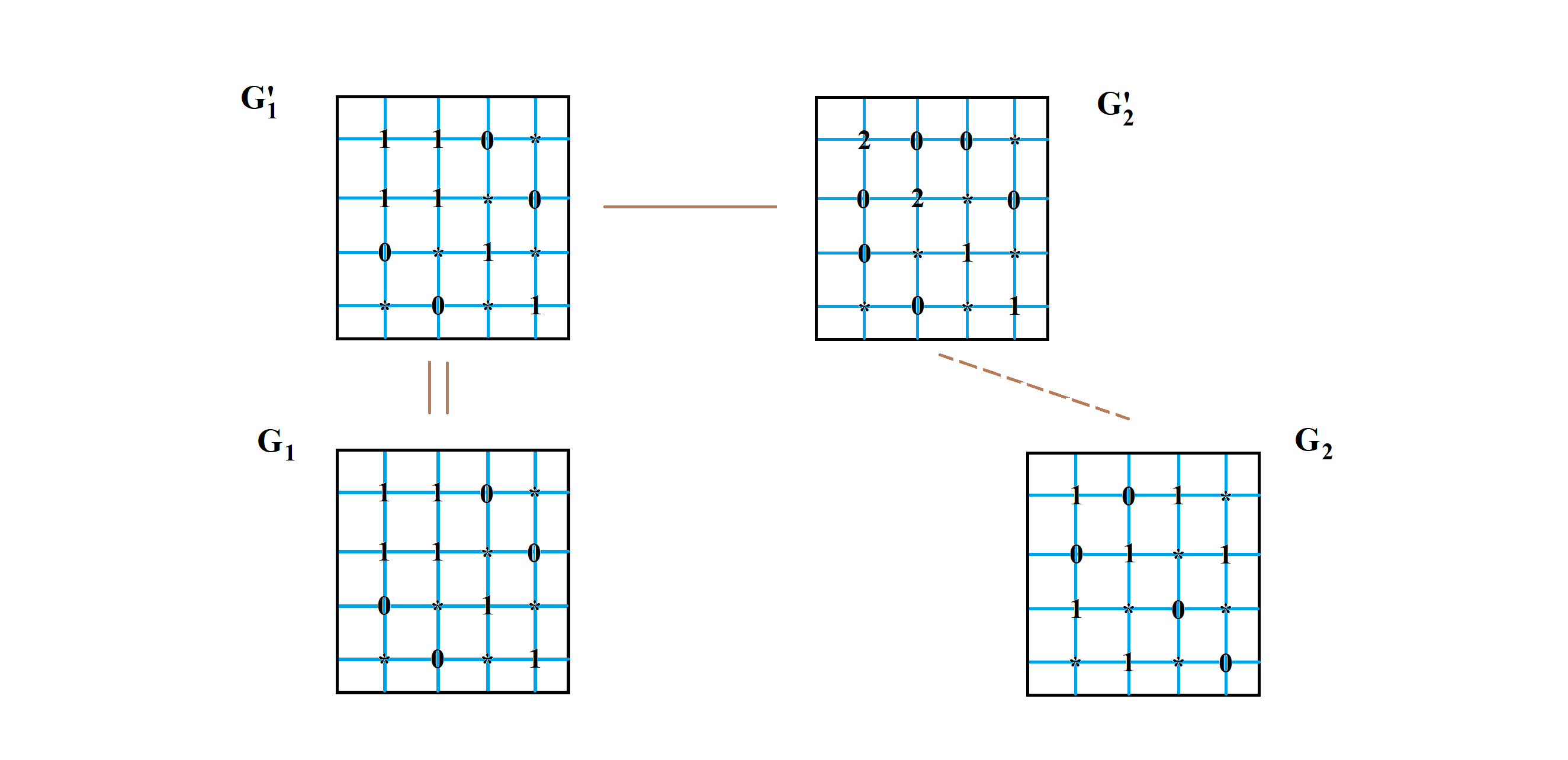}
\label{fig: type3}
\end{figure}

\medskip

Let $\tuple:=(G_1,G_1',G_2,G_2')$ be any admissible $4$--tuple.
Denote by $I(G_2,G_2')$ the collection of all indices $i\leq n$ such that a simple path leading from $G_2'$
to $G_2$ contains a simple switching involving left vertex $i^\ell$.
We will partition $I(G_2,G_2')$ into two subsets
$$
I(G_2,G_2')=I_m(G_2,G_2')\sqcup I_s(G_2,G_2'),
$$
where $I_m(G_2,G_2')$ is defined as the collection of all left vertices $i^\ell$ such that $G_2'$ has a multiple edge
emanating from $i^\ell$. Clearly, $I_s(G_2,G_2')$ then consists of all the left vertices
which do not have incident multiple edges and which are employed in the simple switching operations leading 
from $G_2'$ to $G_2$.
We define $I(G_1,G_1')=I_m(G_1,G_1')\sqcup I_s(G_1,G_1')$ according to the same rules
for types $1$ and $2$; let $I_m(G_1,G_1')=I_s(G_1,G_1')=\emptyset$ for type $3$ tuples,
and let $I(\tuple):= I(G_1,G_1')\cup I(G_2,G_2')$.

\begin{defi}\label{def: standard}
We say that an edge $(i,j)$ is {\bf $m$--standard} with respect to a $4$--tuple $\tuple=(G_1,G_1',G_2,G_2')$
of type $1$ or $2$ if all of the following conditions
are satisfied (see Figure~\ref{fig: m-standard}):
\begin{enumerate}

\item The multiplicity of $(i,j)$ both in $G_1'$ and $G_2'$ is $2$;


\item\label{p: standard-1} The (unique) left vertex $\ione\neq i$ with $(\ione,j)$ an edge in $G_1$ but not an edge in $G_1'$,
is not contained in $I(G_2,G_2')$, and

\item\label{p: standard-2} The (unique) left vertex $\itwo\neq i$ with $(\itwo,j)$ an edge in $G_2$ but not an edge in $G_2'$,
is not contained in $I(G_1,G_1')$;

\item\label{p: 13981749827} For the unique index $\jtwo$ such that $(i,\jtwo)$ an edge in $G_2$ but not an edge in $G_2'$,
the couple $(\ione,\jtwo)$ is not an edge in $G_1$ or in $G_2$;

\item\label{p: standard-3} In case when $G_1'$ and $G_2'$ do not coincide, the simple switching operation used to pass from $G_1'$ to $G_2'$,
does not operate on any of the left vertices $\{i,\ione,\itwo\}$.

\end{enumerate} 
\end{defi}

For future reference, we provide an illustration in the case when all properties above, except for property~\ref{p: 13981749827},
are satisfied; see Figure~\ref{fig: non-standard}.

Denote the set of all $m$--standard edges w.r.t.\ $\tuple$ by $E_S(\tuple)$
(for type $3$ tuples, we will set $E_S(\tuple):=\emptyset$).
Observe that each $m$--standard edge $(i,j)\in E_S(\tuple)$ naturally defines four numbers (indices),
which we will denote by $\ione(i,\tuple)$, $\itwo(i,\tuple)$, $\jone(i,\tuple)$, $\jtwo(i,\tuple)$:

\begin{itemize}

\item $\ione=\ione(i,\tuple)$ is the unique left vertex such that
the pair of vertices $(\ione,i)$ is operated on by a switching in a simple path from $G_1'$ to $G_1$.
Note that $(\ione,j)$ an edge in $G_1$ but not an edge in $G_1'$;

\item $\itwo=\itwo(i,\tuple)$ is the unique left vertex which, together with $i$,
is operated on by a switching in a simple path from $G_2'$ to $G_2$.
Again, we note that that $(\itwo,j)$ an edge in $G_2$ but not an edge in $G_2'$;

\item $\jone=\jone(i,\tuple)$ is the unique right vertex such that $(i,\jone)$ an edge in $G_1$ but not an edge in $G_1'$;

\item $\jtwo=\jtwo(i,\tuple)$ is the unique right vertex such that $(i,\jtwo)$ an edge in $G_2$ but not an edge in $G_2'$.

\end{itemize}

Note that the uniqueness of $\ione(i,\tuple)$ and $\itwo(i,\tuple)$ follows from the fact that each double edge of $G_1'$ or $G_2'$ is uniquely
mapped to a generating simple switching for $G_1$ and $G_2$, i.e.\ is guaranteed by Lemma~\ref{l: 2958720598750}.

\begin{figure}[h]
\caption{A schematic depiction of an $m$--standard edge $(i,j)$
(the multigraphs are represented by their adjacency matrices).}
\centering
\includegraphics[width=0.75\textwidth]{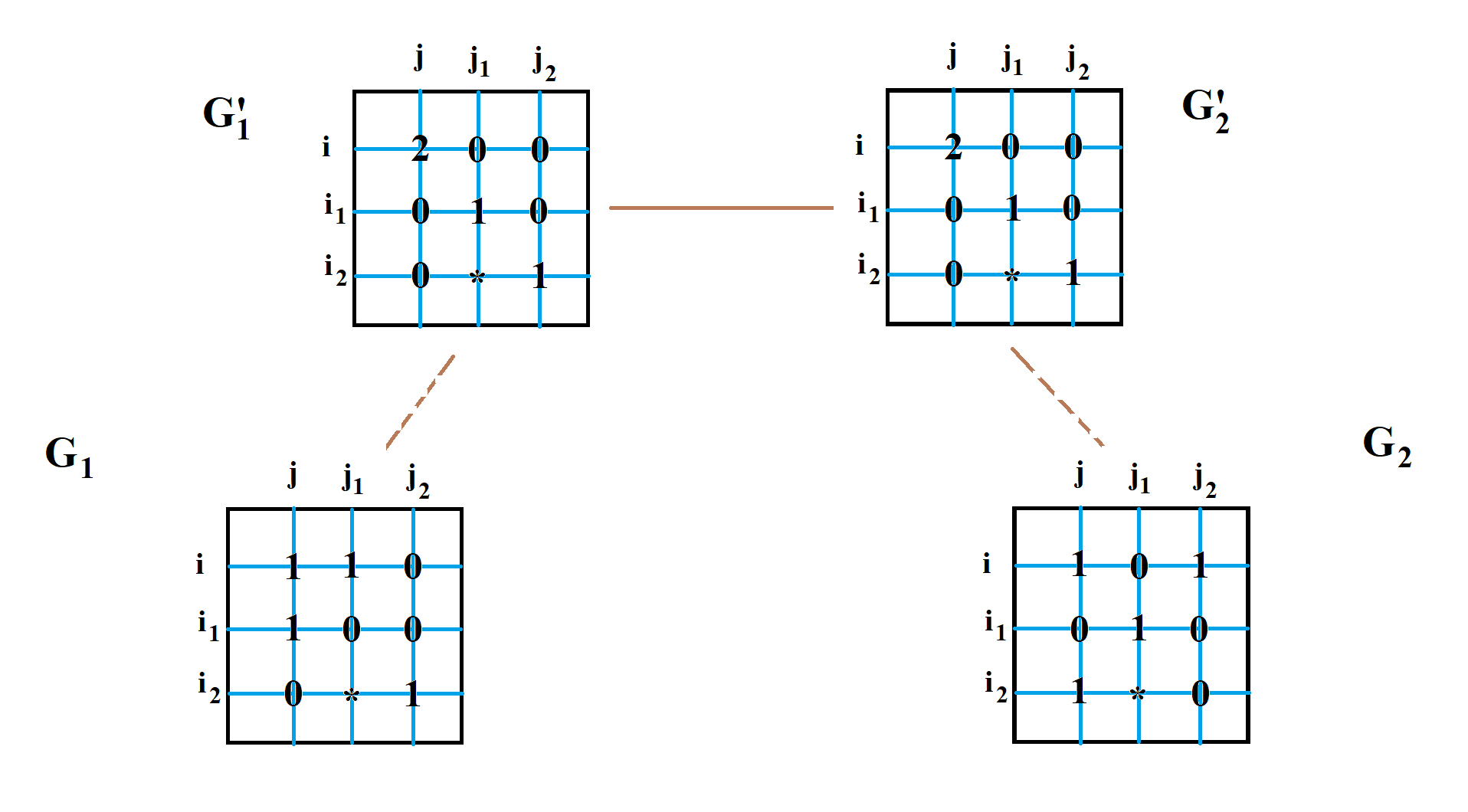}
\label{fig: m-standard}
\end{figure}

\begin{figure}[h]
\caption{An edge $(i,j)$ satisfying all conditions for an $m$--standard edge, except for property~\ref{p: 13981749827}
(the multigraphs are represented by their adjacency matrices). More precisely, we illustrate two situations: the first when the edge $(\ione, \jtwo)$ is present 
in both $G_1$ and $G_2$, while the second is when the edge is present in $G_2$ and absent in $G_1$.}
\centering
\includegraphics[width=0.75\textwidth]{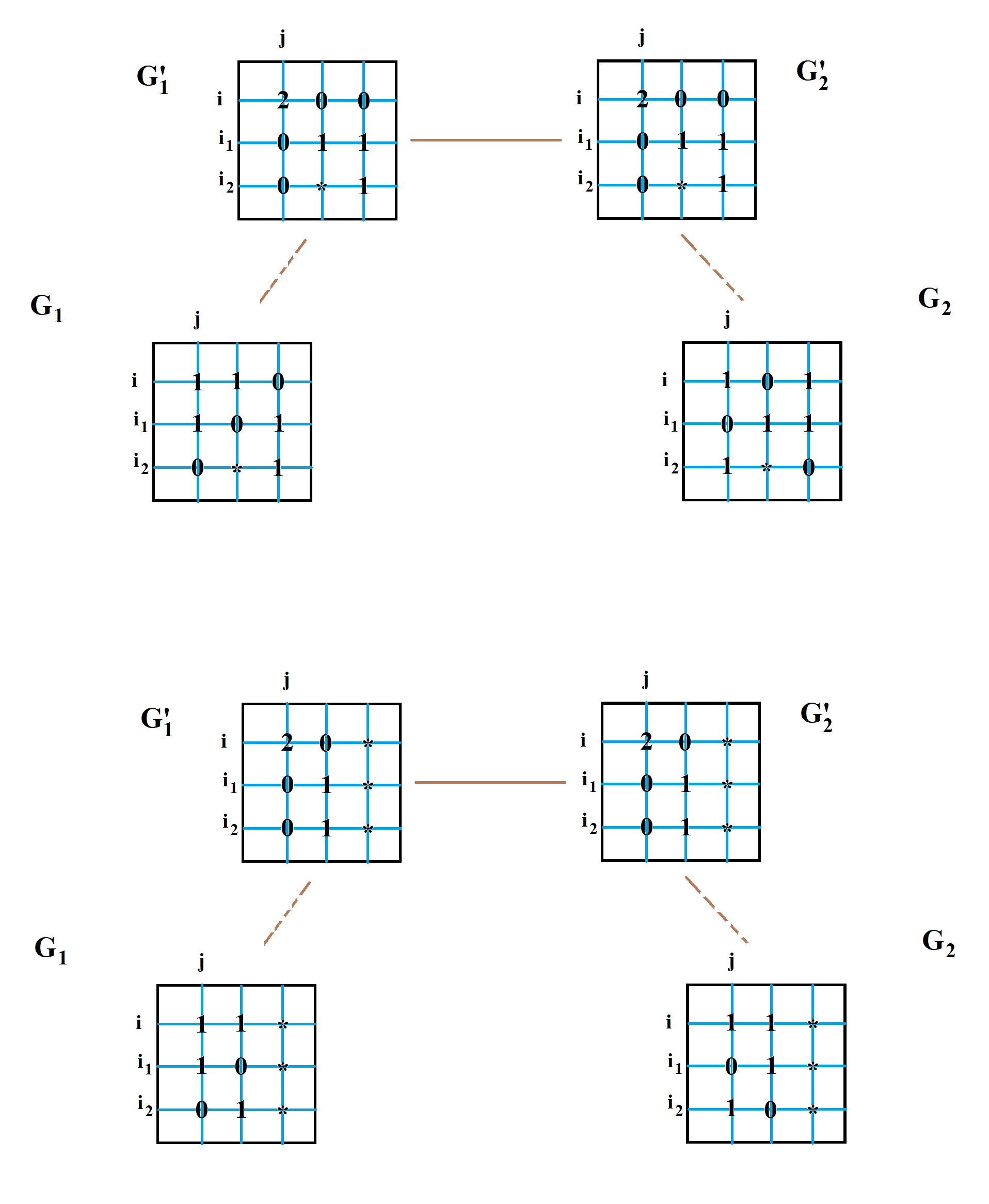}
\label{fig: non-standard}
\end{figure}

For any admissible $\tuple=(G_1,G_1',G_2,G_2')$, denote
\begin{align*}
I_{ST}(\tuple):=&\{i\leq n:\;\mbox{$(i,j)\in E_S(\tuple)$ for some $j$}\}\cup
\{\tilde i\leq n:\; \mbox{there are $i,j\leq n$ such that}\\
&\hspace{6.5cm}\mbox{$(i,j)\in E_S(\tuple)$ and $\tilde i\in\{\ione(i,\tuple),\itwo(i,\tuple)\}$}\}.
\end{align*}
The set $I_{ST}(\tuple)$ comprises ``standard'' left vertices for which a definition of corresponding switching
operations leading from $G_1$ to $G_2$ is relatively straightforward.

\begin{defi}
Let $\tuple=(G_1,G_1',G_2,G_2')$ be of type $1$ or $3$.

\begin{itemize}
\item We say that the simple switching operation used to transfer between $G_1'$ and $G_2'$, is
of {\bf type A}
if the left vertices it operates on are not contained in $I(\tuple)$.


\item The switching is of {\bf type B} if at least one of the left vertices it operates on is in $I(\tuple)$.

\end{itemize}

Note that in case of type $3$ tuples, the switching is necessarily of type $B$, whereas in case of type $1$ tuples
it can be either type $A$ or type $B$.
Moreover, for an admissible tuple $\tuple$ with type A switching, the condition that the pair $(G_1',G_2')$ is not perfect,
immediately implies that either the vertices the switching operates on are incident to
some multiedges of $G_1',G_2'$
or those vertices are adjacent to vertices incident to some multiedges
(see Definition~\ref{def: perfect pair}).
\end{defi}

For any admissible tuple $\tuple=(G_1,G_1',G_2,G_2')$, denote by $I_A(\tuple)$
the set of left vertices operated on by the simple switching of type $A$ leading from $G_1'$ to $G_2'$,
whenever such a switching is defined, and set $I_A(\tuple):=\emptyset$ otherwise.

\medskip

Let $\tuple=(G_1,G_1',G_2,G_2')$ be any admissible $4$--tuple.
We define the collection of ``non-standard'' left indices with respect to $\tuple$ as follows.
Set $\tilde I$ to be empty set if $G_1'=G_2'$ (i.e.\ $\tuple$ is of type $2$)
or if the simple switching operating between the two graphs is of type $A$,
and let $\tilde I$ be the (two) left vertices operated on by the switching from $G_1'$ to $G_2'$ if it is of type $B$.
Then we set
\begin{align*}
I_{NS}(\tuple):=\big(I(\tuple)\setminus I_{ST}(\tuple)\big)\cup \tilde I.
\end{align*}
The set $I_{NS}(\tuple)$ comprises ``non-standard'' left indices.

\medskip

Note that with the above definitions, with each admissible tuple $\tuple=(G_1,G_1',G_2,G_2')$
we have associated three disjoint sets $I_{ST}(\tuple)$, $I_{NS}(\tuple)$ and $I_A(\tuple)$.
The canonical path ({\it{}a connection}) between the simple graphs $G_1$ and $G_2$ will be constructed
by combining simple switching operations on $I_{ST}(\tuple)$, followed by the switching on $I_A(\tuple)$ (if any),
and then switchings on the left vertices from $I_{NS}(\tuple)$.
The major issue in constructing the canonical paths is to find a construction procedure which would guarantee
that there is no pair of adjacent simple graphs belonging to ``too many'' connections
(in which case we would not be able to efficiently bound the Dirichlet form).
The most difficult element of the construction is treatment of the set $I_{NS}(\tuple)$,
because of its relatively complicated structure.

\subsection{Structure of the set $I_{NS}(\tuple)$}

The following simple estimates will be important:
\begin{lemma}\label{l: 2-9572059827098}
For any admissible tuple $\tuple=(G_1,G_1',G_2,G_2')$,
let $M_h$ be 
the number of multiedges incident to $I_{NS}(\tuple)$ in $G_h'$, $h=1,2$.
Then necessarily
$$
M_h\leq |I_{NS}(\tuple)|, \quad h=1,2.
$$
\end{lemma}
\begin{proof}
The lemma follows from the definition of $I_{NS}(\tuple)$ and the definition of an $s$--neighborhood.
\end{proof}

\begin{lemma}\label{l: 208762058709283}
For any admissible tuple $\tuple=(G_1,G_1',G_2,G_2')$, the cardinality of the set of non-standard left indices $I_{NS}(\tuple)$ satisfies
\begin{align*}
&|\{(i,j):\;i\in I_{NS}(\tuple),\,(i,j)\mbox{ is an edge in precisely one of the graphs $G_1,G_2$}\}|\\
&\leq 8\,|I_{NS}(\tuple)|
+4\cdot{\bf 1}_{\{\mbox{$G_1'$ and $G_2'$ are adjacent and connected by a type B switching}\}}.
\end{align*}
\end{lemma}
\begin{proof}
Denote by $Y$ the set of left vertices operated on by the simple switching connecting $G_1'$ and $G_2'$
(we set $Y=\emptyset$ if $G_1'=G_2'$).
First, take any $i\in I_{NS}(\tuple)\setminus Y$,
and observe that $\row_i(\Adj(G_1'))=\row_i(\Adj(G_2'))$, and, in particular, 
the number $M(i)$ of multiedges of $G_h'$ incident to $i^\ell$ is the same for $h=1,2$.
Then the number of indices $j$ such that $\mult_{G_1}(i,j)\neq \mult_{G_2}(i,j)$,
can be roughly bounded from above by 
$4\max(M(i),1)$.

Further, take any $i\in I_{NS}(\tuple)\cap Y$, and observe that, setting $M_h(i)$ to be 
the number of multiedges of $G_h'$ incident to $i^\ell$ ($h=1,2$), the cardinality of the set
$$
\big\{j\leq n:\;\mult_{G_1}(i,j)\neq \mult_{G_2}(i,j)\big\}
$$
can be roughly estimated from above by $2\max(M_1(i),1)+2+2\max(M_2(i),1)$.

Summing up over all $i\in I_{NS}(\tuple)$, we get

\begin{align*}
&|\{(i,j):\;i\in I_{NS}(\tuple),\,(i,j)\mbox{ is an edge in precisely one of the graphs $G_1,G_2$}\}|\\
&\leq 4+4|I_{NS}(\tuple)|+2M_1+2M_2, 
\end{align*}
where $M_h$ denotes 
the number of multiedges incident to $I_{NS}(\tuple)$ in $G_h'$, $h=1,2$.
It remains to apply Lemma~\ref{l: 2-9572059827098}.
\end{proof}

\medskip

\begin{defi}
Assume $I_{NS}(\tuple)\neq\emptyset$.
We say that a partition $W_1\cup W_2\cup\dots\cup W_h$
of $I_{NS}(\tuple)$ is {\it $\tuple$--admissible} if all of the following conditions are satisfied:
\begin{itemize}

\item for any pair of left vertices $\tilde i,\hat i$ used by a simple switching
operation in a simple path leading from $G_1'$ to $G_1$, we have either $\{\tilde i,\hat i\}\subset W_u$
for some $u\leq h$ or $\{\tilde i,\hat i\}\cap I_{NS}(\tuple)=\emptyset$;

\item similarly, for any pair of left vertices $\tilde i,\hat i$ used by a simple switching
operation in the path from $G_2'$ to $G_2$, either $\{\tilde i,\hat i\}\subset W_u$
for some $u\leq h$ or $\{\tilde i,\hat i\}\cap I_{NS}(\tuple)=\emptyset$;

\item additionally, if $G_1'$ and $G_2'$ are adjacent and are connected by a type B switching then
for the left vertices \chng{$\{\tilde i,\hat i\}$} employed by the simple switching operation to transit between $G_1'$ and $G_2'$,
we have $\{\tilde i,\hat i\}\subset W_u$
for some $u\leq h$.

\end{itemize}
\end{defi}

First, the definition of $I_{NS}(\tuple)$ implies that $\{I_{NS}(\tuple)\}$ is a $\tuple$--admissible partition.
It is not difficult to see that there is a common refinement to all $\tuple$--admissible partitions
which is also $\tuple$--admissible, and which we will call {\it the canonical partition of $I_{NS}(\tuple)$}.
It will be convenient to fix an order of elements for a canonical partition: we will write
$I_{NS}(\tuple)=W_1\sqcup W_2\sqcup\dots\sqcup W_h$,
where $\min W_q<\min W_{q+1}$ for all $q<h$. To emphasize dependence on $\tuple$, we will occasionally use notation
$W_1(\tuple),W_2(\tuple),\dots$. Further, we will adopt a convention that $W_q(\tuple)=\emptyset$ whenever $I_{NS}(\tuple)$ is empty
or the index $q$ is larger than the cardinality of the canonical partition of $I_{NS}(\tuple)$, so that $I_{NS}(\tuple)$
can be associated with infinite sequence $(W_q(\tuple))_{q=1}^\infty$.

The canonical partition can be interpreted in the following way. Consider the collection $S$
of all simple switchings operating on left vertices from $I_{NS}(\tuple)$ and participating in a simple path
from $G_1'$ to $G_1$ (if any) or from $G_2'$ to $G_2$, together with the switching between $G_1'$ and $G_2'$ if it is of type B.
We then connect any two distinct switchings $\phi,\psi\in S$ by an edge whenever they operate on intersecting
$2$--subsets of left vertices. Then the canonical partition of $I_{NS}(\tuple)$ will correspond to partitioning the resulting
graph on $S$ into connected components.

\medskip

It will be crucial for us to have a description of the structure of the canonical partition.
Note that a delicate feature of $I_{NS}(\tuple)$ is that number and location of multiedges of $G_1'$
and $G_2'$ incident to $I_{NS}(\tuple)$ might differ in the case when the simple switching
connecting $G_1'$ and $G_2'$ operates on some multiedges (in particular, when $\tuple$ is of type $3$).

\begin{lemma}[Structure of canonical partition for types $1,2$]\label{l: 05610948704987}
Let $\tuple=(G_1,G_1',G_2,G_2')$ be a tuple of type $1$ or $2$.
Assume further that $I_{NS}(\tuple)$ is non-empty, and let
$(W_q)_{q=1}^\infty$ be its canonical partition. Fix any $u$ with $W_u\neq\emptyset$, and denote
by $M_u^{h}$ the 
number of multiedges of $G_h'$ incident to $W_u$, $h=1,2$.
Let $Y$ be the set of left vertices employed in the simple switching operation taking $G_1'$ to $G_2'$
(if $G_1'=G_2'$ then we set $Y=\emptyset$).
Then we have
\begin{itemize}

\item If $Y\cap W_u=\emptyset$ then necessarily $\max(|W_u|-2,1)\leq 2M_u^1-1=2M_u^2-1=M_u^1+M_u^2-1$;

\item If $Y\cap W_u\neq\emptyset$ then $Y\subset W_u$, and $|W_u|\leq 2+M_u^1+M_u^2$.

\end{itemize}

\chng{In particular, in either case we have $|W_u|\leq 2+M_u^1+M_u^2$.}
\end{lemma}
\begin{proof}
First, consider the case $Y\cap W_u=\emptyset$. Note that in this case the location 
of the multiedges incident to left vertices from $W_u$, is the same for both graphs $G_1'$ and $G_2'$,
so necessarily $M_u^1=M_u^2$. Further, for each $i\in W_u$, denote by $M_u(i)$ 
the number of multiedges of $G_1'$ (or $G_2'$) incident to $i$ (observe that $M_u(i)$ can only take values
$0$ or $1$). Taking into account the definition of a simple path and
of $I_{NS}(\tuple)$, each such left vertex $i$ is associated with a set $S(i,h)$ of left vertices not containing $i$,
with $S(i,h)\subset W_u$, where $S(i,h)$ are the vertices operated on by simple switchings involving
$i^\ell$ and participating in generation of the path leading from $G_h'$
to $G_h$, $h=1,2$. Note that, by the conditions on $\tuple$, the set $S(i,h)$ is either empty or has cardinality $1$;
$|S(i,h)|=M_u(i)$, $h=1,2$.
If $\hat I$ is the set of all left vertices in $W_u$ incident to multiedges then
$$
W_u=\bigcup_{i\in \hat I}\big(\{i\}\cup S(i,1)\cup S(i,2)\big).
$$
Now, a crucial observation is that we can introduce a total order ``$\precsim$'' on the vertices in $\hat I$ in such a way that
$$
\big(\{i\}\cup S(i,1)\cup S(i,2)\big)\cap \bigcup\limits_{j\in\hat I,\,j\prec i}\big(\{j\}\cup S(j,1)\cup S(j,2)\big)\neq \emptyset, \quad i\in
\hat I\setminus\min\limits_{\precsim}\hat I.
$$
The last assertion follows from the definition of the canonical partition as a common refinement of all $\tuple$--admissible
partitions of $I_{NS}(\tuple)$: if the condition was not satisfied then it would be possible to split $W_u$
into smaller subsets.
This property implies that
\begin{align*}
| W_u|=\Big|\bigcup_{i\in \hat I}\big(\{i\}\cup S(i,1)\cup S(i,2)\big)\Big|&\leq \sum_{i\in \hat I}\big|\big(\{i\}\cup S(i,1)\cup S(i,2)\big)\big|
-\big(|\hat I|-1\big)\\
&=\sum_{i\in \hat I}\big|\big(S(i,1)\cup S(i,2)\big)\big|+1\\
&\leq 2\sum_{i\in \hat I}M_u(i)+1
=2M^h_u+1,\quad h=1,2.
\end{align*}
Thus, $|W_u|\leq 2M^h_u+1$, $h=1,2$. Moreover, since $M_u^h\geq 1$, $h=1,2$ then necessarily $2M^h_u-1\geq 1$.
Therefore, we get $\max(|W_u|-2,1)\leq 2M^h_u-1$.

\medskip

Now, consider the case $Y\cap W_u\neq\emptyset$, when necessarily $Y\subset W_u$.
This case is more complicated since the number and location of multiedges incident to $W_u$ does not have to be the same
for $G_1'$ and $G_2'$. Denote by $\hat I_h$ the set of all left vertices in $W_u$ incident to multiedges of $G_h'$, $h=1,2$,
and define $S(i,h)$, $i\in \hat I_h$, as in the previous case. 
It will be convenient to set $S(i,h):=\emptyset$ for any $i\in \hat I_{3-h}\setminus \hat I_h$, $h=1,2$,
and define $S(i):=\{i\}\cup S(i,1)\cup S(i,2)$, $i\in \hat I_1\cup \hat I_2$.
We then have
$$
W_u=Y\cup\bigcup_{i\in \hat I_1\cup \hat I_2}\big(\{i\}\cup S(i,1)\cup S(i,2)\big)
=Y\cup \bigcup_{i\in \hat I_1\cup \hat I_2}S(i).
$$
Let us introduce a ``fake'' token $\circledast$ and set $S(\circledast):=Y$.
Then, again using the definition of the canonical partition, we can claim existence of a total order ``$\precsim$'' on
$\hat I:=\hat I_1\cup \hat I_2\sqcup\{\circledast\}$, so that
$$
S(i)\cap \bigcup\limits_{j\in\hat I,\,j\prec i}S(j)\neq \emptyset, \quad i\in
\hat I\setminus\min\limits_{\precsim}\hat I.
$$
Thus,
$$
|W_u|\leq \sum\limits_{i\in\hat I}|S(i)|-(|\hat I|-1)=2+\sum\limits_{i\in\hat I_1\cup\hat I_2}|S(i,1)\cup S(i,2)|
\leq 2+M_u^1+M_u^2.
$$
The result follows.
\end{proof}

\begin{lemma}[Structure of canonical partition for type $3$]
Let $\tuple=(G_1,G_1',G_2,G_2')$ be a tuple of type $3$.
Then $I_{NS}(\tuple)$ is of cardinality $4$, and its canonical partition consists of a single element
$W_1(\tuple)$.
\end{lemma}
\begin{proof}
The structure of $I_{NS}(\tuple)$ for type $3$ is visualized in Figure~\ref{fig: type3}.
\end{proof}

\subsection{Definition of a connection}

Let us start with statements which provide a path between two bipartite irregular graphs. 
The first statement will be used for graphs having at least $2$ left vertices; the second one -- with at least $4$ left vertices;
both satisfying some additional assumptions.
The statements will be used to estimate {\it complexity} (i.e.\ number of distinct realizations) of the set of
non-standard left vertices $I_{NS}$. A crucial observation is that presence of common neighbors between left vertices
reduces the complexity due to the fact that, in our setting, $d$ is less than a small power of $n$ whence
most pairs of left vertices are at distance greater than $2$. The statements below will provide certain
guarantees regarding existence of common neighbors.

\begin{lemma}\label{l: 2095203598305298}
Let $m\geq 2$, and let $\tilde G_1\neq\tilde G_2$ be two simple bipartite
graphs on $[m]\sqcup [n]$, such that 
the degree of every left vertex of each of the graphs is $d\geq 3$, and such that
$\deg_{\tilde G_1}(j^r)=\deg_{\tilde G_2}(j^r)$ for all $j\leq n$.
Assume further that the total number of pairs $(i,j)\in[m]\times[n]$ such that $(i,j)$ is an edge
in exactly one of the two graphs, is bounded above by $k$. Then there is a path $\tilde P$ (formed by simple switchings) 
on the set of all simple bipartite graphs on $[m]\sqcup[n]$
with left vertex degrees $d$ and right vertex degrees $(\deg_{\tilde G_1}(j^r))_{j=1}^n$, satisfying all of the following properties:
\begin{enumerate}
\item $\tilde P$ starts at $\tilde G_1$ and ends at $\tilde G_2$;
\item The length of $\tilde P$ is at most $k+2$;
\item Assume additionally that $m\geq 3$. Then
for every $t\leq\length{\tilde P}$, and the pair $\{i_1,i_2\}$ of left vertices of $\tilde P[t-1]$ operated on by the
simple switching transforming the graph to $\tilde P[t]$, at least one of the following should hold:
\begin{itemize}

\item $i_1$ and $i_2$ have no common neighbors in $\tilde P[t-1]$;

\item there is a left vertex $i\in[m]\setminus\{i_1,i_2\}$ having at least one common neighbor with
some vertex in $[m]\setminus\{i\}$ in $\tilde P[t-1]$.

\end{itemize}

\end{enumerate}

\end{lemma}
\begin{proof}
The case $m=2$, as well as the case when
all rows of $\Adj(\tilde G_1)$ are disjoint, are straightforward and can be checked directly,
so we will restrict our attention to the case when $m\geq 3$ {\bf and}
not all rows of $\Adj(\tilde G_1)$ are disjoint.

We will describe an algorithm which produces a path with the required properties.
Denote the set of simple bipartite graphs on $[m]\sqcup [n]$
with left vertex degrees $d$ and right vertex degrees $(\deg_{\tilde G_1}(j^r))_{j=1}^n$
by $M$.
For any graph $G$ in $M$, let $\disc(G)$ be the total number of pairs $(i,j)\in[m]\times[n]$ such that
$\mult_G(i,j)\neq\mult_{\tilde G_2}(i,j)$. The algorithm produces a sequence of graphs $G^0,G^1,\dots$
in two stages as follows.
We let $G^0:=\tilde G_1$
and set $s:=1$ (the length count).

\medskip

{\bf Stage 1.}

\medskip

{\bf Iteration step of Stage $1$.}
We are given a graph $G^{s-1}$. 
If
\begin{equation}\label{eq: 094620487-5832-09}
\begin{split}
&\mbox{there is a triple of distinct left vertices $\{i_1,i^{(2)},i^{(3)}\}$ such that $i^{(2)}$ and $i^{(3)}$ have}\\
&\mbox{a common neighbor in $G^{s-1}$ and}\;\;
\row_{i_1}(\Adj(G^{s-1}))\neq\row_{i_1}(\Adj(\tilde G_2))
\end{split}
\end{equation}
then we proceed as follows.
We choose right vertices $j_2\neq j_1$ such that $\mult_{G^{s-1}}(i_1,j_1)=1$ and $\mult_{G^{s-1}}(i_1,j_2)=0$
whereas $\mult_{\tilde G_2}(i_1,j_1)=0$ and $\mult_{\tilde G_2}(i_1,j_2)=1$ (existence of such
$j_1,j_2$ follows from the fact that $G^{s-1}$ and $\tilde G_2$ have the same left degree $d$).
Next, since $G^{s-1}$ and $\tilde G_2$
have the same right degrees sequence, there is necessarily an index $i_2\neq i_1$ such that $\mult_{G^{s-1}}(i_2,j_2)=1$ but
$\mult_{\tilde G_2}(i_2,j_2)=0$.
At this point we consider two subcases:

\begin{itemize}

\item {\bf Subcase (a):} $\mult_{G^{s-1}}(i_2,j_1)=0$. Then we produce $G^s$ from $G^{s-1}$ by applying the simple switching
on $(i_1,i_2,j_1,j_2)$. We then set $s:=s+1$ and go to the next iteration.

\item {\bf Subcase (b):} $\mult_{G^{s-1}}(i_2,j_1)=1$. In this subcase the 
switching operation on $(i_1,i_2,j_1,j_2)$
for the graph $G^{s-1}$ is impossible, and we have to invoke an additional argument.
Observe that in this situation there is necessarily an index $i_3\notin\{i_1,i_2\}$ such that
$\mult_{\tilde G_2}(i_3,j_1)=1$ but $\mult_{G^{s-1}}(i_3,j_1)=0$. Further, since both left vertices $i_3$ and $i_2$
have $d$ neighbors in $G^{s-1}$, there must exist an index $j_3\in[n]$ such that
$\mult_{G^{s-1}}(i_3,j_3)=1$ and $\mult_{G^{s-1}}(i_2,j_3)=0$.
Then we generate $G^s$ from $G^{s-1}$ by applying the simple switching on $(i_2,i_3,j_1,j_3)$,
and produce $G^{s+1}$ from $G^s$ via the simple switching on $(i_1,i_2,j_1,j_2)$ (see Figure~\ref{fig: two steps}).
We then set $s:=s+2$ and go to the next iteration.

\end{itemize}

\smallskip

If condition \eqref{eq: 094620487-5832-09} is not satisfied then we complete Stage $1$.

\medskip

As the output of Stage $1$, we obtain a graph $G^{s-1}$ which does not satisfy condition \eqref{eq: 094620487-5832-09}.
Termination of Stage $1$ is guaranteed by the following simple observation: whenever $G^{\tilde s}$ is produced from $G^{\tilde s-1}$
via subcase (a) then we have $\disc(G^{\tilde s})\leq \disc(G^{\tilde s-1})-2$, while in subcase (b),
we get $\disc(G^{\tilde s+1})\leq \disc(G^{\tilde s-1})-2$.
Observe also that at every transition step $G^{\tilde s-1}\to G^{\tilde s}$ of Stage $1$, there is a left vertex $i$ not operated
on by the simple switching leading from $G^{\tilde s-1}$ to $G^{\tilde s}$, which has a common neighbor with $[m]\setminus\{i\}$
in $G^{\tilde s-1}$. 
If $G^{s-1}=\tilde G_2$ then the path construction is complete. Otherwise, we go to Stage $2$ of the algorithm.

\medskip

{\bf Stage 2.}
We are given a graph $G^{s-1}$ which does not satisfy \eqref{eq: 094620487-5832-09}.
In view of our initial assumption (not all rows of $\Adj(\tilde G_1)$ are disjoint, hence
not all rows of $\Adj(G^{s-1})$ are disjoint), there is a pair of left vertices $i'\neq i''$ having
a common neighbor in $G^{s-1}$ such that
$$
\row_{i}(\Adj(G^{s-1}))=\row_{i}(\Adj(\tilde G_2))\quad\mbox{for all $i\notin\{i',i''\}$}.
$$
Note that $\row_{i'}(\Adj(G^{s-1}))\neq\row_{i'}(\Adj(\tilde G_2))$
and $\row_{i''}(\Adj(G^{s-1}))\neq\row_{i''}(\Adj(\tilde G_2))$.
Assume for a moment that there is a right vertex $j$ adjacent to some left vertex $i\notin\{i',i''\}$ in $G^{s-1}$,
together with some vertex in $\{i',i''\}$. For concreteness, we can suppose that $j$ is adjacent to $i$ and $i'$.
But since $\row_{i''}(\Adj(G^{s-1}))\neq\row_{i''}(\Adj(\tilde G_2))$, this would imply that \eqref{eq: 094620487-5832-09} is satisfied.
Thus, there is no such right vertex $j$.

%

Then necessarily for each $i\notin\{i',i''\}$,
$$
\supp\row_{i}(\Adj(G^{s-1}))\cap\big(\supp\row_{i'}(\Adj(G^{s-1}))\cup\supp\row_{i''}(\Adj(G^{s-1}))\big)=\emptyset,
$$
and there is a right vertex $\tilde j$ adjacent to $i',i''$ (and only $i',i''$) in $G^{s-1}$.
Pick arbitrary left vertex $i\notin \{i',i''\}$ and arbitrary $\hat j$ with $\mult_{G^{s-1}}(i,\hat j)=1$.
Then $\mult_{G^{s-1}}(i',\hat j)=0$, $\mult_{G^{s-1}}(i,\tilde j)=0$,
and so the simple switching on $(i,i',\hat j,\tilde j)$ is admissible. Let $G^s$ be the graph
obtained from $G^{s-1}$ via this switching. Now, consider a restriction of $G^s$ to the vertex set $\{i',i''\}\sqcup ([n]\setminus\{\tilde j,\hat j\})$,
and observe that both left and right degrees of $G^s[\{i',i''\}\sqcup ([n]\setminus\{\tilde j,\hat j\})]$
and $\tilde G_2[\{i',i''\}\sqcup ([n]\setminus\{\tilde j,\hat j\})]$ are the same.
We shall apply iteratively simple switching operations restricted to $\{i',i''\}\sqcup ([n]\setminus\{\tilde j,\hat j\})$
to obtain a graph $G^{\tilde s}$ from $G^s$ such that
$$
G^{\tilde s}[\{i',i''\}\sqcup ([n]\setminus\{\tilde j,\hat j\})]=
\tilde G_2[\{i',i''\}\sqcup ([n]\setminus\{\tilde j,\hat j\})],
$$
whereas
$$
\Adj(G^{\tilde s})_{uv}=
\Adj(G^{s})_{uv},\quad (u,v)\notin \{i',i''\}\times([n]\setminus\{\tilde j,\hat j\}).
$$
Since the transformed graphs have two left vertices, both of degrees $d-1$, the switching operations are straightforward;
moreover, in view of the above remarks for Stage $1$, we can guarantee that $\tilde s-1\leq k$.
Further, it is not difficult to see that the only four locations where the entries of $\Adj(G^{\tilde s})$
and $\Adj(\tilde G_2)$ do not agree are $(i,\hat j)$, $(i,\tilde j)$, $(i',\hat j)$, $(i',\tilde j)$.
Applying the simple switching operation on these entries, we obtain the graph $G^{\tilde s+1}=\tilde G_2$ (see Figure~\ref{fig: betacase}).
Thus, the total length of the path does not exceed $k+2$. Note also that in view of our construction procedure,
at every step there existed a left vertex not participating in the simple switching which has a common neighbor with
its complement in $[m]$.
\end{proof}

\begin{figure}[h]
\caption{Proof of Lemma~\ref{l: 2095203598305298}, Stage $1$, Subcase (b).}
\centering
\includegraphics[width=0.85\textwidth]{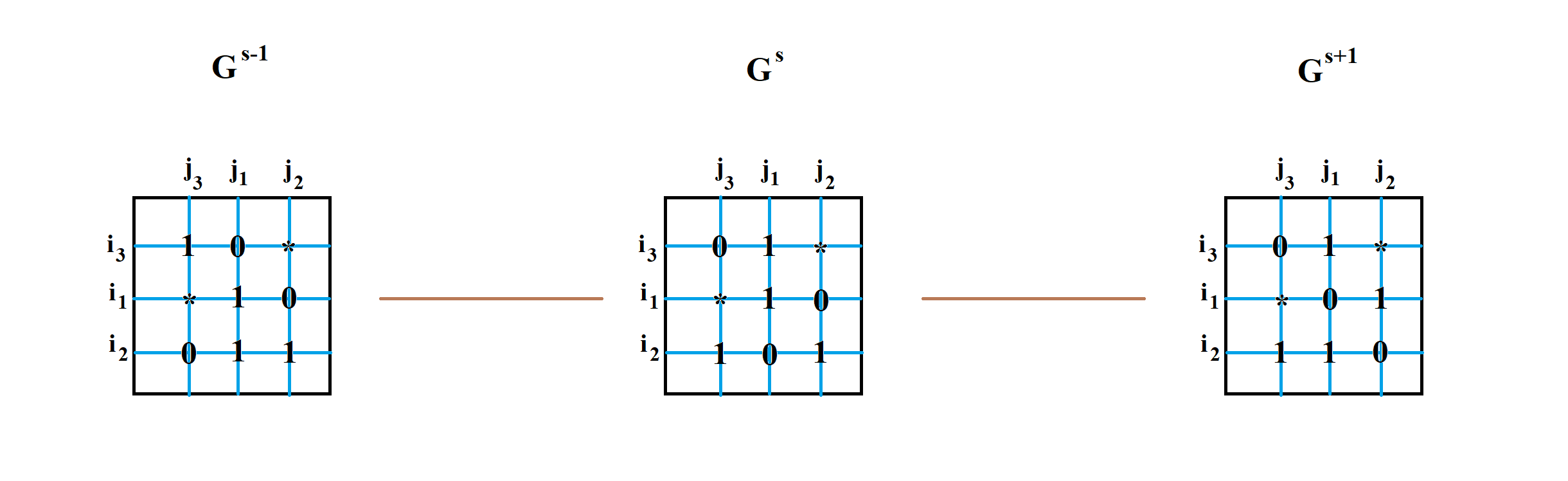}
\label{fig: two steps}
\end{figure}

\begin{figure}[h]
\caption{Proof of Lemma~\ref{l: 2095203598305298}, Stage $2$.
The graph $G^s$ is obtained from $G^{s-1}$ by ``relocating'' $1$ within the graphs' adjacency matrix
from position $(i',\tilde j)$ to $(i,\tilde j)$. Then, necessary switchings are performed on rows $i'$ and $i''$, leaving the
quadruple of elements at the intersection of rows $i',i''$ and columns $\tilde j,\hat j$ intact. As the
last step, the correct positions of elements in the square are ``restored''.}
\centering
\includegraphics[width=0.85\textwidth]{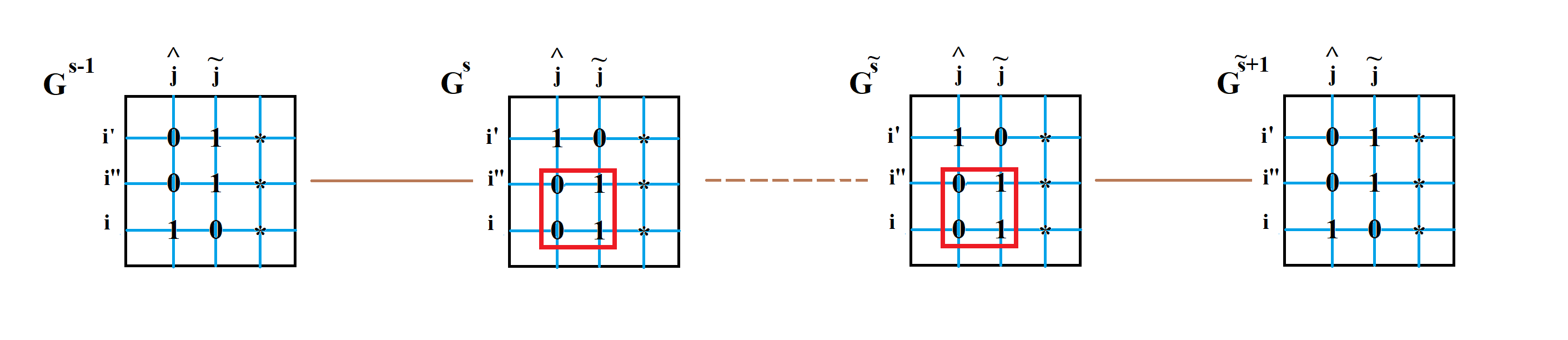}
\label{fig: betacase}
\end{figure}

\medskip

\begin{lemma}\label{l: 2095203598305298 bis}
Let $m\geq 4$, and let $\tilde G_1\neq\tilde G_2$ be two simple bipartite
graphs on $[m]\sqcup [n]$, such that 
the degree of every left vertex of each of the graphs is $d\geq 3$, and such that
$\deg_{\tilde G_1}(j^r)=\deg_{\tilde G_2}(j^r)$ for all $j\leq n$.
Assume that the total number of pairs $(i,j)\in[m]\times[n]$ such that $(i,j)$ is an edge
in exactly one of the two graphs, is bounded above by $k$.
Assume additionally that for each $h=1,2$, there is a couple of distinct left vertices
$i^{(1)}_h,i^{(2)}_h$ such that $i^{(1)}_h$ has a common neighbor with $[m]\setminus\{i^{(1)}_h,i^{(2)}_h\}$ in $\tilde G_h$,
and $i^{(2)}_h$ has a common neighbor with $[m]\setminus\{i^{(2)}_h\}$
in $\tilde G_h$.
Then there is a path $\tilde P$ (formed by simple switchings) 
on the set of all simple bipartite graphs on $[m]\sqcup[n]$
with left vertex degrees $d$ and right vertex degrees $(\deg_{\tilde G_1}(j^r))_{j=1}^n$, satisfying all of the following properties:
\begin{enumerate}
\item $\tilde P$ starts at $\tilde G_1$ and ends at $\tilde G_2$;
\item The length of $\tilde P$ is at most $C_{\text{\tiny\ref{l: 2095203598305298 bis}}}k$;

\item\label{prop: 08756104872-497} For every $t\leq\length{\tilde P}$, and the pair
$\{i_1,i_2\}$ of left vertices of $\tilde P[t-1]$ operated on by the
simple switching transforming the graph to $\tilde P[t]$, there exists a pair of distinct left vertices $\{i_t',i_t''\}$ disjoint from $\{i_1,i_2\}$
such that $i_t'$ has a common neighbor with $[m]\setminus\{i_t'\}$ and $i_t''$ has a common neighbor with
$[m]\setminus\{i_t',i_t''\}$ in the graph $\tilde P[t-1]$.

\end{enumerate}

Here, $C_{\text{\tiny\ref{l: 2095203598305298 bis}}}=3/2$.

\end{lemma}
\begin{proof}
The proof strategy is similar to that of Lemma~\ref{l: 2095203598305298} but involves more cases.
As in the previous proof, we denote the set of simple bipartite graphs on $[m]\sqcup [n]$
with left vertex degrees $d$ and right vertex degrees $(\deg_{\tilde G_1}(j^r))_{j=1}^n$
by $M$, and for any graph $G$ in $M$, let $\disc(G)$ be the total number of pairs $(i,j)\in[m]\times[n]$ such that
$\mult_G(i,j)\neq\mult_{\tilde G_2}(i,j)$.

We let $G^0:=\tilde G_1$
and $s:=1$.

\medskip

{\bf Iteration step.}
We are given a graph $G^{s-1}$ having the following property:
\begin{equation}\label{eq: 094620487-5832-09 bis}
\begin{split}
&\mbox{there is a couple of distinct left vertices
$i^{(1)}_s,i^{(2)}_s$ such that}\\
&\mbox{$i^{(1)}_s$ has a common neighbor with $[m]\setminus\{i^{(1)}_s,i^{(2)}_s\}$, and}\\
&\mbox{$i^{(2)}_s$ has a common neighbor with $[m]\setminus\{i^{(2)}_s\}$ in $G^{s-1}$.}
\end{split}
\end{equation}
If $G^{s-1}=\tilde G_2$ then we are done. Thus, assume that $G^{s-1}\neq\tilde G_2$.
Then there is a left vertex $i_1$ and two right vertices
$j_2\neq j_1$ such that $\mult_{G^{s-1}}(i_1,j_1)=1$ and $\mult_{G^{s-1}}(i_1,j_2)=0$
while $\mult_{\tilde G_2}(i_1,j_1)=0$ and $\mult_{\tilde G_2}(i_1,j_2)=1$.
Since $G^{s-1}$ and $\tilde G_2$
have the same right degrees sequence, there is necessarily a left index $i_2\neq i_1$ such that $\mult_{G^{s-1}}(i_2,j_2)=1$ but
$\mult_{\tilde G_2}(i_2,j_2)=0$.
We shall consider several subcases.

\begin{enumerate}

\item\label{item: 2098560498174987} Assume that $\mult_{G^{s-1}}(i_2,j_1)=0$.

\begin{enumerate}

\item\label{item: 98149872-49} Assume there are two distinct left vertices $i'_s,i''_s\in[m]\setminus\{i_1,i_2\}$ such that
$i_s''$ has a common neighbor with $[m]\setminus \{i_s',i_s''\}$ and $i_s'$ has a common neighbor
with $[m]\setminus \{i_s'\}$ in $G^{s-1}$. In this case, we perform the simple switching on
$(i_1,i_2,j_1,j_2)$ to produce the graph $G^s$.
Observe that the switching preserves the property of vertices $i_s',i_s''$ in the graph $G^s$
(i.e.\ $i_s''$ has a common neighbor with $[m]\setminus \{i_s',i_s''\}$ and $i_s'$ has a common neighbor
with $[m]\setminus \{i_s'\}$ in $G^{s}$),
and that $\disc(G^s)\leq \disc(G^{s-1})-2$. Set $s:=s+1$ and go to the next iteration.

\bigskip We observe that in all subcases below which assume that the assertion in \ref{item: 98149872-49} is not true,
there must exist a left vertex $i_3\notin\{i_1,i_2\}$ which has a common neighbor with $[m]\setminus\{i_3\}$,
such that no vertex $v$ in $[m]\setminus\{i_1,i_2,i_3\}$ has a common neighbor with $\{i_1,i_2\}$.

\bigskip

\item\label{item: 4296520970598} Assume the assertion in \ref{item: 98149872-49} is not true,
and the vertices $i_1$ and $i_2$ have a common neighbor $j_3$ in $G^{s-1}$. 
Pick any left vertex $i_4\in[m]\setminus\{i_1,i_2,i_3\}$. Then
necessarily $i_4$ and $\{i_1,i_2\}$ do not have common neighbors (because otherwise we fall into case \ref{item: 98149872-49}).
Let $j_4$ be any right vertex such that $\mult_{G^{s-1}}(i_4,j_4)=1$. Then necessarily $\mult_{G^{s-1}}(i_1,j_4)=0$.
Now, we produce the graph $G^s$ from $G^{s-1}$ with a switching on $(i_1,i_4,j_3,j_4)$,
then obtain $G^{s+1}$ from $G^s$ by a switching on $(i_1,i_2,j_1,j_2)$, and finally ``return'' the configuration of $i_4$--th neighborhood
by switching on $(i_1,i_4,j_4,j_3)$ to obtain $G^{s+2}$ from $G^{s+1}$.
It can be easily checked that $\disc(G^{s+2})\leq \disc(G^{s-1})-2$, and that the graph $G^{s+2}$ satisfies condition
\eqref{eq: 094620487-5832-09 bis}. Moreover,
for each transition $G^{s-1}\to G^s$, $G^s\to G^{s+1}$, $G^{s+1}\to G^{s+2}$,
the property \ref{prop: 08756104872-497} is satisfied (see Figure~\ref{fig: switching2}).
Set $s:=s+3$ and go to the next step.

\bigskip For future notice, let us explore what the negation of 
\ref{item: 98149872-49} and \ref{item: 4296520970598} implies.
Let $i^{(1)}_s,i^{(2)}_s$ be the couple of vertices from \eqref{eq: 094620487-5832-09 bis}.
Note that since we are not in subcase \ref{item: 98149872-49}, at least one of $i^{(1)}_s,i^{(2)}_s$ must belong to $\{i_1,i_2\}$,
that is, either $i_1$ or $i_2$ has a common neighbor with some vertex in $[m]\setminus\{i_1,i_2\}$
(at this moment, we applied a negation of \ref{item: 4296520970598}). This vertex in $[m]\setminus\{i_1,i_2\}$
having a common neighbor with $\{i_1,i_2\}$,
must be unique, and it must be $i_3$, because, again, otherwise we would be in Subcase~\ref{item: 98149872-49}.
No vertex in $[m]\setminus\{i_1,i_2,i_3\}$ has any common neighbors with its complement in $[m]$,
but, as \eqref{eq: 094620487-5832-09 bis} must still hold and $i_1$ and $i_2$ do not have common neighbors,
we discover the following necessary conditions on $G^{s-1}$: $i_3$ has common neighbors with both $i_1$ and $i_2$;
$i_1$ and $i_2$ do not have any common neighbors;
each vertex $v$ in $[m]\setminus\{i_1,i_2,i_3\}$ does not have any common neighbors with $[m]\setminus\{v\}$ in $G^{s-1}$.

\bigskip

\item\label{item: 0560957-984} Assume the assertions in \ref{item: 98149872-49} and \ref{item: 4296520970598} are not true,
and there is a vertex $u\in [m]\setminus\{i_1,i_2,i_3\}$ such that $\row_u(\Adj(G^{s-1}))\neq \row_u(\Adj(\tilde G_2))$.
Let $j'$ and $j''$ be two right vertices so that $\mult_{G^{s-1}}(u,j')=1$ and $\mult_{G^{s-1}}(u,j'')=0$
while $\mult_{\tilde G_2}(u,j')=0$ and $\mult_{\tilde G_2}(u,j'')=1$.
Since $G^{s-1}$ and $\tilde G_2$
have the same right degrees sequence, there is necessarily a left index $q\neq u$ such that $\mult_{G^{s-1}}(q,j'')=1$ but
$\mult_{\tilde G_2}(q,j'')=0$. Note also that in view of our assumption on $u$ we necessarily have
$\mult_{G^{s-1}}(q,j')=0$. Then we simply apply Subcase~\ref{item: 98149872-49} up to vertex relabelling,
namely we perform the simple switching on
$(u,q,j',j'')$ to produce graph $G^s$ from $G^{s-1}$. It is easy to see, using the above assumptions on $\{i_1,i_2,i_3\}$
that the required conditions on the switching and the graph $G^s$ are fulfilled.
Set $s:=s+1$ and go to the next iteration.

\bigskip Let us note that negation of \ref{item: 98149872-49},
\ref{item: 4296520970598} and \ref{item: 0560957-984} implies that the rows $\row_u(\Adj(G^{s-1}))$, $u\in [m]\setminus\{i_1,i_2,i_3\}$,
coincide with respective rows of $\Adj(\tilde G_2)$, and that each $u$ in that set does not have any common
neighbors in graph $\tilde G_2$. At this stage, we apply the crucial property specified in the lemma:
{\it ``there are
$i^{(1)}_2,i^{(2)}_2$ such that $i^{(1)}_2$ has a common neighbor with $[m]\setminus\{i^{(1)}_2,i^{(2)}_2\}$,
and $i^{(2)}_2$ has a common neighbor with $[m]\setminus\{i^{(2)}_2\}$
in $\tilde G_2$''.} It, together with the above, implies that each of the vertices $i_h$, $h=1,2,3$,
must have a common neighbor with $\{i_1,i_2,i_3\}\setminus\{i_h\}$ in the graph $\tilde G_2$
(and the degree of every right vertex adjacent to $\{i_1,i_2,i_3\}$ in $\tilde G_2$ is at most $2$).
We shall use this observation in the Subcases~\ref{item: 209870491410498} and~\ref{item: 298752098709} below.

\bigskip

\item\label{item: 087562-4=3204-487} Assume the assertions in \ref{item: 98149872-49},
\ref{item: 4296520970598} and \ref{item: 0560957-984} are not true, and both right vertices $j_1$ and $j_2$
have degree one (i.e.\ adjacent to only one left vertex, $i_1$ and $i_2$, respectively, in $G^{s-1}$). 
In this subcase, there exist two right vertices $j_4$ and $j_5$ not in $\{j_1,j_2\}$
such that $j_4$ is adjacent to $i_1$ and $i_3$, and $j_5$ is adjacent to $i_2$ and $i_3$.
We pick any left vertex $i_4$ and a right vertex $j_6$ such that $\mult_{G^{s-1}}(i_4,j_6)=1$
(note that necessarily $i_4$ does not have any common neighbors with $i_1$, $i_2$ or $i_3$).
The strategy here is very similar to Subcase~\ref{item: 4296520970598} and consists in applying three
switching operations in such a way that the required conditions are not violated by switchings and the graphs $G^s,G^{s+1},G^{s+2}$.
The three consecutive switchings are $(i_3,i_4,j_5,j_6)$, $(i_1,i_2,j_1,j_2)$, $(i_3,i_4,j_6,j_5)$ and the process is reflected in Figure~\ref{fig: switching2 00}.
We set $s:=s+3$ and go to the next step.

\item\label{item: 209560284760971} Assume the assertions in \ref{item: 98149872-49},
\ref{item: 4296520970598} and \ref{item: 0560957-984} are not true, and both $j_1$ and $j_2$ have degree $2$
in $G^{s-1}$. Pick a pair $(i_4,j_4)$ such that $i_4\in [m]\setminus\{i_1,i_2,i_3\}$,
and $\mult_{G^{s-1}}(i_4,j_4)=1$. The three consecutive switchings are $(i_3,i_4,j_2,j_4)$, $(i_1,i_2,j_1,j_2)$, $(i_3,i_4,j_4,j_2)$ 
and the  switching process is reflected in Figure~\ref{fig: switching2 11}.
We set $s:=s+3$ and go to the next step.

\item\label{item: 209870491410498} Assume the assertions in \ref{item: 98149872-49},
\ref{item: 4296520970598} and \ref{item: 0560957-984} are not true, $j_1$ has degree $1$ and $j_2$ has degree $2$
in $G^{s-1}$. 
According to the above remark, $i_2$ must have a common neighbor with $\{i_1,i_3\}$ in $\tilde G_2$,
and this common neighbor cannot be $j_2$ since $\mult_{\tilde G_2}(i_2,j_2)=0$. 
Thus, there must exist a right vertex $j_4$ of degree two which is adjacent to $i_2$
in $\tilde G_2$. In the graph $G^{s-1}$, $j_4$ can be either adjacent to both $i_2,i_3$ or be adjacent to both $i_1,i_3$.
In the former sub-subcase, we perform the three-step switching as illustrated in Figure~\ref{fig: switching2 01} (\chng{there, we use
an additional edge $(i_4,j_5)$ for the switching process to complete}).
In the latter sub-subcase, our goal shifts from switching $\langle i_1,i_2,j_1,j_2\rangle$
to switching $\langle i_2,i_1,j_2,j_4\rangle$, again using the three-step process, which is identical,
up to the vertex relabelling, to the one described in \ref{item: 209560284760971}. Observe that as a result of this procedure
we will decrease $\disc(\cdot)$ by at least two since under the assumption that $j_4$ is adjacent to both $i_1,i_3$,
we have $\mult_{G^{s-1}}(i_2,j_4)=0$ while $\mult_{\tilde G_2}(i_2,j_4)=1$.
We set $s:=s+3$ and go to the next step.

\item\label{item: 298752098709} Assume the assertions in \ref{item: 98149872-49},
\ref{item: 4296520970598} and \ref{item: 0560957-984} are not true, $j_1$ has degree $2$ and $j_2$ has degree $1$
in $G^{s-1}$. This subcase is identical to the previous one, up to vertex relabelling.

\end{enumerate}

\item Assume that $\mult_{G^{s-1}}(i_2,j_1)=1$.
In this case there is necessarily an index $i_3\notin\{i_1,i_2\}$ such that
$\mult_{\tilde G_2}(i_3,j_1)=1$ but $\mult_{G^{s-1}}(i_3,j_1)=0$. Further, since both left vertices $i_3$ and $i_2$
have $d$ neighbors in $G^{s-1}$, there must exist an index $j_3\in[n]$ such that
$\mult_{G^{s-1}}(i_3,j_3)=1$ and $\mult_{G^{s-1}}(i_2,j_3)=0$.
Again, we consider two subcases.

\begin{enumerate}

\item Assume that there is a vertex $i_4\in[m]\setminus\{i_1,i_2,i_3\}$ which has a common neighbor with
$[m]\setminus \{i_4\}$ in $G^{s-1}$. In this case we produce $G^s$ from $G^{s-1}$ by a simple switching on $(i_2,i_3,j_1,j_3)$
and $G^{s+1}$ from $G^s$ by a switching on $(i_1,i_2,j_1,j_2)$.
It is not difficult to check that $\disc(G^{s+1})\leq \disc(G^{s-1})-2$ (see also Lemma~\ref{l: 2095203598305298}, Subcase (b)).
Moreover, the graph $G^{s+1}$ satisfies condition
\eqref{eq: 094620487-5832-09 bis} with the two indices $i_3$ and $i_4$, and 
for each transition $G^{s-1}\to G^s$, $G^s\to G^{s+1}$,
the property \ref{prop: 08756104872-497} is satisfied. We set $s:=s+2$ and go to the next iteration.

\item Assume that there is {\bf no} left vertex $i_4\in[m]\setminus\{i_1,i_2,i_3\}$ which has a common neighbor with
$[m]\setminus \{i_4\}$ in $G^{s-1}$. 
Note that there is a right vertex $j_4$ such that $\mult_{G^{s-1}}(i_2,j_4)=0$ but
$\mult_{\tilde G_2}(i_2,j_4)=1$ (the vertex may coincide with $j_3$ or may be different from $j_3$).
Accordingly, there is a left vertex $i_5$ such that $\mult_{G^{s-1}}(i_5,j_4)=1$ but
$\mult_{\tilde G_2}(i_5,j_4)=0$. 
If $\mult_{G^{s-1}}(i_5,j_2)=0$ then we find ourselves in the setting of Case~\ref{item: 2098560498174987},
up to the vertex relabelling, i.e.\ we apply Case~\ref{item: 2098560498174987} setting
to the quadruple $(i_2,i_5,j_2,j_4)$, and go to the next iteration step.
Otherwise, if $\mult_{G^{s-1}}(i_5,j_2)=1$ then necessarily $i_5=i_3$ because
$[m]\setminus\{i_1,i_2,i_3\}$ do not have any common neighbors with $\{i_1,i_2,i_3\}$ in $G^{s-1}$.
We then essentially apply the three-step switching described in Subcase~\ref{item: 209560284760971},
where $i_2$ and $i_3$ swap roles: see Figure~\ref{fig: switching2 11}.
Set $s:=s+3$ and go to the next iteration step.

\end{enumerate}

\end{enumerate}

Note that at every iteration step the value of $\disc(\cdot)$ decreases at least by $2$, and that
each iteration step adds at most $3$ graphs $G^{\cdot}$ to the sequence.
Thus, the total sequence length is bounded above by $3k/2$.

\end{proof}

\begin{figure}[h]
\caption{Proof of Lemma~\ref{l: 2095203598305298 bis}, Subcase~\ref{item: 4296520970598}.}
\centering
\includegraphics[width=0.95\textwidth]{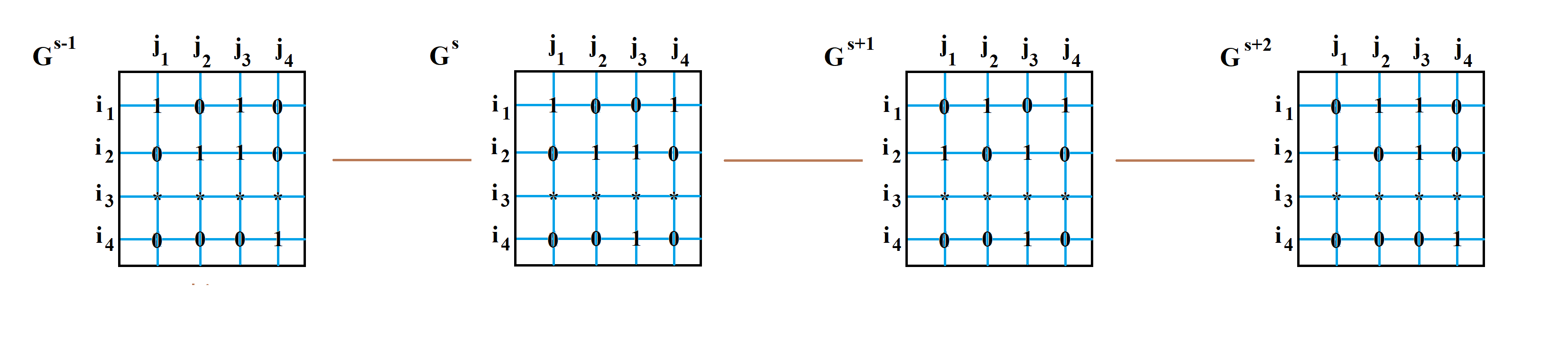}
\label{fig: switching2}
\end{figure}

\begin{figure}[h]
\caption{Proof of Lemma~\ref{l: 2095203598305298 bis}, Subcase~\ref{item: 087562-4=3204-487}.}
\centering
\includegraphics[width=0.95\textwidth]{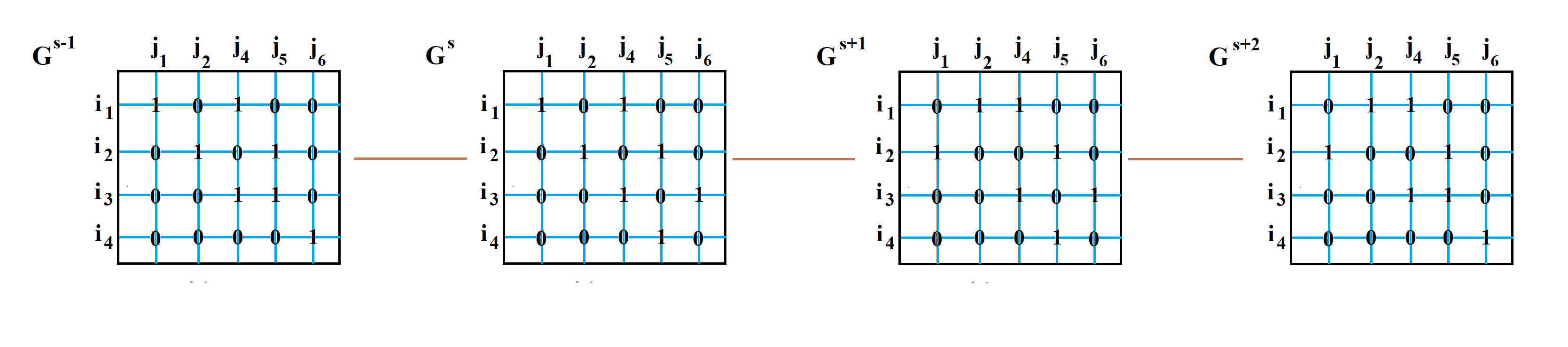}
\label{fig: switching2 00}
\end{figure}

\begin{figure}[h]
\caption{Proof of Lemma~\ref{l: 2095203598305298 bis}, Subcase~\ref{item: 209560284760971}.}
\centering
\includegraphics[width=0.95\textwidth]{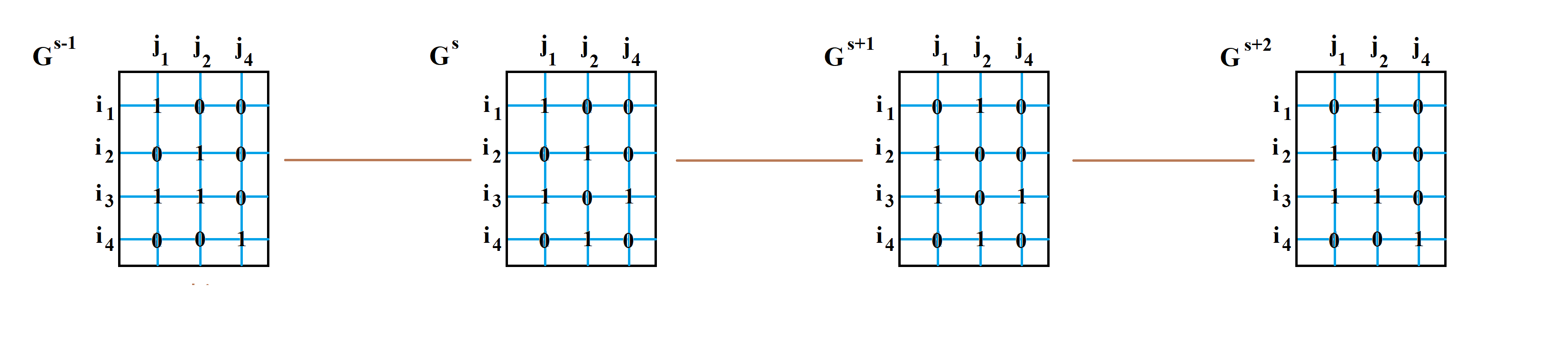}
\label{fig: switching2 11}
\end{figure}

\begin{figure}[h]
\caption{Proof of Lemma~\ref{l: 2095203598305298 bis}, an illustration to Subcase~\ref{item: 209870491410498}.}
\centering
\includegraphics[width=0.95\textwidth]{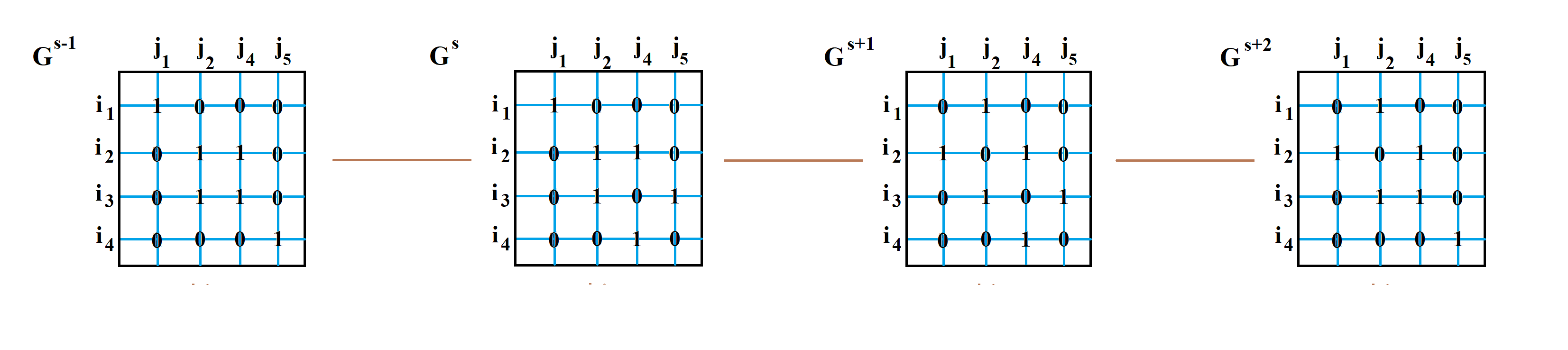}
\label{fig: switching2 01}
\end{figure}

\medskip

Now, we are in position to construct a canonical path between the simple graphs.
Let $\tuple=(G_1,G_1',G_2,G_2')$ be an admissible tuple.
\begin{defi}
A {\it connection} (with respect to $\tuple$) is a path $P\subset \BipGSet_n(d)$ starting at $G_1$ and ending at $G_2$ constructed as follows:
\begin{itemize}

\item We set $P[0]:=G_1$.

\item For each integer $1\leq t\leq 2|E_S(\tuple)|$, the simple graph $P[t]$ is obtained from $P[t-1]$ as follows.
Let $i$ be the $\lceil t/2\rceil$--th smallest index with $(i,j)\in E_S(\tuple)$ for some $j$.
Denote $i_1:=\ione(i,\tuple)$, $i_2:=\itwo(i,\tuple)$, $j_1:=\jone(i,\tuple)$, $j_2:=\jtwo(i,\tuple)$.
Then, if $t$ is odd, $P[t]$ is obtained by switching the edges $(i_1,j)$ and $(i_2,j_2)$ of $P[t-1]$.
Otherwise, if $t$ is even, $P[t]$ is obtained from $P[t-1]$ by switching the edges $(i_1,j_2)$ and $(i,j_1)$
(see Figure~\ref{fig: m-standard}). Once $P[2|E_S(\tuple)|]$ is constructed, set $t_0:=2|E_S(\tuple)|$ and go to the next stage of
the construction.

\item If $G_1'$ and $G_2'$ are adjacent and the switching transforming $G_1'$ into $G_2'$
(say, performed on edges $(v,w)$ and $(\tilde v,\tilde w)$) is of type A, we define
$P[2|E_S(\tuple)|+1]$ by performing a simple switching on the edges
$(v,w)$ and $(\tilde v,\tilde w)$ of $P[2|E_S(\tuple)|]$, and set $t_0:=2|E_S(\tuple)|+1$.

\item
If $I_{NS}(\tuple)$ is empty, the above construction gives the graph $G_2=P[t_0]$. Otherwise,
if $I_{NS}(\tuple)$ is non-empty, we produce $G_2$ from $P[t_0]$ as follows.
First, take the canonical partition $W_1\sqcup W_2\sqcup\dots\sqcup W_h$ of $I_{NS}(\tuple)$,
set $G^{(0)}:=P[t_0]$,
and for each $1\leq u\leq h$ let $G^{(u)}$ be the graph on $[n]\sqcup[n]$ whose adjacency matrix satisfies
$\row_i(\Adj(G^{(u)}))=\row_i(\Adj(G^{(u-1)}))$, $i\in [n]\setminus W_u$, and
$\row_i(\Adj(G^{(u)}))=\row_i(\Adj(G_2))$, $i\in W_u$.
For each $u\leq h$, we apply
\begin{itemize}
\item in the case when $|W_u|\in\{2,3,4\}$ --- Lemma~\ref{l: 2095203598305298};
\item in the case $|W_u|\geq 5$ and $W_u$ comprises type B switching between $G_1'$ and $G_2'$ --- Lemma~\ref{l: 2095203598305298};
\item in the case $|W_u|\geq 5$ and $W_u$ does not comprise type B switching --- Lemma~\ref{l: 2095203598305298 bis}
\end{itemize}
to subgraphs
$G^{(u-1)}[W_u\sqcup[n]]$ and $G^{(u)}[W_u\sqcup[n]]$
to construct a path $\tilde P^{(u)}$
connecting $G^{(u-1)}$ and $G^{(u)}$ and satisfying all the properties listed in the lemmas.
Then we complete the construction of $P$ by sequentially connecting $P[t_0]$
with $G^{(1)}$, $G^{(2)},\dots$, $G^{(h)}=G_2$ using the paths $\tilde P^{(u)}$.

\end{itemize}
\end{defi}

It can be shown that the construction described above is valid, i.e.\ the prescribed simple switching operations at each stage
can be performed, and that the ending vertex of $P$ is indeed $G_2$.
In particular, in the case when $|W_u|\geq 5$ and $W_u$ does not comprise type B switching, the above subgraphs
$G^{(u-1)}[W_u\sqcup[n]]$ and $G^{(u)}[W_u\sqcup[n]]$ do satisfy the conditions of Lemma~\ref{l: 2095203598305298 bis}:
\begin{lemma}
Let $G^{(u-1)},G^{(u)}$ be graphs from the construction procedure above (for some $u$), and assume that
$|W_u|\geq 5$ and $W_u$ does not comprise type B switching. Then the graphs
$\tilde G_1:=G^{(u-1)}[W_u\sqcup[n]]$ and $\tilde G_2:=G^{(u)}[W_u\sqcup[n]]$
satisfy conditions of Lemma~\ref{l: 2095203598305298 bis} with $k\leq 4|W_u|$.
\end{lemma}
\begin{proof}
Obviously, both $\tilde G_1$ and $\tilde G_2$ have the same right degrees sequences because
for every switching within the connection from $G_1$ to $G_2$, either both of the left vertices it operates
on are contained in $W_u$ or none. The condition on $k$ comes from the next simple observation:
any left vertex in $W_u$ can be used by at most one simple switching within the simple path from $G_1'$ to $G_1$,
and by at most one switching within the simple path from $G_2'$ to $G_2$.
Thus, within the corresponding rows of the adjacency matrices of $\tilde G_1$ and $\tilde G_2$, at most four elements can differ.

It remains to check the crucial condition on existence of common neighbors.
The assumption that $|W_u|\geq 5$, implies that there are at least two distinct left vertices $v',v''\in W_u$
and right vertices $z',z''$ such that $(v',z')$ and $(v'',z'')$ are \chng{both} multiedges in $G_1'$.
Further, by the definition of a simple path, the corresponding simple switching operations $\langle v',g_1',z',x_1'\rangle$
and $\langle v'',g_1'',z'',x_1''\rangle$ within the simple path leading from $G_1'$ to $G_1$, must satisfy $g_1'\neq g_1''$
and $g_1',g_1''\notin\{v',v''\}$. In particular, $v'$ must have a common neighbor with $W_u\setminus \{v',v''\}$
and $v''$ must have a common neighbor with $W_u\setminus \{v''\}$ in graph $G_1$, hence in graph $\tilde G_1$.
The same argument works for $\tilde G_2$.
\end{proof}

The main steps of the
construction procedure are reflected in Figure~\ref{fig: connection}.

Note that the path described above may not be uniquely defined whenever $I_{NS}(\tuple)$ is non-empty
(because Lemmas~\ref{l: 2095203598305298} and~\ref{l: 2095203598305298 bis} confirm existence
but not uniqueness of respective subpaths).
In this case we fix any admissible path for the given $4$--tuple $\tuple$ satisfying the above conditions.
Often, to emphasize dependence on the $4$--tuple, we will denote the connection for $\tuple$
by $P_\tuple$.

\begin{rem}\label{rem: 02985720958709}
The length of $P_\tuple$ is always at least $2|E_S(\tuple)|$, and is equal to $2|E_S(\tuple)|$ when $G_1'=G_2'$
and $I_{NS}(\tuple)$ is empty. In general, the length of $P_\tuple$ can be bounded from above by
$2|E_S(\tuple)|+C_{\text{\tiny\ref{l: 2095203598305298 bis}}}(8|I_{NS}(\tuple)|+4)$,
in view of Lemmas~\ref{l: 2095203598305298},~\ref{l: 2095203598305298 bis}, and~\ref{l: 208762058709283}.
\end{rem}

\begin{figure}[h]
\caption{A schematic depiction of a connection for a tuple $\tuple=(G_1,G_1',G_2,G_2')$
in the case when the graphs $G_1'$ and $G_2'$ are connected by a type $A$ switching.
The graphs are represented by their adjacency matrices.
The first $2|E_S(\tuple)|$ steps of the connection deal with rows associated with $m$--standard edges (shown in light blue).
The graphs $P_\tuple[2|E_S(\tuple)|]$ and $P_\tuple[2|E_S(\tuple)|+1]$ are related by a switching of type $A$ (corresponding
couple of rows is colored in dark blue and are shown as ``adjacent'' for simplicity).
The last $\length{P_\tuple}-2|E_S(\tuple)|-1$ steps deal with the set $I_{NS}(\tuple)$ which is in turn represented via
the canonical partition (grey).
The uncolored rows of adjacency matrices coincide with those of $G_2$.}
\centering
\includegraphics[width=0.95\textwidth]{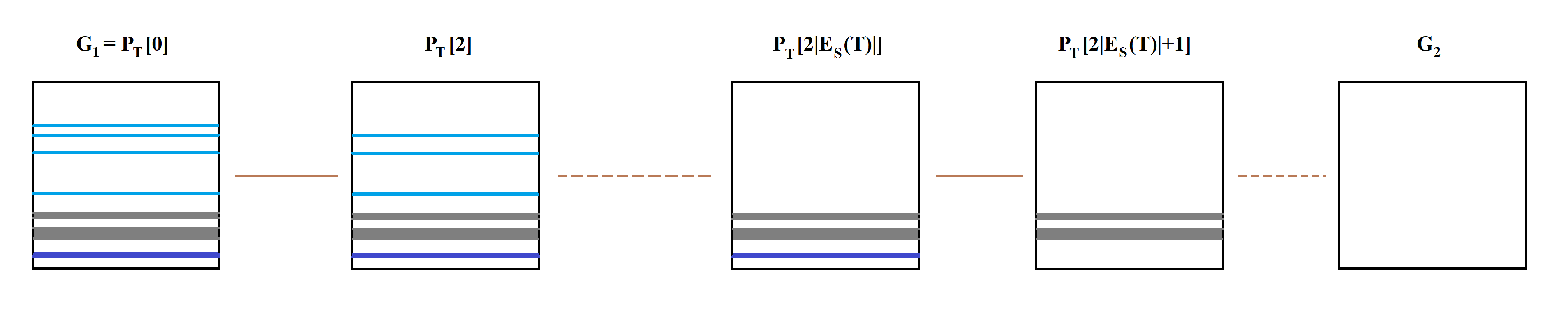}
\label{fig: connection}
\end{figure}

\begin{defi}
For each $t\leq \length{P_\tuple}$, we say that the couple $(P_\tuple[t-1],P_\tuple[t])$ is
\begin{itemize}

\item {\it $m$--standard} if $t\leq 2|E_S(\tuple)|$;

\item {\it of type A} if $G_1'$ and $G_2'$ are connected by a type $A$ switching, and that is the simple switching operation taking $P_\tuple[t-1]$
to $P_\tuple[t]$ (note that necessarily $t=2|E_S(\tuple)|+1$ in this case);

%
%

\item {\it $NS$--couple} if the left vertices participating in the switching from $P_\tuple[t-1]$
to $P_\tuple[t]$, belong to $I_{NS}(\tuple)$.

\end{itemize}
\end{defi}

\subsection{Complexity of $I_{NS}(\tuple)$}

The goal of this subsection is to count the number of possible distinct realizations of the set $I_{NS}(\tuple)$
such that the corresponding connection $P_\tuple$ contains a given pair of adjacent simple graphs $(\Gamma,\tilde\Gamma)$.
This estimate will then be used to bound the total number of tuples $\tuple$ with $P_\tuple$ containing the couple $(\Gamma,\tilde\Gamma)$.
We start with two auxiliary lemmas.

\begin{lemma}\label{l: 98572-987-49}
Let $\tuple=(G_1,G_1',G_2,G_2')$ be an admissible $4$--tuple, and let $W_q\neq\emptyset$
be the $q$--th element of the canonical partition of $I_{NS}(\tuple)$.
Then necessarily there is a right vertex $j\in[n]$ having at least two neighbors in $W_q$ in each of the graphs $G_1$ and $G_2$.
\end{lemma}
\begin{proof}
Note that, by the definition of the canonical partition, the path $P$ from $G_1$ to $G_2$ in $\MultBipGSet_n(d)$
obtained by concatenation of simple paths from $G_1'$ to $G_1$ and $G_2'$ to $G_2$, and of the simple switching connecting
$G_1'$ and $G_2'$ (if any), has the property that for any $t\leq \length{P}$, the simple switching from $P[t-1]$
to $P[t]$ operates on left vertices which either are both contained in $W_q$ or are both {\it not} contained in $W_q$.
Thus, necessarily $G_1[W_q\sqcup[n]]$ has the same right degrees sequence as $G_2[W_q\sqcup[n]]$,
so it is enough to check the statement for one of the graphs $G_1$ or $G_2$. But then the result follows
immediately since each $W_q$ must contain at least one left vertex incident to a multiedge either in $G_1'$ or $G_2'$,
so the corresponding right vertex $j$ has the required properties.
\end{proof}

\begin{lemma}\label{l: 29570295873059873}
Let $\tuple=(G_1,G_1',G_2,G_2')$ be admissible, and let $t\geq 1$ be such that
$(P_\tuple[t-1],P_\tuple[t])$ is $NS$--couple in $P_\tuple$.
Let $W_q=W_q(\tuple)$ be the element of the canonical partition of $I_{NS}(\tuple)$ which contains left vertices $\{i_1,i_2\}$
operated on by the simple switching
leading from $P_\tuple[t-1]$ to $P_\tuple[t]$, and assume that $|W_q|\geq 3$.
Then necessarily there is a left vertex
$i_3\in W_q\setminus\{i_1,i_2\}$ which has a common neighbor with $W_q\setminus\{i_3\}$ in $P_\tuple[t-1]$.
\end{lemma}
\begin{proof}
By the definition of a connection, the part of the path $P_\tuple$ which transforms the subgraph of $G_1$ on $W_q\sqcup[n]$
into the respective subgraph of $G_2$, satisfies all the conditions listed in Lemma~\ref{l: 2095203598305298}
or~\ref{l: 2095203598305298 bis}.
Specifically, at least one of the following must be true:
\begin{itemize}

\item $i_1$ and $i_2$ have no common neighbors in $P_\tuple[t-1]$;

\item there is a left vertex $i_3\in W_q\setminus\{i_1,i_2\}$ having at least one common neighbor with $W_q\setminus\{i_3\}$
in $P_\tuple[t-1]$.
\end{itemize}
Note that, in view of Lemma~\ref{l: 98572-987-49}, there is a right vertex $j$ having at least two neighbors in $W_q$
in both graphs $G_1$ and $G_2$, hence in graphs $P_\tuple[t-1]$ and $P_\tuple[t]$ as well.

In the first case,
since $i_1$ and $i_2$ have no common neighbors, $j$ must be adjacent to at least one left vertex $i_3\notin\{i_1,i_2\}$.
This left vertex satisfies the necessary conditions.

In the second case we have the required property automatically. The result follows.
\end{proof}

\medskip

In the following three lemmas, we consider the problem of estimating the number of possible realizations
for a given element of the canonical partition of $I_{NS}(\tuple)$.

\medskip

In the first two lemmas, we study the complexity of the sets $W_q(\tuple)$ under the assumption that
the switching connecting $\Gamma$ and $\tilde\Gamma$ is not ``contained'' in $W_q(\tuple)$.

\begin{lemma}\label{l: 295872095875098}
Let $(\Gamma,\tilde \Gamma)$ be a pair of adjacent graphs in $\BipGSet_n(d)$,
and $q,r_q\in\N$. Let $T$ be the collection of all admissible $4$--tuples 
$\tuple=(G_1,G_1',G_2,G_2')$ such that
\begin{itemize}
\item $P_\tuple$ contains the pair $(\Gamma,\tilde \Gamma)$, and 
\item the left vertices operated on by the simple switching leading from $\Gamma$ to
$\tilde\Gamma$ do not belong to $W_q(\tuple)$, and
\item $|W_q(\tuple)|=r_q$, and
\item $G_1'=G_2'$ or the simple switching leading from $G_1'$ to $G_2'$ operates
on left vertices that do not belong to $W_q(\tuple)$.
\end{itemize}
Then the cardinality of the set
\begin{align*}
{\bf W}:=\big\{W\subset[n]:\;\mbox{$W=W_q(\tuple)$ for some $\tuple\in T$}\big\}
\end{align*}
is bounded above by $4n^{\max(1,r_q-2)}\big(\max(1,r_q-2)d(d-1)\big)^2$.
\end{lemma}
\begin{proof}
Note that, by the conditions on the couple $(\Gamma,\tilde \Gamma)$,
for every $\tuple\in T$, either
$$\Gamma[W_q(\tuple)\sqcup [n]]=\tilde \Gamma[W_q(\tuple)\sqcup [n]]
=G_1[W_q(\tuple)\sqcup [n]]$$ or 
$$\Gamma[W_q(\tuple)\sqcup [n]]=\tilde \Gamma[W_q(\tuple)\sqcup [n]]
=G_2[W_q(\tuple)\sqcup [n]].$$
Denote the subset of $T$ of tuples satisfying the former condition by $T_1$,
and the latter condition --- by $T_2$.

For any $\tuple\in T$, denote by $M_q(\tuple)$ 
the number of multiedges of $G_h'$ incident to $W_q(\tuple)$, $h=1,2$ (observe that the quantity is the same for both graphs $G_1'$
and $G_2'$). 
Denote by $T'$ the subset of all $\tuple\in T$ with $M_q(\tuple)>1$.

Now, let us consider several cases:
\begin{itemize}
\item {\it Treatment of $T'\cap T_1$.}
For any $\tuple\in T'\cap T_1$ we observe that, in view of the definition of a simple path and sets $T',T_1$,
$W_q(\tuple)$ contains at least two left vertices from $I_s(G_1,G_1')$
which have common neighbors with $W_q(\tuple)\cap I_m(G_1,G_1')$ in $\Gamma$. 
Thus, the set $W_q(\tuple)$ for $\tuple\in T'\cap T_1$ can be determined first by choosing $r_q-2$ left vertices from $[n]$
and then choosing two more vertices having common neighbors with the given $(r_q-2)$--set.
Clearly, the number of choices for the latter is bounded above by $((r_q-2)d(d-1))^2$, whence
$$
\big|\big\{W_q(\tuple):\;\tuple\in T'\cap T_1\big\}\big|\leq n^{r_q-2}((r_q-2)d(d-1))^2.
$$
\item {\it Treatment of $T'\cap T_2$.}
Similarly, we get
$$
\big|\big\{W_q(\tuple):\;\tuple\in T'\cap T_2\big\}\big|\leq n^{r_q-2}((r_q-2)d(d-1))^2.
$$
\item {\it Treatment of $T\setminus T'$.}
Now, consider the case when $\tuple\in T\setminus T'$. Note that in this case the set $W_q(\tuple)$
can have cardinality either $2$ or $3$. In the former case, we necessarily have
$W_q(\tuple)\cap I_s(G_1,G_1')=W_q(\tuple)\cap I_s(G_2,G_2')$, and the unique left vertex contained in this set
must have at least one common neighbor with $W_q(\tuple)\cap I_m(G_1,G_1')=W_q(\tuple)\cap I_m(G_2,G_2')$.

In the latter case, the only reason why the multiedge of $G_1'$ (and $G_2'$) incident to
$W_q(\tuple)\cap I_m(G_1,G_1')=W_q(\tuple)\cap I_m(G_2,G_2')$ is not $m$--standard is a violation of
property~\eqref{p: 13981749827} in the definition of $m$--standard edges. It is not difficult to see that
without this property, there must exist a left vertex $i\in W_q(\tuple)$ having common neighbors
with {\it both} remaining vertices of $W_q(\tuple)\setminus\{i\}$ in $\Gamma$
(see Figures~\ref{fig: non-standard}).

Taking into account these observations we obtain that
$$
\big|\big\{W_q(\tuple):\;\tuple\in T\setminus T'\big\}\big|\leq
nd(d-1)+nd(d-1)^3
\leq 2n (d(d-1))^2.
$$
\end{itemize}
The result follows by combining the estimates.
\end{proof}

\begin{lemma}\label{l: -298570596098709}
Let $(\Gamma,\tilde \Gamma)$ be a pair of adjacent graphs in $\BipGSet_n(d)$,
and $q,r_q\in\N$. Let $T$ be the collection of all admissible $4$--tuples 
$\tuple=(G_1,G_1',G_2,G_2')$ such that
\begin{itemize}
\item $P_\tuple$ contains the pair $(\Gamma,\tilde \Gamma)$, and 
\item the left vertices operated on by the simple switching leading from $\Gamma$ to
$\tilde\Gamma$ are not contained in $W_q(\tuple)$, and
\item $|W_q(\tuple)|=r_q$, and
\item the simple switching leading from $G_1'$ to $G_2'$ operates
on left vertices contained in $W_q(\tuple)$.
\end{itemize}
Then the cardinality of the set
\begin{align*}
{\bf W}:=\big\{W\subset[n]:\;\mbox{$W=W_q(\tuple)$ for some $\tuple\in T$}\big\}
\end{align*}
is bounded above by $n^{r_q-1}(r_q-1)d(d-1)$.
\end{lemma}
\begin{proof}
Applying Lemma~\ref{l: 98572-987-49}, we immediately get that
for any $\tuple=(G_1,G_1',G_2,G_2')\in T$, there is right vertex $j=j(\tuple)$
which has at least two neighbors in $W_q(\tuple)$ in each of the two graphs $G_1$, $G_2$, hence in graphs $\Gamma$ and $\tilde \Gamma$
as well. Thus, each set from ${\bf W}$ can be identified first by choosing $(r_q-1)$--subset of $[n]$ and then
choosing a left vertex having at least one common neighbor with the given subset. Thus,
$$
|{\bf W}|\leq n^{r_q-1}(r_q-1)d(d-1).
$$
\end{proof}

\medskip

Now, we consider the situation when the simple switching connecting $\Gamma$ and $\tilde \Gamma$ is ``contained''
in $W_q(\tuple)$.
\begin{lemma}\label{l: 985709575987}
Let $(\Gamma,\tilde \Gamma)$ be a pair of adjacent graphs in $\BipGSet_n(d)$,
and $q,r_q\in\N$. Let $T$ be the collection of all admissible $4$--tuples 
$\tuple=(G_1,G_1',G_2,G_2')$ such that
\begin{itemize}
\item $(\Gamma,\tilde \Gamma)$ is a $NS$--couple in the connection $P_\tuple$,
and
\item the left vertices operated on by the simple switching leading from $\Gamma$ to
$\tilde\Gamma$ are contained in $W_q(\tuple)$, and
\item $|W_q(\tuple)|=r_q$.
\end{itemize}
Then the cardinality of the set
\begin{align*}
{\bf W}:=\big\{W\subset[n]:\;\mbox{$W=W_q(\tuple)$ for some $\tuple\in T$}\big\}
\end{align*}
is bounded above by $\max\big(1,n^{r_q-3}(r_q-1)d(d-1)\big)$.

Furthermore, assume that $r_q\geq 5$, and denote by $T'$ the subset of tuples $\tuple=(G_1,G_1',G_2,G_2')\in T$
such that either $G_1'=G_2'$ or the switching connecting $G_1'$ and $G_2'$ does not operate on 
vertices in $W_q(\tuple)$.
Then the cardinality of the set
$$
{\bf W'}:=\big\{W\subset[n]:\;\mbox{$W=W_q(\tuple)$ for some $\tuple\in T'$}\big\}
$$
can be bounded above by $n^{r_q-4}(r_q-1)(r_q-2)d^2(d-1)^2$.
\end{lemma}
\begin{proof}
When $r_q=2$, the above bound for $|{\bf W}|$ is obvious, so we shall consider the case when $r_q\geq 3$.
Let $\{i_1,i_2\}$ be the pair of left vertices operated on by the simple switching transforming $\Gamma$ to $\tilde\Gamma$.
Applying Lemma~\ref{l: 29570295873059873}, we get that any subset $W\in{\bf W}$
must contain at least one left vertex $i$ (distinct from $i_1,i_2$) which has common neighbors with $W\setminus\{i\}$.
Hence, any set $W\in{\bf W}$ can be identified by first picking $(r_q-3)$--subset of $[n]\setminus\{i_1,i_2\}$
and then selecting a left vertex having at least one common neighbor with either the constructed $(r_q-3)$--subset
or $\{i_1,i_2\}$. The $d$--regularity implies that the number of choices for this last index is at most $(r_q-1)d(d-1)$,
and the first bound follows.

Now, consider the case $r_q\geq 5$. Note that in this case the total multiplicity of multiedges of each of the graphs
$G_1',G_2'$ incident to $W_q(\tuple)$ must be at least four, and the 
part of the connection is constructed according to Lemma~\ref{l: 2095203598305298 bis}. 
Hence, we get that
for each tuple $\tuple \in T'$ there are two left vertices $\{i',i''\}\subset W_q(\tuple)$ disjoint from the vertices operated on by
the switching connecting $\Gamma$ and $\tilde\Gamma$, such that $i'$ has a common neighbor with $W_q(\tuple)\setminus\{i'\}$
in $\Gamma$ and $i''$ has a common neighbor with $W_q(\tuple)\setminus\{i',i''\}$.
Hence, a set $W\in{\bf W'}$ can be identified by first picking $(r_q-4)$--subset of $[n]\setminus\{i_1,i_2\}$,
then selecting a left vertex having at least one common neighbor with either the constructed $(r_q-4)$--subset
or $\{i_1,i_2\}$, and finally picking one more left vertex having at least one common neighbor with the $(r_q-1)$--subset.
The total number of choices can be bounded above by $n^{r_q-4}(r_q-2)d(d-1)(r_q-1)d(d-1)$.
The result follows.
\end{proof}

\bigskip

Finally, we are in position to compute ``complexity'' of $I_{NS}(\tuple)$ by combining the above three lemmas.
Let us introduce additional classification on the set of admissible $4$--tuples of graphs.
\begin{defi}\label{def: connection1}
Let $k_1,k_2,r\in\N$ and $u\in\{A,B\}$.
We say that a $4$--tuple $\tuple=(G_1,G_1',G_2,G_2')$ is {\it of subtype $1$--$(k_1,k_2,u,r)$} if all of the following conditions are satisfied:
\begin{itemize}

\item $\tuple$ is of type $1$;

\item $G_1'$ is of category $k_1$ and $G_2'$ is of category $k_2$;

\item $|I_{NS}(\tuple)|=r$;

\item $G_1'$ and $G_2'$ are adjacent, and the simple switching operation transforming $G_1'$ to $G_2'$ is of type $u$.

\end{itemize}
\end{defi}

Similarly,
\begin{defi}\label{def: connection2}
We say that an admissible $4$--tuple $\tuple=(G_1,G',G_2,G')$ is {\it of subtype $2$--$(k,r)$}, $k,r\in \N$, if
\begin{itemize}

\item $G_1,G_2\in \SNeigh(G')$;

\item $G'$ is of category $k$;

\item $|I_{NS}(\tuple)|=r$.

\end{itemize}
\end{defi}

For the reader's convenience, we group together the classification of $4$--tuples used in this section in the following table.
\begin{center}
\begin{tabular}{ |p{7cm}|p{7cm}| }
\hline
{\bf Type} & {\bf Subtype\,/\, Remarks}\\
\hline
\multirow{2}{7cm}{{\bf{}Type $1$ tuples.} Tuples
$(G_1,G_1',G_2,G_2')$ such that $G_1\in \SNeigh(G_1')$, $G_2\in \SNeigh(G_2')$, the graphs $G_1'$ and $G_2'$ are adjacent in
$\MultBipGSet_n(d)$; $G_1',G_2'\in \categ([1,\cconst])$, and the pair $(G_1',G_2')$ is not perfect.} & {\bf Subtype $1$--$(k_1,k_2,A,r)$.} $G_1'$ has category $k_1$, $G_2'$ --- category $k_2$,
$|I_{NS}(\tuple)|=r$, and $G_1'$ and $G_2'$ are connected by a type $A$ switching. \\
 & {\bf Subtype $1$--$(k_1,k_2,B,r)$.} $G_1'$ has category $k_1$, $G_2'$ --- category $k_2$,
$|I_{NS}(\tuple)|=r$, and $G_1'$ and $G_2'$ are connected by a type $B$ switching. \\
\hline
{\bf{}Type $2$ tuples.} Tuples $(G_1,G',G_2,G')$ such that $G_1\in \SNeigh(G')$, $G_2\in \SNeigh(G')$, and
$G'\in \categ([1,\cconst])$. & 
{\bf Subtype $2$--$(k,r)$.} $G'$ has category $k$, and $|I_{NS}(\tuple)|=r$.\\
\hline
{\bf{}Type $3$ tuples.} Tuples $(G_1,G_1,G_2,G_2')$ such that
$G_1\in \BipGSet_n(d)$ is adjacent to $G_2'\in\categ([1,\cconst])$, and $G_2\in \SNeigh(G_2')$, $G_1\notin \SNeigh(G_2')$. & 
For every type $3$ tuple, the category of $G_2'$ is equal to $2$, and $G_1$ and $G_2'$ are connected by a type $B$ switching. \\
\hline
\end{tabular}
\end{center}

\begin{prop}[Complexity of $I_{NS}(\tuple)$]\label{p: complexity of INS}
Let $(\Gamma,\tilde \Gamma)$ be a pair of adjacent graphs in $\BipGSet_n(d)$,
and let $k_1,k_2\in\N$, $r\geq 1$, $N_{ms}\geq 0$.
Then
\begin{enumerate}

\item The number of distinct realizations of $I_{NS}(\tuple)$ over all
$4$--tuples $\tuple=(G_1,G_1',G_2,G_2')$ of subtype $1$--$(k_1,k_2,A,r)$ such that $|E_S(\tuple)|=N_{ms}$ and
$(\Gamma,\tilde \Gamma)$ is an $m$--standard couple or an $A$--couple in $P_\tuple$, is bounded above by
$$
8^r n^{k_1+k_2-2N_{ms}-1}\big(rd(d-1)\big)^{r}.
$$
Similarly, the number of distinct realizations of $I_{NS}(\tuple)$ over
tuples $\tuple=(G_1,G',G_2,G')$ of subtype $2$--$(k_1,r)$ such that $|E_S(\tuple)|=N_{ms}$ and
$(\Gamma,\tilde \Gamma)$ is an $m$--standard couple in $P_\tuple$, is bounded above by
$$
8^r n^{2k_1-2N_{ms}-1}\big(rd(d-1)\big)^{r}.
$$

\item The number of distinct realizations of $I_{NS}(\tuple)$ over all
$4$--tuples $\tuple=(G_1,G_1',G_2,G_2')$ of subtype $1$--$(k_1,k_2,B,r)$ such that $|E_S(\tuple)|=N_{ms}$ and
$(\Gamma,\tilde \Gamma)$ is an $m$--standard couple in $P_\tuple$, is bounded above by
$$
8^r\,r\, n^{k_1+k_2-2N_{ms}+1}\big(rd(d-1)\big)^{r}.
$$

\item The number of distinct realizations of $I_{NS}(\tuple)$ over all
$4$--tuples $\tuple=(G_1,G_1',G_2,G_2')$ of subtype $1$--$(k_1,k_2,A,r)$ such that $|E_S(\tuple)|=N_{ms}$ and
$(\Gamma,\tilde \Gamma)$ is an $NS$--couple in $P_\tuple$, is bounded above by
$$
(r-1)d(d-1),\quad \mbox{ if $r\leq 3$},
$$
and 
$$
8^r\,r\, n^{k_1+k_2-2N_{ms}-3}\big(rd(d-1)\big)^{r},\quad \mbox{ if $r\geq 4$}.
$$
Similarly, the number of distinct realizations of $I_{NS}(\tuple)$ over
tuples $\tuple=(G_1,G',G_2,G')$ of subtype $2$--$(k_1,r)$ such that $|E_S(\tuple)|=N_{ms}$ and
$(\Gamma,\tilde \Gamma)$ is an $NS$--couple in $P_\tuple$, is bounded above by
$
(r-1)d(d-1)
$ for $r\leq 3$ and by
$
8^r\,r\, n^{2k_1-2N_{ms}-3}\big(rd(d-1)\big)^{r}
$ for $r\geq 4$.

\item The number of distinct realizations of $I_{NS}(\tuple)$ over all
$4$--tuples $\tuple=(G_1,G_1',G_2,G_2')$ of subtype $1$--$(k_1,k_2,B,r)$ such that $|E_S(\tuple)|=N_{ms}$ and
$(\Gamma,\tilde \Gamma)$ is an $NS$--couple in $P_\tuple$, is bounded above by
$$
8^r\,r^2\, n^{k_1+k_2-2N_{ms}-1}\big(rd(d-1)\big)^{r}.
$$
\end{enumerate}
\end{prop}
\begin{proof}\hspace{0cm}\\
\begin{enumerate}
\item Fix any sequence of non-negative integers $(r_1,r_2,\dots)$, such that $r_{q}=0$ implies $r_{q+1}=0$ for all $q\geq 1$,
and such that $\sum_{q=1}^\infty r_q=r$. First, we shall estimate the number of realizations of sets $I_{NS}(\tuple)$
whose canonical partitions $(W_q(\tuple))_{q=1}^\infty$ satisfy $|W_q(\tuple)|=r_q$, $q\geq 1$.
Applying Lemma~\ref{l: 295872095875098}, we get that the number of realizations of each non-empty $W_q$
is bounded above by $4n^{\max(1,r_q-2)}\big(\max(1,r_q-2)d(d-1)\big)^2$ (regardless whether
we consider type $1$ or type $2$).
Taking the product, we obtain that the number of realizations of $I_{NS}(\tuple)$ with the cardinalities of the elements
of the canonical partition encoded by the sequence $(r_q)_{q=1}^\infty$, is bounded above by
$$
4^r n^{\sum_{q\geq 1,\,r_q\neq 0}\max(1,r_q-2)}\big(rd(d-1)\big)^{r},
$$
where we have used that the number of non-zero elements of the sequence $(r_q)$ is at most $r/2$.
Observe that, \chng{in notation of Lemma~\ref{l: 05610948704987},
$$
\sum_{q\geq 1,\,W_q(\tuple)\neq\emptyset} M^h_q=k_h-N_{ms},\quad h=1,2.
$$
Thus, in view of Lemma~\ref{l: 05610948704987},}
for every $\tuple$ of subtype $1$--$(k_1,k_2,A,r)$ satisfying the conditions
in the first part of the proposition we have
$$\sum_{q\geq 1,\,W_q(\tuple)\neq\emptyset}\max(1,|W_q(\tuple)|-2)\leq k_1+k_2-2N_{ms}-1\quad\mbox{ whenever $r\neq 0$.}$$
Similarly, for every $\tuple$ of subtype $2$--$(k_1,r)$ satisfying the conditions
in the first part of the proposition we have
$\sum_{q\geq 1,\,W_q(\tuple)\neq\emptyset}\max(1,|W_q(\tuple)|-2)\leq 2k_1-2N_{ms}-1$, whenever $r\neq 0$.
Further, the total number of admissible realizations of $(r_1,r_2,\dots)$ can be (roughly)
bounded from above by $2^r$; thus the number of admissible realizations of $I_{NS}(\tuple)$ for type $1$ case
is bounded from above by
$$
8^r n^{k_1+k_2-2N_{ms}-1}\big(rd(d-1)\big)^{r},
$$
with the straightforward modification for subtype $2$--$(k_1,r)$ tuples.

\item As in the previous case, we fix a sequence of non-negative integers $(r_1,r_2,\dots)$, and,
additionally, fix a number $u\leq r$ which will serve as a label for the element of the canonical partition
containing the left vertices operated on by the simple switching connecting $G_1'$ and $G_2'$.
Applying Lemma~\ref{l: 295872095875098} and Lemma~\ref{l: -298570596098709},
we get that the total number of realizations of $I_{NS}(\tuple)$ satisfying $|W_q(\tuple)|=r_q$, $q\geq 1$,
and $W_u(\tuple)$ containing the left vertices used to switch from $G_1'$ to $G_2'$,
is bounded above by
\begin{align*}
n^{r_u-1}(r_u-1)d(d-1)
\prod\limits_{q\neq u,\,r_q\neq 0}\Big(4n^{\max(1,r_q-2)}\big(\max(1,r_q-2)d(d-1)\big)^2\Big).
\end{align*}
Invoking Lemma~\ref{l: 05610948704987} (\chng{the ``in particular'' part of the lemma, to be precise}), we get that
$$
r_u+\sum_{q\neq u,\,r_q\neq0}\max(1,r_q-2)\leq k_1+k_2-2N_{ms}+2
$$
(the worst-case estimate corresponds to the situation when $r_u=r_1=r$),
whence the previous expression can be bounded from above by
$$
4^r\big(rd(d-1)\big)^{r}n^{k_1+k_2-2N_{ms}+1}.
$$
The result follows by summing over all possible choices of $(r_q)_{q=1}^\infty$ and $u\leq r$.

\item Again, we fix a sequence $(r_1,r_2,\dots)$ with $\sum_{q=1}^\infty r_q=r$, and,
additionally, fix a number $u\leq r$ which identifies the element of the canonical partition
containing the left vertices $\{i_1,i_2\}$ operated on by the simple switching connecting $\Gamma$ and $\tilde \Gamma$.

First, consider the case $r\leq 3$.
Note that in this case $I_{NS}(\tuple)$ consists of a single set $W_1$, which comprises exactly one multiedge. 
Applying Lemma~\ref{l: 985709575987}, we get upper bound $(r-1)d(d-1)$ for the complexity,
both for subtype $1$--$(k_1,k_2,A,r)$ tuples and for subtype $2$--$(k_1,r)$ tuples.


Now, assume that $r\geq 4$. We have to consider two subcases then.

\begin{itemize}

\item There is $q\neq u$ with $r_q\neq 0$. Then, applying Lemmas~\ref{l: 295872095875098} and~\ref{l: 985709575987}
as above, we get that the number of possible realizations is bounded above by
\begin{align*}
&n^{\max(1,r_u-2)-1}(r_u-1)d(d-1)\cdot\\
&\prod\limits_{q\neq u,\,r_q\neq 0} \big(4n^{\max(1,r_q-2)}\big(\max(1,r_q-2)d(d-1)\big)^2\big)\\
&\leq 4^r\big(rd(d-1)\big)^{r}n^{k_1+k_2-2N_{ms}-3},
\end{align*}
where we also used Lemma~\ref{l: 05610948704987} (\chng{the first assertion $\max(1,r_x-2)\leq 2M_x^1-1=2M_x^2-1$, to be precise}),
and this time took into account that there are factors
corresponding to $q\neq u,\,r_q\neq 0$,
with straightforward modification for subtype $2$--$(k_1,r)$.

\item There is no $q\neq u$ with $r_q\neq 0$. Then necessarily $r_u\geq 4$.
If $r_u=4$ then using Lemma~\ref{l: 985709575987} we bound the complexity by
$$
n^{r_u-3}(r_u-1)d(d-1)=3nd(d-1)\leq 3d(d-1)n^{k_1+k_2-2N_{ms}-3},
$$
where we use the simple observation that $k_1+k_2-2N_{ms}\geq 4$ under the assumption $r>3$.
If $r_u\geq 5$ then we apply the ``furthermore'' part of Lemma~\ref{l: 985709575987}
to get the upper bound
$$
n^{r_u-4}(r_u-1)(r_u-2)d^2(d-1)^2\leq r^2 d^2 (d-1)^2 n^{k_1+k_2-2N_{ms}-3},
$$
where the last inequality follows by appying Lemma~\ref{l: 05610948704987}.
Again, a direct modification of the argument gives $r^2 d^2 (d-1)^2 n^{2k_1-2N_{ms}-3}$
as the upper bound in subtype $2$--$(k_1,r)$ case.
\end{itemize}

Summing over $(r_q)_{q=1}^\infty$ and $u\leq r$, we get the estimate.

\item For the $4$--th case, we condition first on cardinalities $(r_q)_{q=1}^\infty$ of the elements of the
canonical partition and on two indices $u$ and $v$, where $u$ identifies the element of the partition of $I_{NS}(\tuple)$
containing the left vertices used to switch between $\Gamma$ and $\tilde\Gamma$, and
$v$ identifies the element of the canonical partition
containing the left vertices operated on by the simple switching connecting $G_1'$ and $G_2'$.
When $u\neq v$, we can bound from above the number of realizations of $I_{NS}(\tuple)$ by combining
Lemmas~\ref{l: 985709575987},~\ref{l: 295872095875098}, and~\ref{l: -298570596098709}, by
\begin{align*}
&n^{\max(1,r_u-2)-1}(r_u-1)d(d-1)\;
n^{r_v-1}(r_v-1)d(d-1)\cdot\\
&\prod\limits_{q\neq u,v,\,r_q\neq 0} \big(4n^{\max(1,r_q-2)}\big(\max(1,r_q-2)d(d-1)\big)^2\big),
\end{align*}
which in turn can be bounded with help of Lemma~\ref{l: 05610948704987} by
$$
4^r\big(rd(d-1)\big)^{r}n^{k_1+k_2-2N_{ms}-1}.
$$
Further, when $u= v$ then we use a combination of Lemmas~\ref{l: 295872095875098} and~\ref{l: 985709575987}
to obtain upper bound
$$
n^{\max(0,r_u-3)}(r_u-1)d(d-1)\prod\limits_{q\neq u,\,r_q\neq 0} \big(4n^{\max(1,r_q-2)}\big(\max(1,r_q-2)d(d-1)\big)^2\big).
$$
When either $r_u\geq 3$ or $\{q\neq u,\,r_q\neq 0\}\neq \emptyset$, this expression is clearly bounded from above,
using Lemma~\ref{l: 05610948704987}, by
$$
4^r\big(rd(d-1)\big)^{r}n^{k_1+k_2-2N_{ms}-1}.
$$
On the other hand, when $r_u=r=2$, the couple $(\Gamma,\tilde \Gamma)$ defines the set $I_{NS}(\tuple)$
uniquely, so, using that $k_1+k_2-2N_{ms}-1\geq 0$, we can trivially bound the number of choices by the above expression.
Summing over all choices of $v,u$ and $(r_q)$ then gives the result.
\end{enumerate}
\end{proof}

\subsection{Complexity of the data set for $m$--standard edges}

In this subsection, we consider the problem of ``recovering'' the information about $m$--standard
edges of an admissible tuple, by observing a couple of graphs in the corresponding connection.
For convenience, let us introduce a data structure $\mdata(\tuple)$ associated with every admissible
$\tuple=(G_1,G_1',G_2,G_2')$.
The data structure $\mdata(\tuple)$ is defined as a collection of $5$-tuples:
$$
\mdata(\tuple):=\big\{\big((i,j),\ione(i,\tuple),\itwo(i,\tuple),\jone(i,\tuple),\jtwo(i,\tuple)\big):\;(i,j)\in E_S(\tuple)\big\}.
$$

\begin{prop}[Complexity of $\mdata(\tuple)$]\label{p: mdata counting}
Let $(\Gamma,\tilde \Gamma)$ be a pair of adjacent graphs in $\BipGSet_n(d)$,
and let $N_{ms}>0$.
Then
\begin{enumerate}

\item The number of distinct realizations of $\mdata(\tuple)$ over all admissible
$\tuple=(G_1,G_1',G_2,G_2')$ such that $|E_S(\tuple)|=N_{ms}$ and
$(\Gamma,\tilde \Gamma)$ is an $m$--standard couple in $P_\tuple$, is bounded above by
$$
4{n-1 \choose N_{ms}-1}
\big(d^2(d-1)^2n\big)^{N_{ms}-1}d^2.
$$

\item The number of distinct realizations of $\mdata(\tuple)$ over all admissible
$\tuple=(G_1,G_1',G_2,G_2')$ such that $|E_S(\tuple)|=N_{ms}$ and
$(\Gamma,\tilde \Gamma)$ is an $NS$--couple 
in $P_\tuple$, is bounded above by
$$
{n\choose N_{ms}}\big(d^2(d-1)^2 n\big)^{N_{ms}}.
$$

\item The number of distinct realizations of $\mdata(\tuple)$ over all
$4$--tuples $\tuple=(G_1,G_1',G_2,G_2')$ of type $1$ such that $|E_S(\tuple)|=N_{ms}$ and
$(\Gamma,\tilde \Gamma)$ is an 
$A$--couple
in $P_\tuple$, is bounded above by
$$
2d^2 {n-1\choose N_{ms}-1}\big(d^2(d-1)^2 n\big)^{N_{ms}}.
$$

\end{enumerate}

\end{prop}
\begin{proof}\hspace{0cm}\\
\begin{enumerate}

\item
Suppose that $\tilde \Gamma$ is obtained from $\Gamma$ by a simple switching on edges $(v,w)$ and $(\tilde v,\tilde w)$.
If $(\Gamma,\tilde \Gamma)=(P_\tuple[t-1],P_\tuple[t])$ for odd $t$
then one of the right vertices $w$ or $\tilde w$ must be incident to the corresponding $m$--standard edge,
and thus there are at most $2(d-1)$ choices for the corresponding $m$--standard edge. Once that multiple edge
of $G_1'$ (and $G_2'$)
is determined, we have at most $d-1$ possible choices for the simple switching taking $\tilde\Gamma$ to $P_\tuple[t+1]$,
so overall number of the choices for $P_\tuple[t+1]$ and the corresponding $m$--standard edge can be bounded by $2d^2$.
Similarly, it is not difficult to check that when $t$ is even, the complexity of the choices for $P_\tuple[t-2]$ and 
the corresponding $m$--standard edge of $G_1'$ and $G_2'$ is at most $2d^2$ as well (see Figure~\ref{fig: m-standard}).
Note that this operation uniquely identifies the $5$--tuple associated with that $m$--standard edge.

Let $i_\Gamma$ be the left vertex incident to the $m$--standard edge in $G_1'$ and $G_2'$ associated with $(\Gamma,\tilde \Gamma)$.
Then $(N_{ms}-1)$ $5$--tuples associated with the remaining $m$--standard edges
can be uniquely encoded
with a parameter taking at most
\begin{align*}
{n-1 \choose N_{ms}-1}
\big(d^2(d-1)^2n\big)^{N_{ms}-1}
\end{align*}
distinct values.
Here, the multiple ${n-1 \choose N_{ms}-1}$ is responsible for choosing the left vertices for the remaining $N_{ms}-1$
standard edges (the order of that choice is not important). 
Further, having chosen an $m$--standard edge $(i,j)$, there are $(d-1)^2$ choices for two edges,
one incident to $i^\ell$ and the other to $j^r$, and $nd$ choices for the remaining ``free'' edge completing the pattern in
Figure~\ref{fig: m-standard}. 
The required estimate follows (the additional factor $2$ comes from the separate treatment of even and odd indices at the beginning
of the argument. 

\item Treatment of this case essentially repeats the second part of the above argument, with some simplifications.
The total number of choices of $N_{ms}$ $m$--standard edges is clearly bounded above by ${n\choose N_{ms}}d^{N_{ms}}$.
For every edge, we complete the pattern (i.e.\ determine $\ione,\itwo,\jone,\jtwo$) by picking one of at most
$(d-1)^2 nd$ combinations. Thus, we obtain the upper bound
$$
{n\choose N_{ms}}\big(d^2(d-1)^2 n\big)^{N_{ms}}.
$$

\item We shall use the definition of $A$--switching in combination with the fact that
the type $1$ tuples correspond to non-perfect pairs of multigraphs (see Definition~\ref{def: perfect pair}).
Let $i',i''$ and $j',j''$ be the two left and right vertices, respectively,
operated on by the switching connecting $\Gamma$ and $\tilde\Gamma$. Then one of the following must be true:
\begin{itemize}
\item $\{j',j''\}$ is incident to an $m$--standard edge (in $G_1',G_2'$);
\item For some $h\in\{1,2\}$, $\{i',i''\}$ is adjacent to a right vertex incident to an $m$--standard edge in graph $G_h'$;
\item For some $h\in\{1,2\}$, $\{j',j''\}$ is adjacent to a left vertex incident to an $m$--standard edge in graph $G_h'$.
\end{itemize}
Note that in the first case, $\{j',j''\}$ must be incident to an $m$--standard edge in both $\Gamma$ and $\tilde \Gamma$.
In the second case, $\{i',i''\}$ is adjacent to a right vertex incident to an $m$--standard edge in both $\Gamma$ and $\tilde \Gamma$,
because otherwise we would have $\{i',i''\}\cap I(\tuple)\neq\emptyset$ which contradicts to the assumption that
$G_1'$ and $G_2'$ are connected by $A$-- (and not $B$--) switching. In the third case, 
$\{j',j''\}$ is adjacent to a left vertex incident to an $m$--standard edge in both $\Gamma$ and $\tilde \Gamma$
because the edges $\{(i,v):\;(i,j)\in E_S(\tuple),\;v\in \neigh_{G_1'}(i^\ell)\setminus\{j\}\}$ are not affected by the switching
operations and are also edges in graphs $G_1,G_2',G_2$ and any graph in the connection $P_\tuple$.
To summarize,
if we denote by $U$ the collection of all left vertices which are either adjacent to $j'$ or $j''$,
or at distance two from $i'$ or $i''$ in graph $\Gamma$ then necessarily $U$ is incident to an $m$--standard edge.
Clearly, the cardinality of $U$ is at most $2d+2d(d-1)=2d^2$. Thus,
the complexity of $\mdata(\tuple)$ can be estimated as follows:
first, we choose an $m$--standard edge incident to $U$ in $\Gamma$ (there are at most $2d^3$ choices).
Then we determine corresponding $\ione,\jone,\itwo,\jtwo$ (at most $(d-1)^2 nd$ choices, see Figure~\ref{fig: m-standard}).
After that, we identify the part of the data structure associated with the remaining $(N_{ms}-1)$ $m$--standard
edges; we have at most
$$
{n-1\choose N_{ms}-1}\big(d^2(d-1)^2 n\big)^{N_{ms}-1}
$$
choices for that. The estimate follows.

\end{enumerate}
\end{proof}

\subsection{Complexity of the set of $4$--tuples}


\begin{prop}[Type $1$ counting]\label{prop: number-tuples1}
Let $(\Gamma,\tilde \Gamma)$ be a pair of adjacent graphs in $\BipGSet_n(d)$,
and let $k_1,k_2\geq 1$, $r\in\N_0$, and $u\in \{A,B\}$.
Then the number of $4$--tuples $\tuple=(G_1,G_1',G_2,G_2')$ of subtype $1$--$(k_1,k_2,u,r)$ such that
$\Gamma=P_\tuple[t-1]$ and $\tilde \Gamma=P_\tuple[t]$ for some $t\geq 1$, is bounded above by
$$
\frac{C_r (k_1+k_2)^{r+2}}{\lfloor k_1/2+k_2/2\rfloor !}
\, n^{k_1+k_2-1}d^{3+r+C_{\text{\tiny{\ref{l: 2095203598305298 bis}}}}(16r+8)+ k_1+k_2}
(d-1)^{r+k_1+k_2},
$$
where $C_r>0$ may only depend on $r$.
\end{prop}
\begin{proof}
Let us first fix the number of $m$--standard edges, i.e.\ fix an integer $N_{ms}$
and estimate the number of tuples satisfying the conditions of the proposition under the additional
assumption that $|E_S(\tuple)|=N_{ms}$.

Let us make a remark that, conditioned on $I_{NS}(\tuple)=I$ (for a fixed subset $I\subset[n]$ of size $r$),
the $NS$--couples in $P_\tuple$, the corresponding multiple edges in $G_1'$ and $G_2'$, and the simple switchings
on simple paths from $G_h'$ to $G_h$, $h=1,2$, operating on $I_{NS}(\tuple)=I$,
can all be {\it uniquely} identified with a parameter
taking at most
\begin{equation}\label{eq: 027619284760498}
\big((rd)^2\big)^{C_{\text{\tiny\ref{l: 2095203598305298 bis}}}(8r+4)}\big(rd\,(d-1)(r-1)\big)^{k_1+k_2-2N_{ms}}
\end{equation}
values.
Here, the multiple $\big((rd)^2\big)^{C_{\text{\tiny\ref{l: 2095203598305298 bis}}}(8r+4)}$
corresponds to identification of the $NS$--couples in $P_\tuple$
(where we have applied 
Lemmas~\ref{l: 208762058709283},~\ref{l: 2095203598305298},~\ref{l: 2095203598305298 bis}),
and the multiple $(rd\,(d-1)(r-1))^{k_1+k_2-2N_{ms}}$ corresponds to recovering the simple paths leading from $G_h'$ to $G_h$, $h=1,2$. 
In what follows, we will use that $k_1+k_2-2N_{ms}\leq 2r$ (see Lemma~\ref{l: 2-9572059827098}). 

We shall consider three cases corresponding to the type of the couple $(\Gamma,\tilde \Gamma)$ in the connection $P_\tuple$.

\begin{itemize}

\item Assume first that $(\Gamma,\tilde \Gamma)$ is $m$--standard.
In this case, the complexity of the data structure $\mdata$, according to Proposition~\ref{p: mdata counting},
is
$$
4{n-1 \choose N_{ms}-1}
\big(d^2(d-1)^2n\big)^{N_{ms}-1}d^2.
$$
Further, we consider several subcases.
\begin{itemize}

\item {\it $u=A$ and $r\neq 0$, i.e.\ there are non-standard multiedges.}
There are at most $(nd)^2$ choices for the simple switching between $G_1'$ and $G_2'$. 
Further, applying corresponding part of Proposition~\ref{p: complexity of INS},
we get that the complexity of $I_{NS}(\tuple)$ is bounded above by
$$
8^r n^{k_1+k_2-2N_{ms}-1}\big(rd(d-1)\big)^{r}.
$$
Taking into account \eqref{eq: 027619284760498},
we get the combined upper bound
\begin{align*}
4&{n-1 \choose N_{ms}-1}
\big(d^2(d-1)^2n\big)^{N_{ms}-1}d^2\cdot (nd)^2\cdot
8^r n^{k_1+k_2-2N_{ms}-1}\\
&\cdot \big(rd(d-1)\big)^{r}\big((rd)^2\big)^{C_{\text{\tiny\ref{l: 2095203598305298 bis}}}(8r+4)}(rd\,(d-1)r)^{k_1+k_2-2N_{ms}},
\end{align*}
which can be in turn bounded above by
$$
\frac{C_r}{(N_{ms}-1)!}\,n^{k_1+k_2-1}d^{2+r+C_{\text{\tiny\ref{l: 2095203598305298 bis}}}(16r+8)+k_1+k_2}(d-1)^{r-2+k_1+k_2},
$$
for some $C_r>0$ depending only on $r$. 

\item {\it $u=B$ and $r\neq 0$, i.e.\ there are non-standard multiedges.}
The only difference from the above subcase is that we do not introduce the additional multiple $(nd)^2$ to account for the simple
switching connecting $G_1'$ and $G_2'$ and use the bound
$$
8^r\,r\, n^{k_1+k_2-2N_{ms}+1}\big(rd(d-1)\big)^{r}
$$
for the complexity of $I_{NS}(\tuple)$.
Overall, we obtain the upper bound
\begin{align*}
4&{n-1 \choose N_{ms}-1}
\big(d^2(d-1)^2n\big)^{N_{ms}-1}d^2\cdot
8^r\,r\, n^{k_1+k_2-2N_{ms}+1}\big(rd(d-1)\big)^{r}\\
&\cdot \big((rd)^2\big)^{C_{\text{\tiny\ref{l: 2095203598305298 bis}}}(8r+4)}(rd\,(d-1)r)^{k_1+k_2-2N_{ms}},
\end{align*}
which gives, after simplification, the estimate
$$
\frac{C_r}{(N_{ms}-1)!}\,n^{k_1+k_2-1}d^{r+C_{\text{\tiny\ref{l: 2095203598305298 bis}}}(16r+8)+k_1+k_2}(d-1)^{r-2+k_1+k_2}.
$$

\item {\it $r=0$, i.e.\ there are no non-standard multiedges, and $2N_{ms}=k_1+k_2$.} 
As the data structure $\mdata$ has already been defined,
we can uniquely identify the set $U$ of left vertices which are at distance
at most one from some $m$--standard edge in $G_1'$ (and $G_2'$). The switching between $G_1'$ and $G_2'$
must operate on a vertex from $U$ because the couple $(G_1',G_2')$ is not perfect (see Definition~\ref{def: perfect pair}).
The cardinality of $U$ is at most $N_{ms}+N_{ms}(d-1)=d N_{ms}$. Thus,
the switching connecting $G_1'$ and $G_2'$ can be identified with a parameter taking at most $d^2 N_{ms}\cdot dn$ values.
Thus, we can estimate the overall complexity by
$$
\frac{C N_{ms}}{(N_{ms}-1)!}\,n^{k_1+k_2-1}d^{k_1+k_2+3}(d-1)^{k_1+k_2-2}.
$$

\end{itemize}

\item Assume that $(\Gamma,\tilde \Gamma)$ is the $A$--couple.
We have two subcases:

\begin{itemize}

\item {\it $r\neq 0$, i.e.\ there are non-standard multiedges.}
The complexity of $I_{NS}(\tuple)$ is bounded, using Proposition~\ref{p: complexity of INS}, by
$$
8^r n^{k_1+k_2-2N_{ms}-1}\big(rd(d-1)\big)^{r}.
$$
The complexity of $\mdata(\tuple)$ is given by Proposition~\ref{p: mdata counting}:
$$
\chng{2d^2 {n-1\choose N_{ms}-1}\big(d^2(d-1)^2 n\big)^{N_{ms}}.
}
$$
Combining the estimate with \eqref{eq: 027619284760498},
we get a bound
$$
\chng{\frac{C_r}{(N_{ms}-1)!}\,n^{k_1+k_2-2}d^{2+r+C_{\text{\tiny\ref{l: 2095203598305298 bis}}}(16r+8)
+k_1+k_2}(d-1)^{r+k_1+k_2}.}
$$

\item {\it $r=0$.} 
Applying Proposition~\ref{p: mdata counting}, we get the complexity bound for $\mdata$:
$$
2d^2 {n-1\choose N_{ms}-1}\big(d^2(d-1)^2 n\big)^{N_{ms}},
$$
which can be further simplified to
$$
\frac{C}{(N_{ms}-1)!}\,n^{k_1+k_2-1}d^{k_1+k_2+2}(d-1)^{k_1+k_2}.
$$

\end{itemize}

\item Finally, consider the case when $(\Gamma,\tilde \Gamma)$ is an $NS$--couple.
By Proposition~\ref{p: mdata counting}, the complexity of $\mdata$ is at most
$$
{n\choose N_{ms}}\big(d^2(d-1)^2 n\big)^{N_{ms}}.
$$
We shall consider the following subcases:

\begin{itemize}

\item $u=B$. Then, applying Proposition~\ref{p: complexity of INS}, we get the complexity bound
$$
8^r\,r^2\, n^{k_1+k_2-2N_{ms}-1}\big(rd(d-1)\big)^{r}
$$
for $I_{NS}$.
Thus, in view of \eqref{eq: 027619284760498}, the total complexity is bounded above by
\begin{align*}
&{n\choose N_{ms}}\big(d^2(d-1)^2 n\big)^{N_{ms}}\cdot
8^r\,r^2\, n^{k_1+k_2-2N_{ms}-1}\big(rd(d-1)\big)^{r}\\
&\cdot
\big((rd)^2\big)^{C_{\text{\tiny\ref{l: 2095203598305298 bis}}}(8r+4)}(rd\,(d-1)r)^{k_1+k_2-2N_{ms}},
\end{align*}
which can be simplified to
$$
\frac{C_r}{N_{ms}!}\, n^{k_1+k_2-1}d^{r+C_{\text{\tiny\ref{l: 2095203598305298 bis}}}(16r+8)+ k_1+k_2}
(d-1)^{r+k_1+k_2}.
$$

\item {\it $u=A$ and $r\geq 4$.} Applying Proposition~\ref{p: complexity of INS}, we get the complexity bound
$$
8^r\,r\, n^{k_1+k_2-2N_{ms}-3}\big(rd(d-1)\big)^{r}
$$
for $I_{NS}$. Combining this with an estimate $(nd)^2$ for the number of choices of the $A$--switching
and \eqref{eq: 027619284760498}, we get upper bound
$$
\frac{C_r}{N_{ms}!}\, n^{k_1+k_2-1}d^{2+r+C_{\text{\tiny\ref{l: 2095203598305298 bis}}}(16r+8)+ k_1+k_2}
(d-1)^{r+k_1+k_2}.
$$

\item {\it $u=A$ and $r<4$.} Note that necessarily $2N_{ms}+2=k_1+k_2$. 
From Proposition~\ref{p: complexity of INS}, we get that the complexity of $I_{NS}$ is at most
$$
2d(d-1).
$$
The complexity of the $A$--switching can be estimated by $nd\cdot N_{ms}\cdot d^2$, using the assumption that
the couple of multigraphs is not perfect.
Together with \eqref{eq: 027619284760498}, this gives upper bound
\begin{align*}
&{n\choose N_{ms}}\big(d^2(d-1)^2 n\big)^{N_{ms}}\cdot nd\cdot N_{ms}\cdot d^2\cdot
2d(d-1)\\
&\cdot
\big((rd)^2\big)^{C_{\text{\tiny\ref{l: 2095203598305298 bis}}}(8r+4)}(rd\,(d-1)r)^{k_1+k_2-2N_{ms}},
\end{align*}
simplifying to the upper bound
$$
\frac{C_r}{(N_{ms}-1)!}\, n^{k_1+k_2-1}d^{3+r+C_{\text{\tiny{\ref{l: 2095203598305298 bis}}}}(16r+8)+ k_1+k_2}(d-1)^{r+k_1+k_2}.
$$

\end{itemize}
\end{itemize}

Combining all the above bounds, we get the following: the number of
$4$--tuples $\tuple=(G_1,G_1',G_2,G_2')$ of subtype $1$--$(k_1,k_2,u,r)$ such that
$\Gamma=P_\tuple[t-1]$ and $\tilde \Gamma=P_\tuple[t]$ for some $t\geq 1$, and $|E_S(\tuple)|=N_{ms}$,
is bounded above by
$$
\frac{C_r N_{ms}}{(N_{ms}-1)!}\, n^{k_1+k_2-1}d^{3+r+C_{\text{\tiny{\ref{l: 2095203598305298 bis}}}}(16r+8)+ k_1+k_2}
(d-1)^{r+k_1+k_2}
$$
for some $C_r>0$ depending only on $r$.

Now, we observe that the parameter $N_{ms}$ can take only a constant number of values:
$2N_{ms}\leq k_1+k_2$, whereas $k_1+k_2-2N_{ms}\leq 2r$ (see Lemma~\ref{l: 2-9572059827098}).
Thus, for any admissible $N_{ms}$ we have
\begin{align*}
&\frac{C_r N_{ms}}{(N_{ms}-1)!}\, n^{k_1+k_2-1}d^{3+r+C_{\text{\tiny{\ref{l: 2095203598305298 bis}}}}(16r+8)+ k_1+k_2}
(d-1)^{r+k_1+k_2}\\
&\leq
\frac{C_r' (k_1+k_2)^{r+2}}{\lfloor k_1/2+k_2/2\rfloor !}
\, n^{k_1+k_2-1}d^{3+r+C_{\text{\tiny{\ref{l: 2095203598305298 bis}}}}(16r+8)+ k_1+k_2}
(d-1)^{r+k_1+k_2}.
\end{align*}
The result follows.

\end{proof}

\begin{prop}[Type $2$ counting]\label{prop: number-tuples}
Let $(\Gamma,\tilde \Gamma)$ be a pair of adjacent graphs in $\BipGSet_n(d)$,
and let $k\geq 1$, $r\in\N_0$.
Then the number of $4$--tuples $\tuple=(G_1,G',G_2,G')$ of subtype $2$--$(k,r)$ such that
$\Gamma=P_\tuple[t-1]$ and $\tilde \Gamma=P_\tuple[t]$ for some $t\geq 1$, is bounded above by
$$
\frac{C_r k^{r+1}}{k !}
\, n^{2k-2}d^{r+C_{\text{\tiny{\ref{l: 2095203598305298 bis}}}}(16r+8)+ 2k}
(d-1)^{r+2k},
$$
where $C_r>0$ may only depend on $r$.
\end{prop}
\begin{proof}
As in the previous proof, we fix the number of $m$--standard edges $N_{ms}$
and estimate the number of tuples satisfying the conditions of the proposition under the assumption that $|E_S(\tuple)|=N_{ms}$.

Again, we note that, conditioned on $I_{NS}(\tuple)=I$ (for a fixed subset $I\subset[n]$ of size $r$),
the $NS$--couples in $P_\tuple$, the corresponding multiple edges in $G'$ and the simple switchings
on simple paths from $G'$ to $G_h$, $h=1,2$, operating on $I_{NS}(\tuple)=I$,
can all be {\it uniquely} identified with a parameter
taking at most
\begin{equation}\label{eq: 027619284760498 bis}
\big((rd)^2\big)^{C_{\text{\tiny\ref{l: 2095203598305298 bis}}}(8r+4)}(r^2d\,(d-1))^{2k-2N_{ms}}
\end{equation}
values. In what follows, we will use that $k-N_{ms}\leq r$ (see Lemma~\ref{l: 2-9572059827098}). 

We shall consider two cases corresponding to the type of the couple $(\Gamma,\tilde \Gamma)$ in the connection $P_\tuple$.

\begin{itemize}

\item Assume that $(\Gamma,\tilde \Gamma)$ is $m$--standard.
The complexity of the data structure $\mdata$, according to Proposition~\ref{p: mdata counting},
is
$$
4{n-1 \choose N_{ms}-1}
\big(d^2(d-1)^2n\big)^{N_{ms}-1}d^2.
$$
Further, we consider two subcases.
\begin{itemize}

\item {\it $r\neq 0$, i.e.\ there are non-standard multiedges.}
Applying corresponding part of Proposition~\ref{p: complexity of INS},
we get that the complexity of $I_{NS}(\tuple)$ is bounded above by
$$
8^r n^{2k-2N_{ms}-1}\big(rd(d-1)\big)^{r}.
$$
Taking into account \eqref{eq: 027619284760498 bis},
we get the combined upper bound
\begin{align*}
4&{n-1 \choose N_{ms}-1}
\big(d^2(d-1)^2n\big)^{N_{ms}-1}d^2\cdot 
8^r n^{2k-2N_{ms}-1}\big(rd(d-1)\big)^{r}\\
&\cdot \big((rd)^2\big)^{C_{\text{\tiny\ref{l: 2095203598305298 bis}}}(8r+4)}(rd\,(d-1)r)^{2k-2N_{ms}},
\end{align*}
which can be in turn bounded above by
$$
\frac{C_r}{(N_{ms}-1)!}\,n^{2k-3}d^{r+C_{\text{\tiny\ref{l: 2095203598305298 bis}}}(16r+8)+2k}(d-1)^{r-2+2k},
$$
for some $C_r>0$ depending only on $r$.

\item {\it $r=0$, i.e.\ there are no non-standard multiedges.}
In this subcase, 
we just estimate the overall complexity by complexity of $\mdata$ i.e.\ by
$$
4{n-1 \choose N_{ms}-1}
\big(d^2(d-1)^2n\big)^{N_{ms}-1}d^2\leq \frac{8}{(k-1)!}\,n^{2k-2} d^{2k}(d-1)^{2k-2}.
$$

\end{itemize}

\item $(\Gamma,\tilde \Gamma)$ is an $NS$--couple.
By Proposition~\ref{p: mdata counting}, the complexity of $\mdata$ is at most
$$
{n\choose N_{ms}}\big(d^2(d-1)^2 n\big)^{N_{ms}}.
$$
Applying Proposition~\ref{p: complexity of INS}, we get the complexity bound
$$
8^r\,r\, n^{2k-2N_{ms}-2}\big(rd(d-1)\big)^{r}
$$
for $I_{NS}$ (\chng{where we combined the two subcases $r\leq 3$ and $r\geq 4$ to treat them simultaneously}). Combining this with \eqref{eq: 027619284760498 bis}, we get upper bound
$$
\frac{C_r}{N_{ms}!}\, n^{2k-2}d^{r+C_{\text{\tiny\ref{l: 2095203598305298 bis}}}(16r+8)+ 2k}
(d-1)^{r+2k}.
$$
%
%

\end{itemize}

Finally, as in the previous proposition, we observe
that the parameter $N_{ms}$ can take only a constant number of values:
$N_{ms}\leq k$, and $k-N_{ms}\leq r$ (see Lemma~\ref{l: 2-9572059827098}).
Thus, for any admissible $N_{ms}$ we have
\begin{align*}
&\frac{C_r}{(N_{ms}-1)!}\, n^{2k-2}d^{r+C_{\text{\tiny{\ref{l: 2095203598305298 bis}}}}(16r+8)+ 2k}
(d-1)^{r+2k}\\
&\leq
\frac{C_r' k^{r+1}}{k !}
\, n^{2k-2}d^{r+C_{\text{\tiny{\ref{l: 2095203598305298 bis}}}}(16r+8)+ 2k}
(d-1)^{r+2k}.
\end{align*}
The result follows.

\end{proof}

\begin{prop}[Type $3$ counting]\label{prop: type3 count}
Let $(\Gamma,\tilde \Gamma)$ be a pair of adjacent graphs in $\BipGSet_n(d)$.
Then the number of $4$--tuples $\tuple=(G_1,G_1,G_2,G_2')$ of type $3$ such that
$\Gamma=P_\tuple[t-1]$ and $\tilde \Gamma=P_\tuple[t]$ for some $t\geq 1$, is bounded above by
$$
C_{\text{\tiny\ref{prop: type3 count}}}d^{20}n,
$$
where $C_{\text{\tiny\ref{prop: type3 count}}}>0$ is a universal constant.
\end{prop}
\begin{proof}
First, applying Lemma~\ref{l: 2095203598305298} and the definition of type $3$ tuples (see Figure~\ref{fig: type3}),
we get that any connection for a tuple of type $3$ has length at most $10$.

Fix any $t\leq 10$. We will bound from above the number of tuples $\tuple$ such that
$\Gamma=P_\tuple[t-1]$ and $\tilde \Gamma=P_\tuple[t]$.
Let $\{i_1,i_2\}$ be the couple of left vertices operated on by the switching taking $\Gamma$ to $\tilde\Gamma$.
Note that the total number of left vertices operated on by the connection for any type $3$ tuple, is four.
We start by bounding the number of choices of the other two left vertices $\{i_3,i_4\}$.
Note that, according to Lemma~\ref{l: 2095203598305298}, at least one of the following must be true:
\begin{itemize}
\item $i_1$ and $i_2$ have no common neighbors in graph $P_\tuple[t-1]$;
\item there is a left vertex $i\in\{i_3,i_4\}$ having at least one common neighbor with $\{i_1,i_2,i_3,i_4\}\setminus\{i\}$
in $P_\tuple[t-1]$.

\end{itemize}

At the same time, we know (see Figure~\ref{fig: type3}) that there are at least two right vertices $j_1,j_2$ such that
$j_h$ has at least two neighbors in $\{i_1,i_2,i_3,i_4\}$, $h=1,2$, in $P_\tuple[t-1]$,
so the second condition above must hold {\it in any case}.
Thus, the choice of $\{i_3,i_4\}$ can be accomplished by choosing one ``free'' left vertex $i_3$ and then
choosing $i_4$ having a common neighbor with $\{i_1,i_2,i_3\}$. This gives 
$3d(d-1)n$ possible realizations.

Having chosen $\{i_1,i_2,i_3,i_4\}$, we need to reconstruct the tuple $\tuple$ by applying simple switching operations
within the set. Note that at each step there are at most $16d^2$ choices for a simple switching operation.
Since $\length{P_\tuple}\leq 10$,
we need at most $9$ operations to reconstruct both $G_1$ and $G_2$.
Once this is done, there is only a constant (independent of $d$) number of possible choices for a pair
of edges in $G_1$ and $G_2$ which become multiedges in $G_2'$ (see Figure~\ref{fig: type3}). This will identify $\tuple$ uniquely.
Thus, we get that the total number of choices for $\tuple$ is bounded above by
$$
O\big(3d(d-1)n\cdot (16d^2)^{9}\big).
$$
The result follows.
\end{proof}

\section{Upper bound on the Dirichlet form of $\tilde f$}\label{sec: Dirichlet}

The goal of this section is to compare the Dirichlet form of a function $f$ on $( \BipGSet_n(d), \Psimple, Q_u)$ to that of its extension on $(\ConfBipGSet_n(d), \Pconfig, Q_c)$ which we constructed in Definition~\ref{def: extension}. Since the Dirichlet form is invariant under translations, we can (and will) assume without loss of generality that $f$ is centered. 
We  
decompose the set of pairs of adjacent multigraphs into several categories. Let us define
\begin{align*}
&\Gamma_1:= \{(G_1',G_2')\in   \BipGSet_n(d)\times   \BipGSet_n(d):\, G_1'\sim G_2'\},\\
 &\Gamma_2:= \{(G_1',G_2')\in  \BipGSet_n(d)\times  \categ([1,2]):\, G_1'\sim G_2' \},\\
 &\Gamma_3:= \{(G_1',G_2'):\;(G_1',G_2')\in  \Perfmatch(k) \mbox{ for some $1\leq k\leq \cconst$}\},\\
 &\Gamma_4:= \{(G_1',G_2')\in  \categ([1,\cconst])\times  \categ([1,\cconst]):\, G_1'\sim G_2'\}\setminus \Gamma_3,\\
 &\Gamma_5:= \{(G_1',G_2')\in   \categ([1,\cconst])\times  \Ugly(\cconst):\, G_1'\sim G_2'\},\\
 &\Gamma_6:=\{(G_1',G_2')\in  \Ugly(\cconst)\times \Ugly(\cconst):\, G_1'\sim G_2'\},
\end{align*}
where $\Ugly(\cconst)$ and $\Perfmatch(k)$ were introduced in Definitions~\ref{def: category} and \ref{def: perfect pair}, respectively. 
Note that given $G_1\in  \BipGSet_n(d)$, any multigraph $G_2\in \ConfBipGSet_n(d)\setminus  \BipGSet_n(d)$ adjacent to $G_1$ necessarily belongs to $\categ([1,2])$. Therefore, the above partition, up to changing the order of graphs, covers all pairs of adjacent elements from $\ConfBipGSet_n(d)$. 

We start by estimating the part of the Dirichlet form of $\tilde f$ involving simple graphs. Recall that $\tilde f(G)$ is only random when $G\in \categ([1,\cconst])$. Therefore, the expectation in 
the second part of the lemma below refers to the randomness of $\tilde f(G_2')$ when $(G_1',G_2')\in \Gamma_2$. 
\begin{lemma}\label{lem: Dirichlet-gamma1}
Let $f:  \BipGSet_n(d)\to \R$ be of mean zero and $\tilde f$ be its randomized extension from Definition~\ref{def: extension}. 
Then
$$
\sum_{(G_1',G_2')\in \Gamma_1} \Pconfig(G_1')Q_c(G_1',G_2') \big(\tilde f(G_1')-\tilde f(G_2')\big)^2
\leq 20\Dir_{\Psimple}(f,f), 
$$
and 
\begin{align*}
\Exp\sum_{(G_1',G_2')\in \Gamma_2} &\Pconfig(G_1')Q_c(G_1',G_2') \big(\tilde f(G_1')-\tilde f(G_2')\big)^2\\
&\leq \Dir_{\Psimple}(f,f)+2\sum_{(G_1',G_2')\in  \BipGSet_n(d) \times  \categ(2)}  \frac{\Pconfig(G_1')Q_c(G_1',G_2')}{|\SNeigh(G_2')|}  \sum_{G\in \SNeigh(G_2')} \big(  f(G_1')- f(G)\big)^2\\
&+\frac{16}{(nd-1)} \sum_{G_2'\in \categ([1,2])}  \frac{\Pconfig(G_2')}{|\SNeigh(G_2')|^2} \sum_{G,G'\in \SNeigh(G_2')} \big( f(G)- f(G')\big)^2.
\end{align*}
\end{lemma}
\begin{proof}
 For $(G_1',G_2')\in \Gamma_2$, we can write 
\begin{align*}
\Exp \big(\tilde f(G_1')-\tilde f(G_2')\big)^2&= \Big(\tilde f(G_1')- \frac{1}{|\SNeigh(G_2')|}  \sum_{G\in \SNeigh(G_2')}f(G)\Big)^2 \\
&\quad + \frac{1}{|\SNeigh(G_2')|}  \sum_{G\in \SNeigh(G_2')} \Big(f(G)-\frac{1}{|\SNeigh(G_2')|}  \sum_{G'\in \SNeigh(G_2')}f(G')\Big)^2.
\end{align*}
Using the Cauchy--Schwarz inequality and the reversibility, we get
\begin{align}
\alpha_2:&=\Exp\sum_{(G_1',G_2')\in \Gamma_2} \Pconfig(G_1')Q_c(G_1',G_2') \big(\tilde f(G_1')-\tilde f(G_2')\big)^2\nonumber\\
&\leq \sum_{(G_1',G_2')\in \Gamma_2} \frac{\Pconfig(G_1')Q_c(G_1',G_2') }{|\SNeigh(G_2')|}  \sum_{G_2\in \SNeigh(G_2')} \big( \tilde f(G_1')- f(G_2)\big)^2\nonumber\\
&\quad+  \sum_{(G_1',G_2')\in \Gamma_2} \frac{\Pconfig(G_2')Q_c(G_2',G_1') }{|\SNeigh(G_2')|^2} \sum_{G,G'\in \SNeigh(G_2')} \big( f(G)- f(G')\big)^2.\nonumber
\end{align}
Note that given $(G_1',G_2')\in \Gamma_2$, the category number of $G_2'$ is at most $2$ and thus 
$$
Q_c(G_2',G_1')\leq \frac{8}{nd(nd-1)},
$$
where we used that $G_2'$ has at most $2$ multiedges of multiplicity $2$ each. Moreover, any such $G_2'$ has at most $nd$ adjacent simple graphs $G_1'$. 
Therefore, we have
\begin{align*}
\alpha_2&\leq \sum_{(G_1',G_2')\in  \BipGSet_n(d) \times  \categ([1,2])}  \frac{\Pconfig(G_1')Q_c(G_1',G_2')}{|\SNeigh(G_2')|}  \sum_{G\in \SNeigh(G_2')} \big( \tilde f(G_1')- f(G)\big)^2\\
&\quad+\frac{8}{(nd-1)} \sum_{G_2'\in \categ([1,2])}  \frac{\Pconfig(G_2')}{|\SNeigh(G_2')|^2} \sum_{G,G'\in \SNeigh(G_2')} \big( f(G)- f(G')\big)^2.
\end{align*}
Note that if $(G_1',G_2')\in  \BipGSet_n(d)\times  \categ(1)$, and the two graphs are adjacent
then necessarily $G_1'$ belongs to the $s$-neighborhood of $G_2'$. 
Moreover, in such a case, using Lemma~\ref{l: 2-9857-2985} we have 
\begin{equation}\label{eq: dir-simple}
\frac{2}{nd-1} \frac{\Pconfig(G_2')}{|\SNeigh(G_2')|}\leq \Pconfig(G_1') Q_c(G_1',G_2')\leq \frac{4}{nd-1} \frac{\Pconfig(G_2')}{|\SNeigh(G_2')|},
\end{equation}
where we also used that $G_2'$ has exactly one multiedge of multiplicity $2$ and thus $\Pconfig(G_1') = 2\Pconfig(G_2')$. 
Therefore, we deduce that 
\begin{equation}\label{eq: dirichlet-gamma2}
\begin{split}
\alpha_2&\leq \sum_{(G_1',G_2')\in  \BipGSet_n(d) \times  \categ(2)}  \frac{\Pconfig(G_1')Q_c(G_1',G_2')}{|\SNeigh(G_2')|}  \sum_{G\in \SNeigh(G_2')} \big( \tilde f(G_1')- f(G)\big)^2\\
&\quad +\frac{4}{(nd-1)}\sum_{G_2'\in \categ(1)}\frac{\Pconfig(G_2')}{|\SNeigh(G_2')|^2} \sum_{G,G'\in \SNeigh(G_2')} \big( \tilde f(G)- f(G')\big)^2\\
&\quad+\frac{8}{(nd-1)} \sum_{G_2'\in \categ([1,2])}  \frac{\Pconfig(G_2')}{|\SNeigh(G_2')|^2} \sum_{G,G'\in \SNeigh(G_2')} \big( f(G)- f(G')\big)^2.
\end{split}
\end{equation}
Note that if \eqref{eq: condition-blow up-def} is satisfied then $\tilde f=f$ on $ \BipGSet_n(d)$ and the statement of the lemma follows from the above.
 Thus, for the remainder of the proof we will suppose that \eqref{eq: condition-blow up-def} is violated which, together with the mean zero assumption on $f$ and Lemma~\ref{l: 972-987-9587}, implies
$$
\sum_{G\in \Expset(\zconst)} \Psimple(G)f(G)^2 \leq \frac{1}{2^6} \sum_{G\in \BipGSet_n(d)\setminus\Expset(\zconst)} \Psimple(G)f(G)^2.
$$
Using this together with Lemma~\ref{lem: size-bad graphs}, we can write
\begin{align*}
 &\sum_{\underset{\antiexp(G_1')>\antiexp(G_2')}{(G_1',G_2')\in \BipGSet_n(d)\times\BipGSet_n(d)}  } \Psimple(G_1')Q_c(G_1',G_2')  f(G_2')^2
\\
&\hspace{1cm}= \frac{2}{nd(nd-1)}\sum_{G_2'\in \BipGSet_n(d)} \Psimple(G_2') f(G_2')^2 \vert\{G_1'\in  \BipGSet_n(d):\, G_1'\sim G_2',\, \antiexp(G_1')>\antiexp(G_2')\}\vert\\ 
 &\hspace{1cm}\leq \frac{2nd^5}{nd(nd-1)} \sum_{G_2'\in \BipGSet_n(d)} \Psimple(G_2') f(G_2')^2\\
 &\hspace{1cm}\leq \frac{4nd^5}{nd(nd-1)}\sum_{G_1'\in \BipGSet_n(d)\setminus\Expset(\zconst)} \Psimple(G_1')f(G_1')^2\\
 &\hspace{1cm}\leq  \frac{4nd^5}{\zconst(nd-2d^4)} \sum_{\underset{\antiexp(G_1')>\antiexp(G_2')}{(G_1',G_2')\in \BipGSet_n(d)\times\BipGSet_n(d)}  }  \Psimple(G_1')Q_c(G_1',G_2') f(G_1')^2,
 \end{align*}
 where in the last line we used that, by Lemma~\ref{lem: size-bad graphs}, there are at least $\frac12 \zconst(nd-2d^4)$ adjacent graphs $G_2'$ to a given graph $G_1'\in \BipGSet_n(d)\setminus\Expset(\zconst)$ with $\antiexp(G_1')>\antiexp(G_2')$. By the choice of $\zconst$ in \eqref{eq: parameters} and the assumption on $d$ in \eqref{eq: assumption-d}, 
 this implies that \\
\begin{equation}\label{eq: Dirichlet-bad graphs}
  \sum_{\underset{\antiexp(G_1')>\antiexp(G_2')}{G_1',G_2'\in \BipGSet_n(d)}  } \Psimple(G_1')Q_c(G_1',G_2') \big(f(G_1')-f(G_2')\big)^2
  \geq \frac12   \sum_{\underset{\antiexp(G_1')>\antiexp(G_2')}{G_1',G_2'\in \BipGSet_n(d)}  } \Psimple(G_1')Q_c(G_1',G_2') f(G_1')^2,
 \end{equation}
 where we made use of Cauchy-Schwarz inequality. 
At the level of intuition, the above relation must hold true since violation of \eqref{eq: condition-blow up-def}
essentially means that the function $f$ tends to take greater values on the set $\BipGSet_n(d)\setminus\Expset(\zconst)$
than on the set $\Expset(\zconst)$.
Now, using reversibility, we write 
\begin{align*}
\alpha_1:&= \sum_{(G_1',G_2')\in \Gamma_1} \Pconfig(G_1')Q_c(G_1',G_2') \big(\tilde f(G_1')-\tilde f(G_2')\big)^2\\
&= \sum_{(G_1',G_2')\in \BipGSet_n(d)\setminus\Expset(\zconst)\times\BipGSet_n(d)\setminus\Expset(\zconst)  }\Pconfig(G_1')Q_c(G_1',G_2') \big(\tilde f(G_1')-\tilde f(G_2')\big)^2\\
&\hspace{1cm}+ \sum_{(G_1',G_2')\in\Expset(\zconst)\times\Expset(\zconst)  } \Pconfig(G_1')Q_c(G_1',G_2') \big(\tilde f(G_1')-\tilde f(G_2')\big)^2\\
&\hspace{1cm}+2 \sum_{(G_1',G_2')\in \BipGSet_n(d)\setminus\Expset(\zconst) \times\Expset(\zconst)} \Pconfig(G_1')Q_c(G_1',G_2') \big(\tilde f(G_1')-\tilde f(G_2')\big)^2.
\end{align*}
Note that $Q_c$ and $Q_u$ coincide on $\BipGSet_n(d)\times \BipGSet_n(d)$. Moreover, using the definition of $\tilde f$ and that $\Pconfig(G_1')= \Psimple(G_1')\Pconfig\big(\BipGSet_n(d)\big)\leq \frac12\Psimple(G_1')$, $G_1'\in \BipGSet_n(d)$, we get 
\begin{align*}
\alpha_1&\leq 4\Dir_{\Psimple}(f,f) + 4 \sum_{(G_1',G_2')\in \BipGSet_n(d)\setminus\Expset(\zconst) \times\Expset(\zconst)   } \Psimple(G_1')Q_c(G_1',G_2')  f(G_1')^2, 
\end{align*}
where we made use of the inequality 
\begin{equation}\label{eq: dir-1}
\Bigg(\frac{f(G_1')}{\sqrt{\Pconfig\big(\BipGSet_n(d)\big)}}- f(G_2')\Bigg)^2\leq 
\frac{2}{\Pconfig(\BipGSet_n(d))} f(G_1')^2+ 2 \big( f(G_1')- f(G_2')\big)^2.
\end{equation}
It remains to apply \eqref{eq: Dirichlet-bad graphs} to obtain the required bound for the sum over $\Gamma_1$. 
To get the bound on $\Gamma_2$, we combine  \eqref{eq: dir-1} with \eqref{eq: dirichlet-gamma2} to write 
\begin{align*}
\alpha_2&\leq 
2\sum_{(G_1',G_2')\in  \BipGSet_n(d) \times  \categ(2)}
\frac{\Pconfig(G_1')Q_c(G_1',G_2')}{|\SNeigh(G_2')|}  \sum_{G\in \SNeigh(G_2')}
\big( f(G_1')- f(G)\big)^2\\
&\hspace{1cm}+\frac{2}{\Pconfig(\BipGSet_n(d))} \sum_{(G_1',G_2')\in  \BipGSet_n(d)\setminus \Expset(\zconst) \times  \categ(2)}
\sum_{G\in \SNeigh(G_2')}
\frac{\Pconfig(G_1')Q_c(G_1',G_2')}{|\SNeigh(G_2')|}\,f(G_1')^2\\
&\hspace{1cm}+\frac{8}{(nd-1)}\sum_{G_2'\in \categ(1)}\frac{\Pconfig(G_2')}{|\SNeigh(G_2')|^2} \sum_{G,G'\in \SNeigh(G_2')} \big(
f(G)- f(G')\big)^2\\
&\hspace{1cm} +\frac{4}{(nd-1)}\sum_{G_2'\in \categ(1)}\frac{\Pconfig(G_2')}{|\SNeigh(G_2')|} \sum_{G\in\SNeigh(G_2')\setminus \Expset(\zconst)} 
\frac{2}{\Pconfig(\BipGSet_n(d))} f(G)^2\\
&\hspace{1cm}+\frac{8}{(nd-1)} \sum_{G_2'\in \categ([1,2])}  \frac{\Pconfig(G_2')}{|\SNeigh(G_2')|^2} \sum_{G,G'\in \SNeigh(G_2')} \big( f(G)- f(G')\big)^2\\
&\leq
6\sum_{(G_1',G_2')\in  \BipGSet_n(d)\setminus \Expset(\zconst)   \times\categ([1,2])} \Psimple(G_1')Q_c(G_1',G_2')f(G_1')^2 \\
&\hspace{1cm}+2\sum_{(G_1',G_2')\in  \BipGSet_n(d) \times  \categ(2)}  \frac{\Pconfig(G_1')Q_c(G_1',G_2')}{|\SNeigh(G_2')|}  \sum_{G\in \SNeigh(G_2')} \big(  f(G_1')- f(G)\big)^2\\
&\hspace{1cm}+\frac{16}{(nd-1)} \sum_{G_2'\in \categ([1,2])}  \frac{\Pconfig(G_2')}{|\SNeigh(G_2')|^2} \sum_{G,G'\in \SNeigh(G_2')} \big( f(G)- f(G')\big)^2,
\end{align*}
where we also made use of \eqref{eq: dir-simple}. 
To bound the first term above, note that given $G_1'\in\BipGSet_n(d)\setminus \Expset(\zconst)$, there are at most $nd^3$ multigraphs in $\categ([1,2])$ adjacent to it. Therefore, 
\begin{align*}
&\sum_{(G_1',G_2')\in  \BipGSet_n(d)\setminus \Expset(\zconst)   \times\categ([1,2])} \Psimple(G_1')Q_c(G_1',G_2')f(G_1')^2
\\
&\hspace{1cm}\leq \frac{2nd^3}{nd(nd-1)} \sum_{G_1'\in  \BipGSet_n(d)\setminus \Expset(\zconst) }  \Psimple(G_1')f(G_1')^2\\
&\hspace{1cm}\leq \frac{2nd^3}{\zconst(nd-2d^4)} \sum_{\underset{\antiexp(G_2')<\antiexp(G_1')}{(G_1',G_2')\in  \BipGSet_n(d)\setminus \Expset(\zconst)\times\BipGSet_n(d)}}  \Psimple(G_1')Q_c(G_1',G_2')f(G_1')^2,
\end{align*}
where we made use of Lemma~\ref{lem: size-bad graphs} in the last line. It remains to use \eqref{eq: Dirichlet-bad graphs},  \eqref{eq: parameters} and \eqref{eq: assumption-d} to finish the proof. 
\end{proof}

The next lemma is a preparatory step for sets $\Gamma_3,\Gamma_4$ and $\Gamma_5$,
where we apply standard functional inequalities to estimate $\Exp (\tilde f(G_1')-\tilde f(G_2'))^2$,
with the couple $(G_1',G_2')$ belonging to an appropriate $\Gamma_i$.
In regard to sets $\Gamma_3,\Gamma_4$, we will apply the crucial property of the Gaussian field used in our randomized extension,
namely, that components of the field for adjacent graphs are highly correlated.
\begin{lemma}\label{lem: pre-upper-bound-dir}
Let $f:  \BipGSet_n(d)\to \R$ be of mean zero and let $\tilde f$ be its randomized extension defined previously. 
The following assertions hold:
\begin{itemize}
\item[i.] If $(G_1',G_2')\in  \Gamma_3$ and $\psi_{G_1',G_2'}: \SNeigh(G_1')\to \SNeigh(G_2')$ denotes the bijection from Proposition~\ref{prop: bijection}, then 
\begin{align*}
\Exp \big(\tilde f(G_1')-\tilde f(G_2')\big)^2& \leq \frac{5}{|\SNeigh(G_1')|} \sum_{G\in \SNeigh(G_1')} \big( f(G)- f\circ \psi_{G_1',G_2'}(G)\big)^2 \\
 &\hspace{1cm}+\frac{8}{\gconst |\SNeigh(G_1')|^2} \sum_{G,G'\in \SNeigh(G_1')} \big(f(G)-f(G')\big)^2 \\
 &\hspace{1cm}+\frac{8}{\gconst |\SNeigh(G_2')|^2}\sum_{G,G'\in \SNeigh(G_2')} \big(f(G)-f(G')\big)^2.
\end{align*} 
\item[ii.] If $(G_1',G_2')\in  \Gamma_4$, then 
\begin{align*}
\Exp \big(\tilde f(G_1')-\tilde f(G_2')\big)^2& \leq \frac{5}{|\SNeigh(G_1')||\SNeigh(G_2')|} \sum_{(G_1,G_2)\in \SNeigh(G_1')\times \SNeigh(G_2')} \big( f(G_1)- f(G_2)\big)^2 \\
 &\hspace{1cm}+\frac{8}{\gconst |\SNeigh(G_1')|^2} \sum_{G,G'\in \SNeigh(G_1')} \big(f(G)-f(G')\big)^2 \\
 &\hspace{1cm}+\frac{8}{\gconst |\SNeigh(G_2')|^2}\sum_{G,G'\in \SNeigh(G_2')} \big(f(G)-f(G')\big)^2.
\end{align*}
\item[iii.] If $(G_1',G_2')\in \Gamma_5$, then 
\begin{align*}
\Exp &\big(\tilde f(G_1')-\tilde f(G_2')\big)^2\\
&\leq \frac{1}{|\SNeigh(G_1')|}  \sum_{G\in \SNeigh(G_1')} f(G)^2+  \frac{1}{|\SNeigh(G_1')|^2} \sum_{G,G'\in \SNeigh(G_1')} \big( f(G)- f(G')\big)^2.
\end{align*}
\end{itemize}
\end{lemma}
\begin{proof}\hspace{0cm}\\
\begin{itemize}
\item[i.] Let $(G_1',G_2')\in \Gamma_3$. 
Using the definition of $\tilde f$ and the second property in Proposition~\ref{prop: 29870536509385}, we have 
\begin{align*}
\Exp &\big(\tilde f(G_1')-\tilde f(G_2')\big)^2\\
&=  \big(h(G_1')-h(G_2')\big)^2+w(G_1')+w(G_2')-2\sqrt{w(G_1')}\,\sqrt{w(G_2')}\,\Exp(\xi_{G_1'}\xi_{G_2'})
\\
&\leq \big(h(G_1')-h(G_2')\big)^2 + \Big(\sqrt{w(G_1')}-\sqrt{w(G_2')}\Big)^2 +\frac{8}{\gconst} \sqrt{w(G_1')w(G_2')}.
\end{align*}
Note that $w(G_i')= \frac{1}{|\SNeigh(G_i')|} \Big\Vert (f(G')-h(G_i'))_{G'\in \SNeigh(G_i')} \Big\Vert_2^2$ for $i=1,2$, and by the Cauchy-Schwarz inequality we have 
\begin{equation}\label{eq: cauchy-schwarz}
w(G_i')\leq \frac{1}{|\SNeigh(G_i')|^2}\sum_{G,G'\in \SNeigh(G_i')} \Big(f(G)-f(G')\Big)^2.
\end{equation}
Using the triangle inequality and that $|\SNeigh(G_1')|=|\SNeigh(G_2')|$ (recall there is a bijective mapping between
the $s$--neighborhoods), we have
\begin{align*}
\Big(&\sqrt{w(G_1')}-\sqrt{w(G_2')}\Big)^2\\
&\leq \frac{1}{|\SNeigh(G_1')|} \sum_{G'\in \SNeigh(G_1')} \big( f(G')- f\circ \psi_{G_1',G_2'}(G')+ h(G_1')-h(G_2')\big)^2\\
&\leq 2\big(h(G_1')-h(G_2')\big)^2+ \frac{2}{|\SNeigh(G_1')|} \sum_{G'\in \SNeigh(G_1')} \big( f(G')- f\circ \psi_{G_1',G_2'}(G')\big)^2.
\end{align*}
Now, writing $h(G_1')-h(G_2')=  \frac{1}{|\SNeigh(G_1')|} \sum_{G'\in \SNeigh(G_1')} \big(f(G')-f\circ\psi_{G_1',G_2'}(G')\big)$ and using
the Cauchy--Schwarz inequality, we deduce  
\begin{align*}
\Exp &\big(\tilde f(G_1')-\tilde f(G_2')\big)^2 \\
&\leq \frac{5}{|\SNeigh(G_1')|} \sum_{G'\in \SNeigh(G_1')} \big( f(G')- f\circ \psi_{G_1',G_2'}(G')\big)^2 +\frac{8}{\gconst}\sqrt{w(G_1')w(G_2')}.
\end{align*}
It remains to use that $ \sqrt{w(G_1')w(G_2')}\leq \frac12 \big(w(G_1')+w(G_2')\big)$, and use \eqref{eq: cauchy-schwarz} to finish the proof. 
\item[ii.] For $(G_1',G_2')\in \Gamma_4$, the proof follows similar lines to the one above. One only needs to note that 
$$w(G_i')= \frac{1}{|\SNeigh(G_1')| |\SNeigh(G_2')|} \Big\Vert (f(G)-h(G_i'))_{(G,G')\in \SNeigh(G_i')\times \SNeigh(G_{3-i}')} \Big\Vert_2^2,$$ for $i=1,2$, and use the triangle inequality 
to get 
\begin{align*}
&\Big(\sqrt{w(G_1')}-\sqrt{w(G_2')}\Big)^2\\
&\leq \frac{1}{|\SNeigh(G_1')| |\SNeigh(G_2')|}
\sum_{(G_1,G_2)\in \SNeigh(G_1')\times \SNeigh(G_2')} \big( f(G_1)- f(G_2)- h(G_1')+h(G_2')\big)^2\\
&\\
&\leq 2\big(h(G_1')-h(G_2')\big)^2\\
&\quad + \frac{2}{|\SNeigh(G_1')||\SNeigh(G_2')|} \sum_{(G_1,G_2)\in \SNeigh(G_1')\times \SNeigh(G_2')} \big( f(G_1)- f(G_2)\big)^2.
\end{align*}
Writing $h(G_1')-h(G_2')=  \frac{1}{|\SNeigh(G_1')||\SNeigh(G_2')|} \sum_{(G_1,G_2)\in \SNeigh(G_1')\times \SNeigh(G_2')} \big(f(G_1)-f(G_2)\big)$ and using the Cauchy--Schwarz inequality, we deduce 
\begin{align*}
\Exp &\big(\tilde f(G_1')-\tilde f(G_2')\big)^2 \\
&\leq \frac{5}{|\SNeigh(G_1')||\SNeigh(G_2')|} \sum_{(G_1,G_2)\in \SNeigh(G_1')\times \SNeigh(G_2')} \big( f(G_1)- f(G_2)\big)^2 \\
&\quad +\frac{8}{\gconst} \sqrt{w(G_1')w(G_2')},
\end{align*}
where we used the definition of $\tilde f$ and the second property in Proposition~\ref{prop: 29870536509385}. The final step of the proof is identical to the previous case. 
\item[iii.] Let $(G_1',G_2')\in \Gamma_5$. By the definition of $\tilde f$, we can write 
$$
\Exp \big(\tilde f(G_1')-\tilde f(G_2')\big)^2= h(G_1')^2 + w(G_1'). 
$$
It remains to use the Cauchy--Schwarz inequality as before. 
\end{itemize}
\end{proof}

The next lemma uses the above calculations to provide a preliminary bound for the Dirichlet form of the extension. 

\begin{lemma}\label{lem: upper-bound-dir-good}
Let $f: \BipGSet_n(d)\to \R$ satisfy $\Exp_{\Psimple} f=0$, and
let $\tilde f: \ConfBipGSet_n(d)\to \R$ be its randomized extension. Then
\begin{align*}
\Exp\, \Dir_{\Pconfig}(\tilde f, \tilde f)&\leq C\Dir_{\Psimple}(f,f) + C\frac{d}{ n^2}\Var_{\Psimple}(f) \\
&\hspace{1cm}+C\sum_{(G_1',G_2')\in  \Gamma_4} \frac{\Pconfig(G_1')Q_c(G_1',G_2')}{|\SNeigh(G_1')|\cdot |\SNeigh(G_2')|}  \sum_{(G,G')\in \SNeigh(G_1')\times \SNeigh(G_2')} \big( f(G)- f(G')\big)^2\\
&\hspace{1cm}+ \frac{Cd}{n}\sum_{G_1'\in  \categ([1,\cconst])}\frac{\Pconfig (G_1')}{  |\SNeigh(G_1')|^2}  \sum_{G,G'\in \SNeigh(G_1')} \big(f(G)-f(G')\big)^2\\ 
&\hspace{1cm}+C\sum_{(G_1',G_2')\in  \BipGSet_n(d) \times  \categ(2)}
\frac{\Pconfig(G_1')Q_c(G_1',G_2')}{|\SNeigh(G_2')|}  \sum_{G\in \SNeigh(G_2')} \big( f(G_1')- f(G)\big)^2,
\end{align*}
where $C>0$ is a universal constant. 
\end{lemma}
\begin{proof}
Using the chain reversibility, we start by writing 
\begin{align*}
\Exp\,\Dir_{\Pconfig}(\tilde f, \tilde f)
&=  \sum_{(G_1',G_2')\in \Gamma_2\cup \Gamma_5} \Pconfig(G_1')Q_c(G_1',G_2') \Exp\big(\tilde f(G_1')-\tilde f(G_2')\big)^2\\
&+\frac12 \sum_{(G_1',G_2')\in \Gamma_1\cup \Gamma_3\cup \Gamma_4} \Pconfig(G_1')Q_c(G_1',G_2') \Exp\big(\tilde f(G_1')-\tilde f(G_2')\big)^2.
\end{align*}
For $i=1,\ldots,5$, we denote by $\alpha_i$ the corresponding sum over $\Gamma_i$ in the right hand side above. 
Note that a bound on $\alpha_1$ and $\alpha_2$ was provided in Lemma~\ref{lem: Dirichlet-gamma1}. We will estimate each of the remaining $\alpha_i$ separately. 

Given $(G_1',G_2')\in \Gamma_3$, let $\psi_{G_1',G_2'}: \SNeigh(G_1')\to \SNeigh(G_2')$ be the bijection from Propostion~\ref{prop: bijection}. 
Recall from Lemma~\ref{lem: nice-multigraph-prob and transition} that if $G_1'\in \categ(k)$ then 
$$
\Pconfig(G_1')= \frac{(d!)^{2n}}{(nd)! 2^{k}}= \Pconfig\big(\BipGSet_n(d)\big)\frac{\Psimple(G)}{2^k},
$$
for any $G\in \BipGSet_n(d)$. 
Moreover, it follows from the definition of $\Perfmatch(k)$ that $Q_c(G_1',G_2')= \big(nd(nd-1)/2\big)^{-1}$ for every $(G_1',G_2')\in \Perfmatch(k)$, $1\leq k\leq \cconst$.  
Following Lemma~\ref{lem: pre-upper-bound-dir} and using Lemma~\ref{l: 2-9857-2985}, 
we first estimate 
\begin{align*}
\sum_{(G_1',G_2')\in \Gamma_3} &\frac{5\Pconfig (G_1') Q_c(G_1',G_2')}{|\SNeigh(G_1')|} \sum_{G'\in \SNeigh(G_1')} \big( f(G')- f\circ \psi_{G_1',G_2'}(G')\big)^2\\
&\leq 10\Pconfig\big(\BipGSet_n(d)\big) \sum_{(G,G')\in  \Gamma_1}  \Psimple(G) Q_u(G,G')  \big(f(G)-f(G')\big)^2 \sum_{k=1}^{\cconst}\frac{\gamma_k(G,G')}{2^{k} (nd)^k}, 
\end{align*}
where $\gamma_k(G,G')$ denotes the number of pairs $(G_1',G_2')\in \Perfmatch(k)$ such that $G\in \SNeigh(G_1')$ and $G'=\psi_{G_1',G_2'}(G)\in \SNeigh(G_2')$. 
Combining Proposition~\ref{prop: matching perfect pairs} and Lemma~\ref{l: 98562098724}, we get that $\gamma_k(G,G')\leq \frac{(d-1)^{2k}}{k!} (nd)^k$ and hence 
\begin{align*}
\sum_{(G_1',G_2')\in \Gamma_3} &\frac{5\Pconfig (G_1') Q_c(G_1',G_2')}{|\SNeigh(G_1')|} \sum_{G'\in \SNeigh(G_1')} \big( f(G')- f\circ \psi_{G_1',G_2'}(G')\big)^2\\
&\leq 10\Pconfig\big(\BipGSet_n(d)\big) \sum_{(G,G')\in  \Gamma_1}  \Psimple(G) Q_u(G,G')  \big(f(G)-f(G')\big)^2 \sum_{k=1}^{\cconst}\frac{(d-1)^{2k}}{2^{k} k!}\\
&\leq 40 \Dir_{\Psimple}(f,f),
\end{align*}
where we made use of \eqref{eq: prob-simple}. 
Note that for a given $G_1'$ we have $\sum_{G_2'\sim G_1'} Q_c(G_1',G_2')\leq 1$. Using this with Lemma~\ref{lem: pre-upper-bound-dir}, we obtain from the above 
$$
\alpha_3\leq 
40\, \Dir_{\Psimple}(f,f)+\frac{16}{\gconst}\sum_{G_1'\in\categ([1,\cconst])}\frac{\Pconfig (G_1')}{|\SNeigh(G_1')|^2}  \sum_{G,G'\in \SNeigh(G_1')} \big(f(G)-f(G')\big)^2.
$$

To estimate the sum over $\Gamma_4$, we use again Lemma~\ref{lem: pre-upper-bound-dir}  with the fact that $\sum_{G_2'\sim G_1'} Q_c(G_1',G_2')\leq 1$ for a given $G_1'$, and write 
\begin{align*}
 \alpha_4&\leq  5\sum_{(G_1',G_2')\in \Gamma_4} \frac{\Pconfig(G_1')Q_c(G_1',G_2')}{|\SNeigh(G_1')|\cdot |\SNeigh(G_2')|}  \sum_{(G,G')\in \SNeigh(G_1')\times \SNeigh(G_2')} \big( f(G)- f(G')\big)^2\\
&+\frac{16}{\gconst} \sum_{G_1'\in \categ([1,\cconst])}\frac{\Pconfig (G_1')}{  |\SNeigh(G_1')|^2}  \sum_{G,G'\in \SNeigh(G_1')} \big(f(G)-f(G')\big)^2.
\end{align*}

Finally, we estimate $\alpha_5$. 
Take any $G_1'\in \categ(k)$, for some $1\leq k\leq \cconst$. 
Recall that by Lemma~\ref{lem: nice-multigraph-prob and transition}, if $k\leq \cconst-2$, there are at most $4k d^3$ multigraphs $G_2'\in\Ugly(\cconst)$ adjacent to $G_1'$ while if $k\in [\cconst-1,\cconst]$, then there are at most $d^3(4k+n)$ such adjacent multigraphs. 
Thus, using Lemma~\ref{lem: nice-multigraph-prob and transition} and  Lemma~\ref{lem: pre-upper-bound-dir}, we get 
\begin{align*} 
\alpha_5&\leq \sum_{(G_1',G_2')\in \Gamma_5}\mbox{$\frac{2\Pconfig(G_1')}{nd(nd-1)|\SNeigh(G_1')|}$}\bigg(  \sum_{G\in \SNeigh(G_1')} f(G)^2+  \frac{1}{|\SNeigh(G_1')|} \sum_{G,G'\in \SNeigh(G_1')} \big( f(G)- f(G')\big)^2\bigg)\\
&\leq\frac{8d^3 \Pconfig\big(\BipGSet_n(d)\big)}{ nd(nd-1)} \sum_{k\in [1, \cconst-2]}\frac{ k}{2^{k}}\sum_{G_1'\in \categ(k)}   \sum_{G\in \SNeigh(G_1')} \frac{\Psimple(G)}{|\SNeigh(G_1')|}  f(G)^2\\
&+\frac{2d^3 \Pconfig\big(\BipGSet_n(d)\big)}{ nd(nd-1)} \sum_{k\in [\cconst-1, \cconst]}\frac{ (4k+n)}{2^{k}}\sum_{G_1'\in \categ(k)}   \sum_{G\in \SNeigh(G_1')} \frac{\Psimple(G)}{|\SNeigh(G_1')|}  f(G)^2\\
&+\frac{2d^3(n+4\cconst)}{nd(nd-1)}\sum_{G_1'\in \categ([1, \cconst])}  \frac{\Pconfig(G_1')}{|\SNeigh(G_1')|^2} \sum_{G,G'\in \SNeigh(G_1')}   \big( f(G)- f(G')\big)^2\\
&\leq  \frac{16d^3 \Pconfig\big(\BipGSet_n(d)\big)}{ nd(nd-1)}  \sum_{G\in \BipGSet_n(d)} \Psimple(G)  f(G)^2 \sum_{k=1}^{\cconst-2} \frac{(d-1)^{2k}}{2^k(k-1)!}\\
&+\frac{4d^3 \Pconfig\big(\BipGSet_n(d)\big)}{ nd(nd-1)}  \sum_{G\in \BipGSet_n(d)} \Psimple(G)  f(G)^2 \sum_{k\in [\cconst-1, \cconst]}\frac{(4k+n)(d-1)^{2k}}{2^k k!}\\
&+\frac{4d^2}{(nd-1)}\sum_{G_1'\in \categ([1, \cconst])}  \frac{\Pconfig(G_1')}{|\SNeigh(G_1')|^2} \sum_{G,G'\in \SNeigh(G_1')}   \big( f(G)- f(G')\big)^2,
\end{align*}
where we applied Lemmas~\ref{l: 2-9857-2985} and \ref{l: 98562098724}, and used that $4\cconst\leq n$ to get the last inequality. 
Finally, note that by the choice of $\cconst$ in \eqref{eq: parameters}, we have $\sum_{k\in [\cconst-1, \cconst]}\frac{(4k+n)(d-1)^{2k}}{2^k k!}\leq 1$. 
In view of this, we finish the proof after using \eqref{eq: prob-simple}, the assumption on $d$ in \eqref{eq: assumption-d} and the choice of $\gconst$ in \eqref{eq: parameters}, and putting together the above estimates.
\end{proof}

In the sequel, we will bound each of the last three terms in the right hand side of the expression in
Lemma~\ref{lem: upper-bound-dir-good}. To this aim, we will make use of the connections constructed (see Section~\ref{sec: connections}) 
Note that those three terms correspond to three different types of connections (tuples): the ones connecting simple graphs belonging to the $s$-neighborhood of two adjacent multigraphs, those connecting two simple graphs within the same $s$-neighborhood, and the one 
between a simple graph and the $s$-neighborhood of an adjacent multigraph.  
We start with the first type and refer the reader to Definition~\ref{def: connection1}. 
\begin{lemma}[Type $1$ tuples with small $|I_{NS}|$]\label{lem: upper bound-small r-2}
Let $f: \BipGSet_n(d)\to \R$ satisfy $\Exp_{\Psimple} f=0$ and $r_0$ be as in \eqref{eq: parameters}. 
Then we have 
$$
\sum_{\underset{\underset{u\in \{A,B\}}{0\leq r\leq r_0}}{1\leq k_1,k_2\leq \cconst}}  \sum_{\underset{\text{of subtype $1$--$(k_1,k_2,u,r)$}}{(G_1,G_1',G_2,G_2')}} \frac{\Pconfig(G_1')Q_c(G_1',G_2')}{|\SNeigh(G_1')|\cdot |\SNeigh(G_2')|}  \big(f(G_1)-f(G_2)\big)^2\leq C \Dir_{\Psimple}(f,f),
$$
where $C>0$ is a universal constant. 
\end{lemma}
\begin{proof}
Let $P_\tuple$ be the  connection corresponding to a tuple $\tuple=(G_1,G_1',G_2,G_2')$ and $L_\tuple$ its length. 
Then we can write 
\begin{align*}
\alpha:&=\sum_{\underset{\underset{u\in \{A,B\}}{0\leq r\leq r_0}}{1\leq k_1,k_2\leq \cconst}}  \sum_{\underset{\text{of subtype $1$--$(k_1,k_2,u,r)$}}{(G_1,G_1',G_2,G_2') }} \frac{\Pconfig(G_1')Q_c(G_1',G_2')}{|\SNeigh(G_1')|\cdot |\SNeigh(G_2')|}  \big(f(G_1)-f(G_2)\big)^2 \\
&=\sum_{\underset{\underset{u\in \{A,B\}}{0\leq r\leq r_0}}{1\leq k_1,k_2\leq \cconst}}  \sum_{\underset{\text{of subtype $1$--$(k_1,k_2,u,r)$}}{\tuple=(G_1,G_1',G_2,G_2') }}   \frac{\Pconfig(G_1')Q_c(G_1',G_2')}{|\SNeigh(G_1')|\cdot |\SNeigh(G_2')|}   \big(f(P_\tuple[0])-f(P_\tuple[L_\tuple])\big)^2.
\end{align*}
Using the Cauchy--Schwarz inequality, we get 
$$
\alpha
\leq\sum_{\underset{\underset{u\in \{A,B\}}{0\leq r\leq r_0}}{1\leq k_1,k_2\leq \cconst}}  \sum_{\underset{\text{of subtype $1$--$(k_1,k_2,u,r)$}}{\tuple=(G_1,G_1',G_2,G_2') }}   \frac{L_\tuple\, \Pconfig(G_1')Q_c(G_1',G_2')}{|\SNeigh(G_1')|\cdot |\SNeigh(G_2')|} \sum_{t=1}^{L_\tuple}\big(f(P_\tuple[t])-f(P_\tuple[t-1])\big)^2. 
$$
Recall that the category numbers of two adjacent multigraphs differ by at most $2$  and that, in view of Remark~\ref{rem: 02985720958709},
$L_\tuple\leq 2\max(k_1,k_2)+C_{\text{\tiny\ref{l: 2095203598305298 bis}}}(8r+4)$.
Moreover, note that 
$$
\Pconfig(G_1')Q_c(G_1',G_2')\leq \frac{8(d!)^{2n}}{2^{k_1}(nd)!\, nd(nd-1)}. 
$$
Using this  together with the estimates of Lemma~\ref{l: 2-9857-2985} and assuming that
$n$ is large enough, we deduce that for some universal constant $C>0$,
\begin{align*}
\alpha
&\leq C \sum_{\underset{\underset{u\in \{A,B\}}{0\leq r\leq r_0}}{\underset{\vert k_1-k_2\vert \leq 2}{1\leq k_1,k_2\leq \cconst} }} \, \sum_{\underset{\text{of subtype $1$--$(k_1,k_2,u,r)$}}{\tuple }} \frac{k_1 (d!)^{2n}}{2^{k_1}(nd)!\, nd(nd-1) (nd)^{k_1+k_2}}\sum_{t=1}^{L_\tuple}  \big(f(P_\tuple[t])-f(P_\tuple[t-1])\big)^2\\
&\leq C \sum_{\underset{\underset{u\in \{A,B\}}{0\leq r\leq r_0}}{\underset{\vert k_1-k_2\vert \leq 2}{1\leq k_1,k_2\leq \cconst} }} \frac{k_1 \Pconfig\big(\BipGSet_n(d)\big)}{2^{k_1}(nd)^{k_1+k_2}} \sum_{\substack{H,H'\in \BipGSet_n(d)\\ H\sim H'}} \Psimple(H)Q_u(H,H') \big(f(H)-f(H')\big)^2 \gamma_{k_1,k_2,u,r}(H,H'),
\end{align*}
where $\gamma_{k_1,k_2,u,r}(H,H')$ is the number of tuples $\tuple=(G_1,G_1',G_2,G_2')$ of subtype $1$--$(k_1,k_2,u,r)$ such that $(H,H')$ is a pair of adjacent graphs in $P_\tuple$. Applying Proposition~\ref{prop: number-tuples1}, we deduce 
\begin{align*}
\alpha
&\leq C'\, \Pconfig\big(\BipGSet_n(d)\big) \Dir_{\Psimple}(f,f) \sum_{\underset{\underset{u\in \{A,B\}}{0\leq r\leq r_0}}{1\leq k_1,k_2\leq \cconst, \vert k_1-k_2\vert \leq 2 }} \frac{d^{3+r+C_{\text{\tiny{\ref{l: 2095203598305298 bis}}}}(16r+8)}(k_1+k_2)^{r+3}}{n 2^{k_1} \lfloor k_1/2+k_2/2\rfloor !}
\, (d-1)^{r+k_1+k_2}\\
&\leq C''\, \frac{d^{5+2r_0+C_{\text{\tiny{\ref{l: 2095203598305298 bis}}}}(16r_0+8)}(2\cconst)^{r_0+3}}{n} \Pconfig\big(\BipGSet_n(d)\big) \Dir_{\Psimple}(f,f) \sum_{k_1=1}^{\cconst}  \frac{(d-1)^{2k_1}}{2^{k_1} k_1!}\\
&\leq \widetilde C\,  \frac{d^{5+2r_0+C_{\text{\tiny{\ref{l: 2095203598305298 bis}}}}(16r_0+8)}(2\cconst)^{r_0+3}}{n} \Dir_{\Psimple}(f,f),
\end{align*}
for some appropriate universal constants $C''>0$ and $\widetilde C>0$. Applying \eqref{eq: parameters} and \eqref{eq: assumption-d}, we finish the proof. 
\end{proof}

Next, we consider the sum over type $2$ tuples with not many non-standard edges (see Definition~\ref{def: connection2}). 

\begin{lemma}[Type $2$ tuples with small $|I_{NS}|$]\label{lem: upper bound-small r-1}
Let $f: \BipGSet_n(d)\to \R$ satisfy $\Exp_{\Psimple} f=0$ and $r_0$ be as in \eqref{eq: parameters}. 
Then, we have 
$$
 \frac{d}{n}\sum_{\underset{0\leq r\leq r_0}{1\leq k\leq \cconst}}  \sum_{\underset{\text{of subtype $2$--$(k,r)$}}{(G_1,G_1',G_2,G_1') }}   \frac{\Pconfig (G_1')}{ |\SNeigh(G_1')|^2}  \big(f(G_1)-f(G_2)\big)^2\leq C\Dir_{\Psimple}(f,f),
$$ 
where $C>0$ is a universal constant. 
\end{lemma}
\begin{proof}
For a tuple $\tuple=(G_1,G_1',G_2,G_1')$ of subtype $2$--$(k,r)$, let $P_\tuple$ be the corresponding connection and $L_\tuple$ its length. 
Then we can write 
\begin{align*}
\alpha:&=\sum_{\underset{0\leq r\leq r_0}{1\leq k\leq \cconst}}  \sum_{\underset{\text{of subtype $2$--$(k,r)$}}{(G_1,G_1',G_2,G_1') }}  \frac{\Pconfig (G_1')}{|\SNeigh(G_1')|^2}  \big(f(G_1)-f(G_2)\big)^2 \\
&=\sum_{\underset{0\leq r\leq r_0}{1\leq k\leq \cconst}}  \sum_{\underset{\text{of subtype $2$--$(k,r)$}}{(G_1,G_1',G_2,G_1') }}  \frac{\Pconfig (G_1')}{ |\SNeigh(G_1')|^2}   \big(f(P_\tuple[0])-f(P_\tuple[L_\tuple])\big)^2. 
\end{align*}
Using the Cauchy--Schwarz inequality, we get 
$$
\alpha
\leq\sum_{\underset{0\leq r\leq r_0}{1\leq k\leq \cconst}}  \sum_{\underset{\text{of subtype $2$--$(k,r)$}}{(G_1,G_1',G_2,G_1') }}  \frac{L_{\tuple}\, \Pconfig (G_1')}{ |\SNeigh(G_1')|^2}\sum_{t=1}^{L_\tuple}  \big(f(P_\tuple[t])-f(P_\tuple[t-1])\big)^2. 
$$
Now note that $\Pconfig (G_1')= \frac{(d!)^{2n}}{(nd)! 2^{k}}$ and recall that, in view of Remark~\ref{rem: 02985720958709},
$L_\tuple\leq 2k+C_{\text{\tiny\ref{l: 2095203598305298 bis}}}(8r+4)$. Using this together with Lemma~\ref{l: 2-9857-2985} and assuming that $n$ is large enough, we get for some universal constant $C$ that 
\begin{align*}
\alpha&\leq C\,\cconst\sum_{\underset{0\leq r\leq r_0}{1\leq k\leq \cconst}}  \sum_{\underset{\text{of subtype $2$--$(k,r)$}}{(G_1,G_1',G_2,G_1') }} \frac{ (d!)^{2n}}{2^k (nd)! (nd)^{2k}}\sum_{t=1}^{L_\tuple}  \big(f(P_\tuple[t])-f(P_\tuple[t-1])\big)^2\\
&\leq C\,\cconst\Pconfig\big(\BipGSet_n(d)\big)\sum_{H,H'\in \BipGSet_n(d): H\sim H'} \Psimple(H)Q_u(H,H') \big(f(H)-f(H')\big)^2 \sum_{\underset{0\leq r\leq r_0}{1\leq k\leq \cconst}}  \frac{\gamma_{k,r}(H,H')}{2^k (nd)^{2k-2}},
\end{align*}
where $\gamma_{k,r}(H,H')$ denotes the number of tuples $\tuple=(G_1,G_1',G_2,G_1')$ of subtype $2$--$(k, r)$ such that $(H,H')$ is a pair of adjacent graphs in $P_\tuple$. 
Using Proposition~\ref{prop: number-tuples}, we obtain 
\begin{align*}
\alpha&\leq 
C'\,\cconst\,\Pconfig\big(\BipGSet_n(d)\big)\sum_{H,H'\in \BipGSet_n(d): H\sim H'} \Psimple(H)Q_u(H,H') \big(f(H)-f(H')\big)^2\times\\
&\hspace{1cm}\sum_{\underset{0\leq r\leq r_0}{1\leq k\leq \cconst}}
\frac{k^{r+1}}{2^k\,k !}
\, d^{2+r+C_{\text{\tiny{\ref{l: 2095203598305298 bis}}}}(16r+8)}
(d-1)^{r+2k},
\end{align*}
whence
$$
\alpha\leq C''\,\cconst^{r_0+2}d^{2+2r_0+C_{\text{\tiny{\ref{l: 2095203598305298 bis}}}}(16r_0+8)}\Dir_{\Psimple}(f,f),
$$
for some universal constant $C''$. The proof follows from the choice of parameters in \eqref{eq: parameters} and \eqref{eq: assumption-d}, and the assumption that $n$ is large enough. 
\end{proof}

The next lemma is needed to deal with tuples of type $1$ and $2$ with many non-standard edges. 

\begin{lemma}[Type $1$ and $2$ tuples with large $|I_{NS}|$]\label{lem: multigraph with large r-2}
Let $f: \BipGSet_n(d)\to \R$ satisfy $\Exp_{\Psimple} f=0$ and $r_0$ be as in \eqref{eq: parameters}. Then, we have
$$
\sum_{\underset{\underset{u\in \{A,B\}}{r> r_0}}{1\leq k_1,k_2\leq \cconst}}  \sum_{\underset{\text{of subtype $1$--$(k_1,k_2,u,r)$}}{(G_1,G_1',G_2,G_2') }} \frac{\Pconfig(G_1')Q_c(G_1',G_2')}{|\SNeigh(G_1')|\cdot |\SNeigh(G_2')|}\big(f(G_1)-f(G_2)\big)^2 \leq \frac{d}{n^2} \Var_{\Psimple}(f),
$$
and 
$$
 \frac{d}{n}\sum_{\underset{r> r_0}{1\leq k\leq \cconst}}  \sum_{\underset{\text{of subtype $2$--$(k,r)$}}{(G_1,G_1',G_2,G_1') }}  \frac{\Pconfig(G_1')}{|\SNeigh(G_1')|^2}\big(f(G_1)-f(G_2)\big)^2 \leq \frac{d}{n^2} \Var_{\Psimple}(f).
 $$
\end{lemma}
\begin{proof}
We start by applying the Cauchy--Schwarz inequality and reversibility to write 
\begin{align*}
\alpha:&=\sum_{\underset{\underset{u\in \{A,B\}}{r> r_0}}{1\leq k_1,k_2\leq \cconst}}  \sum_{\underset{\text{of subtype $1$--$(k_1,k_2,u,r)$}}{(G_1,G_1',G_2,G_2') }} \frac{\Pconfig(G_1')Q_c(G_1',G_2')}{|\SNeigh(G_1')|\cdot |\SNeigh(G_2')|}\big(f(G_1)-f(G_2)\big)^2\\
&\leq 4 \sum_{\underset{\underset{u\in \{A,B\}}{r> r_0}}{1\leq k_1,k_2\leq \cconst}}  \sum_{\underset{\text{of subtype $1$--$(k_1,k_2,u,r)$}}{(G_1,G_1',G_2,G_2') }} \frac{\Pconfig(G_1')Q_c(G_1',G_2')}{|\SNeigh(G_1')|\cdot |\SNeigh(G_2')|} f(G_1)^2\\
&\leq  \frac{128\Pconfig\big(\BipGSet_n(d)\big)}{nd(nd-1)} \sum_{G_1\in \BipGSet_n(d)} \Psimple(G_1) f(G_1)^2\sum_{\underset{\underset{u\in \{A,B\}}{r> r_0}}{1\leq k_1,k_2\leq \cconst}}  \sum_{\underset{\text{of subtype $1$--$(k_1,k_2,u,r)$}}{\underset{(G_1,G_1',G_2,G_2') }{G_1',G_2',G_2 \text{ such that}}}} \frac{1}{2^{k_1}(nd)^{k_1+k_2}},
\end{align*}
where we used Lemmas~\ref{lem: nice-multigraph-prob and transition} and \ref{l: 2-9857-2985}, and that $Q_c(G_1',G_2')\leq \frac{8}{nd(nd-1)}$. 
To prove the first part of the lemma, it is enough to provide an upper bound for 
$$
\alpha_{G_1}:=  \frac{128\Pconfig\big(\BipGSet_n(d)\big)}{nd(nd-1)} \sum_{\underset{\underset{u\in \{A,B\}}{r> r_0}}{1\leq k_1,k_2\leq \cconst}}  \sum_{\underset{\text{of subtype $1$--$(k_1,k_2,u,r)$}}{\underset{(G_1,G_1',G_2,G_2') }{G_1',G_2',G_2 \text{ such that}}}} \frac{1}{2^{k_1}(nd)^{k_1+k_2}},
$$
for every fixed $G_1\in \BipGSet_n(d)$. Fix for a moment $1\leq k_1,k_2\leq \cconst$ (with $\vert k_1-k_2\vert\leq 2$), $r>r_0$ and $u\in \{A,B\}$. Note that the number of $m$-standard edges of any tuple $(G_1,G_1',G_2,G_2')$ of subtype $1$--$(k_1,k_2,u,r)$ lies between $k_1-r$ and
$(k_1+k_2+1)/2-r/3$.
Indeed, according to Lemma~\ref{l: 05610948704987}, 
for every element $W_q$ of the canonical partition of $I_{NS}$ we have $|W_q|\leq M^1_q+M^2_q+1\leq \frac{3}{2}(M^1_q+M^2_q)$
whenever the switching between $G_1'$ and $G_2'$ does not operate on $W_q$, and $|W_q|\leq M^1_q+M^2_q+2\leq \frac{3}{2}(M^1_q+M^2_q)
+3/2$ otherwise,
where $M^1_q$ and $M^2_q$
are the number of non-standard multiedges incident to $W_q$ in graphs $G_1'$ and $G_2'$, respectively.
Thus, $r\leq \frac{3}{2}(k_1+k_2-2N_{ms})+3/2$ implying the range for $N_{ms}$.

Further, it follows from Lemma~\ref{l: 98562098724} that there are at most ${nd \choose k_1} (d-1)^{2k_1}$ choices for graphs $G_1'\in \categ(k_1)$ such that $G_1\in \SNeigh(G_1')$. 
Moreover, there are at most $nd(nd-1)/2$ neighboring multigraphs $G_2'\in \categ(k_2)$ for a given $G_1'$.  
We now fix such $G_1',G_2'$ and $N_{ms}\in [k_1-r,(k_1+k_2+1)/2-r/3]$, and estimate the number of graphs $G_2$ with $G_2\in \SNeigh(G_2')$ and such that $(G_1,G_1',G_2,G_2')$ is of subtype $1$--$(k_1,k_2,u,r)$ and has $N_{ms}$ standard edges.
Following the usual approach within this paper, we will do this by estimating the complexity of a data structure sufficient to identify $G_2$ uniquely.

First, we need to identify the $m$--standard multiedges associated with the tuple. The graph $G_1$, hence its multiedges, have already been defined, so it is only needed to mark those of the $k_1$ multiedges which are $m$--standard.
There are clearly ${k_1\choose N_{ms}}$ choices for these $m$--standard edges and at most $(nd)^{N_{ms}}$ choices for the corresponding switchings in the simple path leading from $G_2'$ to some graph $G_2\in \SNeigh(G_2')$.

Now, assume the collection of $m$--standard edges associated with the tuple is fixed.
To identify $G_2$ uniquely, it remains to define
those switchings generating the simple path from $G_2'$ to $G_2$ which operate on non-standard multiedges of $G_2'$.
Set
$$Q:=\big\{(i',j'):\;(i',j')\mbox{ is a multiedge in $G_2'$ which is not $m$--standard}\big\}.$$
Note that any non-standard multiedge must violate at least one of the properties
listed in Definition~\ref{def: standard}.
Pick any non-standard multiedge $(i,j)\in Q$ of $G_2'$,
and let $\langle u,u',v,v'\rangle$ be the simple switching operation taking $G_1'$ to $G_2'$.
There are two cases.
\begin{enumerate}
\item\label{item: 40983049871049} Either $i\in\{u,u'\}$;
\item Or $i\notin\{u,u'\}$. Note that in this case necessarily $(i,j)$ is a (non-standard) multiedge in $G_1'$ as well.
Let $\langle i,\ione(i),j,\jone(i)\rangle$ be the simple switching operation from the simple path leading from $G_1'$ to $G_1$,
operating on $(i,j)$. We have four possible subcases.
\begin{enumerate}
\item\label{item: 4296194850491} $\ione(i)\in \{u,u'\}$.
\item\label{item: 0397160498274-97} $\ione(i)\notin \{u,u'\}$ and for the simple switching operation in the simple path from
$G_2'$ to $G_2$ operating on $(i,j)$, {\bf both} left vertices it operates on are in $I(G_1,G_1')$.
It is immediate that the number of choices for such switching is at most $k_1d$
(because the set $I(G_1,G_1')$ is already fixed and contains precisely $k_1$ left vertices which are not incident to multiedges in $G_1'$).
\item\label{item: 932641094283058409} The above two do not hold and $\ione(i)\in I(G_2,G_2')$. In this subcase, there is necessarily a simple switching
within the path from $G_2'$ to $G_2$ which operates on $\ione(i)$.
\item\label{item: 09561071-4-9} The first two subcases do not hold and $\ione(i)\notin I(G_2,G_2')\cup\{u,u'\}$. In this setting, necessarily one of the
Properties \ref{p: 13981749827} or \ref{p: standard-3} in Definition~\ref{def: standard} is violated for $(i,j)$
(note that Property~\ref{p: standard-2} cannot be violated because otherwise we would find ourselves in \ref{item: 0397160498274-97}).
There are at most $2d$ choices for the switching violating Property~\ref{p: standard-3}.
Finally, there are at most $d^2$ choices for the switching violating Property~\ref{p: 13981749827} and satisfying all the other properties
(see Figure~\ref{fig: non-standard}).
\end{enumerate}
\end{enumerate}
To summarize, for each edge $(i,j)\in Q$ we can bound the complexity of corresponding switching
in the simple path leading from $G_2'$ to $G_2$,
by $k_1d+d^2+2d$
unless \ref{item: 40983049871049}, \ref{item: 4296194850491} or \ref{item: 932641094283058409} holds.
There are at most $4$ edges $(i,j)\in Q$ such that \ref{item: 40983049871049} or \ref{item: 4296194850491} hold, and complexity of each of the corresponding switchings is at most $nd$.
So, we can proceed as follows: first, choose 
a subset $Z_1\subset Q$ of cardinality at most $4$ which comprises the edges satisfying either \ref{item: 40983049871049} or \ref{item: 4296194850491} above. Next, choose a subset $Z_2\subset Q\setminus Z_1$ of those edges which satisfy \ref{item: 932641094283058409}.
Then $Q\setminus (Z_1\cup Z_2)$ will comprise the edges satisfying either \ref{item: 0397160498274-97} or \ref{item: 09561071-4-9}.
We will find an upper bound on the complexity of each switching associated to an edge from $Z_2$, by $nd$.
Thus, recalling that $|Q|=k_2-N_{ms}$, the complexity of the choice of non-standard switchings is at most
$$
\sum_{z_1=0}^4 \sum_{z_2} {k_2-N_{ms} \choose z_1}{k_2-N_{ms}-z_1 \choose z_2}(nd)^{z_1+z_2}
\big(k_1d+d^2+2d\big)^{k_2-N_{ms}-z_1-z_2},
$$
where $z_1$ and $z_2$ are cardinalities of $Z_1$ and $Z_2$, respectively. To make the above estimate useful, we need to define
the range of the variable $z_2$.
To do that, observe that for every $(i,j)$ from $Z_2$ there is necessarily an edge $(i'_{ij},j'_{ij})\in Q$
such that $\itwo(i'_{ij})=\ione(i)$, where we define \chng{$\itwo(i'_{ij})$ as the unique left vertex which, together with $i'_{ij}$,
is operated on by a switching in a simple path from $G_2'$ to $G_2$.} 
Note further that necessarily this edge $(i'_{ij},j'_{ij})$
satisfies one of the conditions \ref{item: 40983049871049}, \ref{item: 4296194850491} or \ref{item: 0397160498274-97}\footnote{\chng{Indeed, assume that the edge
$(i'_{ij},j'_{ij})$ does not satisfy condition \ref{item: 40983049871049}. Then necessarily
$(i'_{ij},j'_{ij})$ is a multiedge in $G_1'$, so that $i'_{ij}\in I(G_1,G_1')$, whereas, by the choice of 
$(i'_{ij},j'_{ij})$, $\itwo(i'_{ij})=\ione(i)\in I(G_1,G_1')$. Thus, either \ref{item: 4296194850491} or \ref{item: 0397160498274-97} must hold.}},
i.e.\ belongs to $Q\setminus Z_2$.
Thus, we obtain that $|Z_2|\leq |Q\setminus Z_2|$, whence $z_2\leq |Q|/2$. The complexity of the choice of non-standard edges
can then be bounded by
$$
\sum_{z_1=0}^4 \sum_{z_2=0}^{\lfloor (k_2-N_{ms})/2\rfloor} {k_2-N_{ms} \choose z_1}{k_2-N_{ms}-z_1 \choose z_2}(nd)^{z_1+z_2}
\big(k_1d+d^2+2d\big)^{k_2-N_{ms}-z_1-z_2},
$$
which is in turn bounded by
$$
3(k_2-N_{ms})^4\,2^{k_2-N_{ms}}\,(nd)^{4+(k_2-N_{ms})/2}\,\big(k_1d+d^2+2d\big)^{(k_2-N_{ms})/2-4}.
$$

%

We deduce that there are at most 
$$
3{k_1\choose N_{ms}} (nd)^{N_{ms}} 
(k_2-N_{ms})^4\,2^{k_2-N_{ms}}\,(nd)^{4+(k_2-N_{ms})/2}\,\big(k_1d+d^2+2d\big)^{(k_2-N_{ms})/2-4},
$$
choices for a graph $G_2\in \SNeigh(G_2')$ with the desired properties. Putting the estimates together,
we obtain
\begin{align*}
\alpha_{G_1}\leq 
&C'\,\Pconfig\big(\BipGSet_n(d)\big)\sum_{\underset{\underset{u\in \{A,B\}}{r> r_0}}{1\leq k_1,k_2\leq \cconst,\,|k_1-k_2|\leq 2}}
\frac{(d-1)^{2k_1}}{2^{k_1}k_1!}\times\\
&\sum_{N_{ms}\in [k_1-r,(k_1+k_2+1)/2-r/3]}
k_1^{k_1-N_{ms}}  
(k_2-N_{ms})^4\,2^{k_2-N_{ms}}\,\frac{(k_1d+d^2+2d)^{(k_2-N_{ms})/2-4}}{(nd)^{(k_2-N_{ms})/2-4}},
\end{align*}
and using \eqref{eq: prob-simple}, the definition of $r_0$ and that $\min(k_1-N_{ms},k_2-N_{ms})\geq (r_0+1)/3-3/2$, we can obtain the (rough) bound 
\begin{align*}
\alpha_{G_1}&\leq 
C''\,\Pconfig\big(\BipGSet_n(d)\big)\sum_{\substack{1\leq k_1,k_2\leq \cconst\\ |k_1-k_2|\leq 2}}
\frac{(d-1)^{2k_1}}{2^{k_1}k_1!}
k_1^{(r_0+1)/3+1/2}  
\,\bigg(\frac{k_1d+d^2+2d}{nd}\bigg)^{(r_0+1)/6-19/4}\\
&\leq C'''\,\cconst^{(r_0+1)/3+1/2}\bigg(\frac{\cconst }{n}\bigg)^{(r_0+1)/6-19/4}\leq C'''\,\frac{\cconst^{16.25} }{n^{2.083}},
\end{align*}
for some universal constants $C'',C'''>0$. It remains to use \eqref{eq: parameters} and \eqref{eq: assumption-d} to get the first part of the lemma.

We now turn to the second part of the lemma. Similarly, we use the Cauchy--Schwarz inequality and reversibility to write 
\begin{align*}
\beta&:= \frac{d}{n} \sum_{\underset{r> r_0}{1\leq k\leq \cconst}}  \sum_{\underset{\text{of subtype $2$--$(k,r)$}}{(G_1,G_1',G_2,G_1') }}  \frac{\Pconfig(G_1')}{|\SNeigh(G_1')|^2}
\big(f(G_1)-f(G_2)\big)^2\\
&\leq  \frac{4d}{n}\sum_{\underset{r> r_0}{1\leq k\leq \cconst}}  \sum_{\underset{\text{of subtype $2$--$(k,r)$}}{(G_1,G_1',G_2,G_1') }}  \frac{\Pconfig(G_1')}{|\SNeigh(G_1')|^2}f(G_1)^2\\
&\leq \frac{16d\Pconfig\big(\BipGSet_n(d)\big)}{n} \sum_{G_1\in \BipGSet_n(d)} \Psimple(G_1) f(G_1)^2 \sum_{\underset{r> r_0}{1\leq k\leq \cconst}}  \sum_{\underset{\text{of subtype $2$--$(k,r)$}}{\underset{(G_1,G_1',G_2,G_1')}{G'_1,G_2 \text{ such that}}}} \frac{1}{2^k(nd)^{2k}},
\end{align*}
where we used Lemmas~\ref{lem: nice-multigraph-prob and transition} and \ref{l: 2-9857-2985} to get the last inequality. We will prove that 
$$
\beta_{G_1}:= \frac{16d\Pconfig\big(\BipGSet_n(d)\big)}{n}\sum_{\underset{r> r_0}{1\leq k\leq \cconst}}\,  \sum_{\underset{\text{of subtype $2$--$(k,r)$}}{\underset{(G_1,G_1',G_2,G_1')}{\, G'_1,G_2 \text{ such that}}}} \frac{1}{2^k(nd)^{2k}}\leq \frac{d}{n^2},
$$
for every $G_1\in \BipGSet_n(d)$. Since the argument is very similar to the one above, we will skip some details.
Fix $1\leq k\leq \cconst$ and $r>r_0$.
Note that the number of $m$-standard edges of any tuple $(G_1,G_1',G_2,G_1')$ of subtype $2$--$(k,r)$ lies between $k-r$
and $k+1/2-r/3$ (and that we necessarily have $k\geq r/2$; see Lemma~\ref{l: 05610948704987}). 
It follows from Lemma~\ref{l: 98562098724} that there are at most ${nd \choose k} (d-1)^{2k}$ choices for $G_1'\in \categ(k)$ such that $G_1\in \SNeigh(G_1')$. 
We now fix such $G_1'$ and $N_{ms}\in [k-r,k+1/2-r/3]$ and estimate the number of graphs $G_2\in \SNeigh(G_1')$ such that $(G_1,G_1',G_2,G_1')$ is of subtype $2$--$(k,r)$ and has $N_{ms}$ standard edges.
There are clearly ${k\choose N_{ms}}$ choices for these $m$--standard edges and at most $(nd)^{N_{ms}}$ choices for the corresponding switchings in the simple paths leading to some graph $G_2\in \SNeigh(G_1')$.
Further, repeating the above argument, we get that complexity of non-standard switchings is at most
$$
(k-N_{ms})^4\,2^{k-N_{ms}}\,(nd)^{4+(k-N_{ms})/2}\,\big(kd+d^2+2d\big)^{(k-N_{ms})/2-4}.
$$
We deduce that there are at most 
$$
{k\choose N_{ms}} (nd)^{N_{ms}}(k-N_{ms})^4\,2^{k-N_{ms}}\,(nd)^{4+(k-N_{ms})/2}\,\big(kd+d^2+2d\big)^{(k-N_{ms})/2-4}
$$
choices for a graph $G_2\in \SNeigh(G_1')$ with the desired properties. 
Putting the above estimates together and using \eqref{eq: prob-simple}, we deduce that 
\begin{align*}
\beta_{G_1}&\leq \frac{C'd}{n}
\sum_{\underset{r> r_0}{1\leq k\leq \cconst}}\,
\sum_{N_{ms}\in [k-r,k+1/2-r/3]}
{k\choose N_{ms}} (k-N_{ms})^4\,2^{k-N_{ms}}\,\frac{(kd+d^2+2d)^{(k-N_{ms})/2-4}}{(nd)^{(k-N_{ms})/2-4}}\\
&\leq\chng{
\frac{C'_1d}{n}
\sum_{1\leq k\leq \cconst}\,
k^{r_0/3-1/6}\bigg(\frac{k+d+2}{n}\bigg)^{(r_0+1)/6-4.25}}\\
&\leq\chng{
\frac{C'_2 d}{n}\,\cconst^{1+r_0/3-1/6} \bigg(\frac{\cconst}{n}\bigg)^{(r_0+1)/6-4.25}\leq \frac{C'_2d\,\cconst^{16.75}}{n^{3.58}}}, 
\end{align*}
where $C',C'_1,C_2'>0$ are universal constants. Finally, using \eqref{eq: parameters} and \eqref{eq: assumption-d}, we finish the proof. 
\end{proof}

We now estimate the sum over type $3$ tuples. 

\begin{lemma}[Type $3$ tuples]\label{lem: dir-type 3}
There exists a universal constant $C>0$ such that 
$$
\sum_{\underset{\text{ of type $3$}}{(G_1,G_1,G_2,G_2')}} \frac{\Pconfig(G_1)Q_c(G_1,G_2')}{|\SNeigh(G_2')|}  \big( f(G_1)- f(G_2)\big)^2
\leq C \, \Dir_{\Psimple}(f,f).
$$
\end{lemma}
\begin{proof}
For a tuple $\tuple=(G_1,G_1,G_2,G_1')$ of type $3$, let $P_\tuple$ be the corresponding connection and $L_\tuple$ its length. 
We start by using the Cauchy--Schwarz inequality to write
\begin{align*}
\alpha:&=\sum_{\underset{\text{ of type $3$}}{(G_1,G_1,G_2,G_2')}} \frac{\Pconfig(G_1)Q_c(G_1,G_2')}{|\SNeigh(G_2')|}  \big( f(G_1)- f(G_2)\big)^2\\
&\leq \sum_{\underset{\text{ of type $3$}}{\tuple=(G_1,G_1,G_2,G_2')}} \frac{L_\tuple \Pconfig(G_1)Q_c(G_1,G_2')}{|\SNeigh(G_2')|}  \sum_{t=1}^{L_\tuple} \big(f(P_\tuple[t])-f(P_\tuple[t-1])\big)^2\\
&\leq C\sum_{\underset{\text{ of type $3$}}{\tuple=(G_1,G_1,G_2,G_2')}} \frac{ \Pconfig(G_1)Q_c(G_1,G_2')}{|\SNeigh(G_2')|}  \sum_{t=1}^{L_\tuple} \big(f(P_\tuple[t])-f(P_\tuple[t-1])\big)^2,
\end{align*}
where we used that $L_{\tuple}$ is bounded above by a universal constant $C>0$. 
Since $G_2'$ in the above sum is of category $2$, then using Lemma~\ref{l: 2-9857-2985} together with the definition of $Q_c$ and $\Pconfig$, we get for some universal constant $C'$ 
$$
\alpha\leq \frac{C'}{(nd)^2} \sum_{H,H'\in \BipGSet_n(d): H\sim H'} \Psimple(H)Q_u(H,H') \big(f(H)-f(H')\big)^2 \gamma(H,H'),
$$
where $\gamma(H,H')$ denotes the number of tuples $\tuple=(G_1,G_1,G_2,G_2')$ of type $3$ such that $(H,H')$ is a pair of adjacent graphs in $P_\tuple$.  
Using the estimate of Proposition~\ref{prop: type3 count} and the assumption on $d$ in \eqref{eq: assumption-d}, we finish the proof. 
\end{proof}

Combining Lemmas~\ref{lem: upper-bound-dir-good}, ~\ref{lem: upper bound-small r-2},~\ref{lem: upper bound-small r-1},~\ref{lem: multigraph with large r-2} and~\ref{lem: dir-type 3}, we are now ready to state the main result of this section:

\begin{theor}[Comparison of Dirichlet forms]\label{th: main comparison dirichlet}
Let $f: \BipGSet_n(d)\to \R$ and let $\tilde f: \ConfBipGSet_n(d)\to \R$ be its extension. 
Then, we have 
$$
\Exp \Dir_{\Pconfig}(\tilde f, \tilde f)\leq C\,\Dir_{\Psimple}(f,f)+C\,\frac{d}{n^2} \Var_{\Psimple}(f),
$$
for some universal constant $C>0$. 
\end{theor}

\section{Lower bounds on the variance of $\tilde f$}\label{sec: variance}

The goal of this section is to compare the variance of a function defined on $\BipGSet_n(d)$ with that of its extension constructed
in Section~\ref{sec: construction}. Observe that whenever $\tilde f$ is the extension of a function $f$,
we have that $\tilde f+R$ is the extension of $f+R$ for any fixed number $R$;
that is, a map assigning to every function $f$ on $\BipGSet_n(d)$ its extension on $\ConfBipGSet_n(d)$
the way we have defined it, commutes with constant shift operator. Therefore, without loss of generality
we can (and will) assume throughout the section that $\Exp_{\Psimple}\,f=0$.

The next lemma relates the fluctuations of $\tilde f$ to those of $f$. 
\begin{lemma}\label{lem: variation-extension}
Let $G_1,G_2$ be any two distinct graphs in $\categ([1,\cconst])$.
Then
$$
(\tilde f(G_1)-\tilde f(G_2))^2\geq
\frac{\eta_{G_1,G_2}}{3|\SNeigh(G_1)|\,|\SNeigh(G_2)|}\sum_{G'\in \SNeigh(G_1),\,G''\in \SNeigh(G_2)}
(f(G')-f(G''))^2,
$$
where $\eta_{G_1,G_2}$ is defined as in Proposition~\ref{prop: 29870536509385}.
\end{lemma}
\begin{proof}
By the definition of $\tilde f$, we have
\begin{align*}
\tilde f(G_1)-\tilde f(G_2)
&=\xi_{G_1}\bigg(\frac{1}{|\SNeigh(G_1)|}
\sum_{G'\in \SNeigh(G_1)}(f(G')-h(G_1))^2\bigg)^{1/2}\\
&-\xi_{G_2}\bigg(\frac{1}{|\SNeigh(G_2)|}
\sum_{G''\in \SNeigh(G_2)}(f(G'')-h(G_2))^2\bigg)^{1/2}
+h(G_1)-h(G_2),
\end{align*}
where $h(G_i)$ is the arithmetic average of $f$ over the $s$--neighborhood of $G_i$, $i=1,2$.

Recall from Proposition~\ref{prop: 29870536509385} that  $\eta_{G_1,G_2}$ is the indicator of the event that both $\xi_{G_1}(h(G_1)-h(G_2))$ and $-\xi_{G_2}(h(G_1)-h(G_2))$
are non-negative, 
$\xi_{G_1}$ and $\xi_{G_2}$ have opposite signs,
and both $\xi_{G_1}$ and $\xi_{G_2}$ are at least $1$ by the absolute value.
Note that whenever the event holds, the expressions
$$
\xi_{G_1}\Big(\frac{1}{|\SNeigh(G_1)|}
\sum_{G'\in \SNeigh(G_1)}(f(G')-h(G_1))^2\Big)^{1/2},\; -\xi_{G_2}\Big(\frac{1}{|\SNeigh(G_2)|}
\sum_{G''\in \SNeigh(G_2)}(f(G'')-h(G_2))^2\Big)^{1/2}
$$
are either both non-negative or both non-positive, so their product is non-negative.
Then, expanding the square and using the definition of $\eta_{G_1,G_2}$, we can write 
\begin{align*}
(\tilde f(G_1)-\tilde f(G_2))^2
&\geq \eta_{G_1,G_2}\bigg(
(h(G_1)-h(G_2))^2\\
&\hspace{1cm}+\frac{1}{|\SNeigh(G_1)|}
\sum_{G'\in \SNeigh(G_1)}(f(G')-h(G_1))^2\\
&\hspace{1cm}+\frac{1}{|\SNeigh(G_2)|}
\sum_{G''\in \SNeigh(G_2)}(f(G'')-h(G_2))^2
\bigg)\\
&=\frac{\eta_{G_1,G_2}}{|\SNeigh(G_1)|\,|\SNeigh(G_2)|}
\sum_{G'\in \SNeigh(G_1),G''\in \SNeigh(G_2)}\bigg((h(G_1)-h(G_2))^2\\
&\hspace{1cm}+(f(G')-h(G_1))^2+(f(G'')-h(G_2))^2
\bigg)\\
&\geq \frac{\eta_{G_1,G_2}}{3|\SNeigh(G_1)|\,|\SNeigh(G_2)|}
\sum_{G'\in \SNeigh(G_1),G''\in \SNeigh(G_2)}(f(G')-f(G''))^2,
\end{align*}
where we used the Cauchy--Schwarz inequality in the last line. 
\end{proof}

We now use the above calculation to estimate a part of the variance of the extension.

\begin{lemma}\label{lem: variance-multigraphs-extension}
There exists a universal constant $c>0$ such that with probability at least $c$ with respect to the randomness of $\tilde f$,
\begin{align*}
 \sum_{G_1,G_2\in \categ([1,\cconst])} &\Pconfig(G_1)\Pconfig(G_2) \big(\tilde f(G_1)-\tilde f(G_2)\big)^2 \\
&\hspace{2cm}\geq 
 c\sum_{\underset{\tilde d(G',G'')\geq \gconst +8\cconst }{G',G''\in \Expset(\zconst)}}\Psimple(G')\Psimple(G'') (f(G')-f(G''))^2,
\end{align*}
where $\tilde d(G',G''):=\vert\{(i,j)\in [n]\times [n]:\, \Adj(G')_{ij}\neq \Adj(G'')_{ij}\}\vert$
and where the set $\Expset(\zconst)$ was defined in Definition~\ref{def: measure of ae}. 
\end{lemma}
\begin{proof}
Denote by $\mathcal M_{\cconst,\gconst}$ the set of pairs $(G_1,G_2)\in \categ([1,\cconst])\times \categ([1,\cconst])$ such that $\tilde d(G_1,G_2)\geq \gconst$. 
Combining Lemma~\ref{lem: variation-extension} together with Proposition~\ref{prop: 29870536509385}, we get
that with probability at least $c$ (for some universal constant $c>0$), 
\begin{align*}
\alpha&:=  \sum_{G_1,G_2\in \categ([1,\cconst])} \Pconfig(G_1)\Pconfig(G_2) \big(\tilde f(G_1)-\tilde f(G_2)\big)^2\\
&\geq c\sum_{G_1,G_2\in \mathcal M_{\cconst,\gconst}} \frac{\Pconfig(G_1)\Pconfig(G_2)}{\vert\SNeigh(G_1)\vert \vert \SNeigh(G_2)\vert}\sum_{G'\in \SNeigh(G_1), G''\in \SNeigh(G_2)} (f(G')-f(G''))^2\\
&=c\sum_{G',G''\in \BipGSet_n(d)}(f(G')-f(G''))^2 \sum_{\underset{G'\in \SNeigh(G_1), G''\in \SNeigh(G_2)}{G_1,G_2\in \mathcal M_{\cconst,\gconst}:}} \frac{\Pconfig(G_1)\Pconfig(G_2)}{\vert\SNeigh(G_1)\vert \vert \SNeigh(G_2)\vert}.
\end{align*}
Note that by the definition of the simple path, we have $\tilde d(G,G_1)=4k$ for any $G_1\in \categ(k)$, $k\in [1,\cconst]$, and any $G\in \SNeigh(G_1)$. Therefore, if $G', G''\in \BipGSet_n(d)$ are such that $\tilde d(G',G'')\geq \gconst +8\cconst$, then necessarily any $G_1,G_2\in \categ([1,\cconst])$ with $G'\in \SNeigh(G_1)$ and $G''\in \SNeigh(G_2)$ satisfy $\tilde d(G_1,G_2)\geq \gconst $. Thus, we get that with probability
at least $c$,
\begin{align*}
\alpha&\geq c\sum_{\underset{\tilde d(G',G'')\geq \gconst +8\cconst}{G',G''\in \BipGSet_n(d)}}(f(G')-f(G''))^2 \sum_{\underset{G'\in \SNeigh(G_1), G''\in \SNeigh(G_2)}{G_1,G_2\in \categ([1,\cconst]}} \frac{\Pconfig(G_1)\Pconfig(G_2)}{\vert\SNeigh(G_1)\vert \vert \SNeigh(G_2)\vert}\\
&\geq c \sum_{\underset{\tilde d(G',G'')\geq \gconst +8\cconst}{G',G''\in \Expset(\zconst)}}(f(G')-f(G''))^2 \sum_{k,\ell=1}^{\cconst}\sum_{\underset{G'\in \SNeigh(G_1), G''\in \SNeigh(G_2)}{\underset{G_2\in \categ(\ell)}{G_1\in \categ(k)}}} \frac{\Pconfig(G_1)\Pconfig(G_2)}{\vert\SNeigh(G_1)\vert \vert \SNeigh(G_2)\vert}. 
\end{align*}
Using Lemmas~\ref{lem: nice-multigraph-prob and transition}, \ref{l: 2-9857-2985} and \ref{l: 98562098724}, we obtain with the same probability
\begin{align*}
\alpha &\geq 
c \sum_{\underset{\tilde d(G',G'')\geq \gconst +8\cconst}{G',G''\in \Expset(\zconst)}}(f(G')-f(G''))^2 \sum_{k,\ell=1}^{\cconst}\sum_{\underset{G'\in \SNeigh(G_1), G''\in \SNeigh(G_2)}{\underset{G_2\in \categ(\ell)}{G_1\in \categ(k)}}} \frac{ 
\Pconfig\big(\BipGSet_n(d)\big)^2\Psimple(G')\Psimple(G'')
}{2^{k+\ell}\,(nd)^{k+\ell}}\\
&\geq
\frac{c}{4}\, \Pconfig\big(\BipGSet_n(d)\big)^2 \sum_{\underset{\tilde d(G',G'')\geq \gconst +8\cconst}{G',G''\in \Expset(\zconst)}}\Psimple(G')\Psimple(G'') (f(G')-f(G''))^2 \sum_{k,\ell=1}^{\cconst} \frac{(d-1)^{2(k+\ell)}}{k!\ell! 2^{k+\ell}}.
\end{align*}
which combined with  \eqref{eq: prob-simple} and our choice for the parameter $\cconst$ (see \eqref{eq: parameters}), finishes the proof.
\end{proof}

Next, we estimate another part of the variance of the extension. 

\begin{lemma}\label{lem: variance-multigraphs-extension2}
There exists a universal constant $c>0$ such that with probability at least $c$ with respect to the randomness of $\tilde f$,
\begin{align*}
 \sum_{\underset{G_2\in \categ([1,\cconst])}{G_1\in \BipGSet_n(d)}} &\Pconfig(G_1)\Pconfig(G_2) \big(\tilde f(G_1)-\tilde f(G_2)\big)^2 \\
 &\geq c\,\Pconfig\big(\BipGSet_n(d)\big)\sum_{\underset{G'\in \Expset(\zconst)}{G_1\in \BipGSet_n(d)}} \Psimple(G_1)\Psimple(G') \big(\tilde f(G_1)-f(G')\big)^2.
\end{align*}
\end{lemma}
\begin{proof}
Given $(G_1,G_2)\in \BipGSet_n(d)\times  \categ([1,\cconst])$, we have 
$$
\big(\tilde f(G_1)-\tilde f(G_2)\big)^2= \big(\tilde f(G_1)-h(G_2)\big)^2+ \xi_{G_2}^2 w(G_2)-2\xi_{G_2}  \big(\tilde f(G_1)-h(G_2)\big)\sqrt{w(G_2)}. 
$$
We remark that $\tilde f(G_1)$ is non-random and is entirely determined by the structure of $f$.
Denote by $\tilde \eta_{G_1,G_2}$ the indicator of the event that $$
\vert \xi_{G_2}\vert \geq 1\quad \text{ and }\quad \xi_{G_2}\sum_{G_1\in  \BipGSet_n(d)} \Pconfig(G_1) \big(\tilde f(G_1)-h(G_2)\big)\leq 0.
$$ 
With this definition, we can write 
\begin{align*}
\alpha&:= \sum_{\underset{G_2\in \categ([1,\cconst])}{G_1\in \BipGSet_n(d)}} \Pconfig(G_1)\Pconfig(G_2) \big(\tilde f(G_1)-\tilde f(G_2)\big)^2\\
 & \geq 
  \sum_{\underset{G_2\in \categ([1,\cconst])}{G_1\in \BipGSet_n(d)}} \tilde \eta_{G_1,G_2}\frac{\Pconfig(G_1)\Pconfig(G_2)}{\vert \SNeigh(G_2)\vert} \sum_{G'\in \SNeigh(G_2)} \Big(\big(\tilde f(G_1)-h(G_2)\big)^2 +  \big( f(G')-h(G_2)\big)^2\Big)\\
  &\geq \beta:= \frac{1}{2}\sum_{\underset{G_2\in \categ([1,\cconst])}{G_1\in \BipGSet_n(d)}} \tilde \eta_{G_1,G_2}\frac{\Pconfig(G_1)\Pconfig(G_2)}{\vert \SNeigh(G_2)\vert} \sum_{G'\in \SNeigh(G_2)} \big(\tilde f(G_1)-f(G')\big)^2,
\end{align*}
where we made use of the Cauchy--Schwarz inequality in the last line. 
Note that 
$$
\beta\leq \frac{1}{2}\sum_{\underset{G_2\in \categ([1,\cconst])}{G_1\in \BipGSet_n(d)}} \frac{\Pconfig(G_1)\Pconfig(G_2)}{\vert \SNeigh(G_2)\vert} \sum_{G'\in \SNeigh(G_2)} \big(\tilde f(G_1)-f(G')\big)^2,
$$
and 
$$
\Exp \beta \geq c\sum_{\underset{G_2\in \categ([1,\cconst])}{G_1\in \BipGSet_n(d)}} \frac{\Pconfig(G_1)\Pconfig(G_2)}{\vert \SNeigh(G_2)\vert} \sum_{G'\in \SNeigh(G_2)} \big(\tilde f(G_1)-f(G')\big)^2, 
$$
for some universal constant $c$. Therefore, using a reverse Markov inequality, we get that with probability at least $c'$ 
\begin{align*}
\alpha&\geq c' \sum_{\underset{G_2\in \categ([1,\cconst])}{G_1\in \BipGSet_n(d)}} \frac{\Pconfig(G_1)\Pconfig(G_2)}{\vert \SNeigh(G_2)\vert} \sum_{G'\in \SNeigh(G_2)} \big(\tilde f(G_1)-f(G')\big)^2\\
&\geq c'\sum_{\underset{G'\in \Expset(\zconst)}{G_1\in \BipGSet_n(d)}}\big(\tilde f(G_1)-f(G')\big)^2 \sum_{\underset{G'\in \SNeigh(G_2)}{G_2\in \categ([1,\cconst]):}} \frac{\Pconfig(G_1)\Pconfig(G_2)}{\vert \SNeigh(G_2)\vert},
\end{align*}
for some universal constant $c'>0$. Using Lemmas~\ref{lem: nice-multigraph-prob and transition}, \ref{l: 2-9857-2985} and \ref{l: 98562098724}, we get that with the same probability,
\begin{align*}
\alpha&\geq
c'\sum_{\underset{G'\in \Expset(\zconst)}{G_1\in \BipGSet_n(d)}}\big(\tilde f(G_1)-f(G')\big)^2 
\sum_{k=1}^{\cconst}\sum_{\underset{G'\in \SNeigh(G_2)}{G_2\in \categ(k):}} \frac{\Pconfig(G_1)\,\Pconfig(\BipGSet_n(d))}{2^k|\BipGSet_n(d)| (nd)^k}\\
&\geq \frac{c'}{2}\,\Pconfig\big(\BipGSet_n(d)\big)^2 \sum_{\underset{G'\in \Expset(\zconst)}{G_1\in \BipGSet_n(d)}}\Psimple(G_1) \Psimple(G')\big(\tilde f(G_1)-f(G')\big)^2 \sum_{k=1}^{\cconst} \frac{(d-1)^{2k}}{2^k k!}.
\end{align*}
It remains to use \eqref{eq: prob-simple}, together with our assumption on $\cconst$, to finish the proof. 
\end{proof}

We are now ready to prove the main statement of this section which provides the desired comparison between the variance of $f$ and its extension. 
\begin{theor}[Variance comparison]\label{th: variance comparison}
Let $f:\, \BipGSet_n(d)\to \R$ be a function on $\BipGSet_n(d)$ and let $\tilde f$ its extension. 
Then, with probability at least $c$ with respect to the randomness of $\tilde f$, we have
$$
\Var_{\Pconfig}(\tilde f)\geq c\,\Var_{\Psimple}(f),
$$
where $c>0$ is a universal constant.
\end{theor}
\begin{proof}
As before, we suppose that the mean of $f$ is zero.
First, assume that \eqref{eq: condition-blow up-def} is satisfied i.e. 
$$
\sum_{G_1,G_2\in \Expset(\zconst)}\Psimple(G_1)\Psimple(G_2) (f(G_1)-f(G_2))^2 
\geq \frac{1}{2^7}\Var_{\Psimple}(f).
$$
Note that using the Cauchy--Schwarz inequality, we obtain
$$
\sum_{\underset{\tilde d(G_1,G_2)< \gconst +8\cconst}{G_1,G_2\in  \BipGSet_n(d)}} \Psimple(G_1)\Psimple(G_2)\big(f(G_1)-f(G_2)\big)^2
\leq 4\sum_{\underset{\tilde d(G_1,G_2)< \gconst +8\cconst}{G_1,G_2\in  \BipGSet_n(d)}} \Psimple(G_1)\Psimple(G_2) f(G_1)^2,
$$
where $\tilde d(G_1,G_2)=\vert\{(i,j)\in [n]\times [n]:\, \Adj(G_1)_{ij}\neq \Adj(G_2)_{ij}\}\vert$. 
Now, note that given $G_1\in \BipGSet_n(d)$ and $k< \gconst +8\cconst$, there are at most ${n^2 \choose k}$ graphs $G_2\in \BipGSet_n(d)$ satisfying $\tilde d(G_1,G_2)=k$.  
Therefore, for $n$ sufficiently large, we deduce from the above that
\begin{align*}
\sum_{\underset{\tilde d(G_1,G_2)< \gconst +8\cconst}{G_1,G_2\in  \BipGSet_n(d)}} \Psimple(G_1)\Psimple(G_2)\big(f(G_1)-f(G_2)\big)^2
&\leq  4\Var_{\Psimple}(f) \sum_{k< \gconst +8\cconst} \frac{{n^2\choose k}}{\vert\BipGSet_n(d)\vert}\\
&\leq \frac{1}{2^8} \Var_{\Psimple}(f),
\end{align*}
where we made use of \eqref{eq: size-simple} and the choice of parameters in \eqref{eq: parameters} and \eqref{eq: assumption-d}. 
Putting the above estimates together, we get
$$
\sum_{\underset{\tilde d(G',G'')\geq \gconst +8\cconst}{G',G''\in \Expset(\zconst)}}\Psimple(G')\Psimple(G'') (f(G')-f(G''))^2 
\geq \frac{1}{2^8}\Var_{\Psimple}(f).
$$
With this in hand, a direct application of Lemma~\ref{lem: variance-multigraphs-extension} yields the desired result. 

Now, suppose that \eqref{eq: condition-blow up-def} is violated.
By Lemma~\ref{l: 972-987-9587}, we have
%
\begin{equation}\label{eq1: proof-lower bound variance}
\sum_{G'\in \Expset(\zconst)}\Psimple(G')f(G')^2
\leq \frac{1}{2^6}\Var_{\Psimple}(f),
\end{equation}
or, equivalently,
\begin{equation}\label{eq2: proof-lower bound variance}
\sum_{G_1\in \BipGSet_n(d)\setminus  \Expset(\zconst)}\Psimple(G_1) f(G_1)^2\geq \frac{2^6-1}{2^6}\Var_{\Psimple}(f). 
\end{equation}
Applying Lemma~\ref{lem: variance-multigraphs-extension2}, we have that with probability at least $c$ 
\begin{align*}
\Var_{\Pconfig}(\tilde f)&\geq c\,\Pconfig\big(\BipGSet_n(d)\big)\sum_{\underset{G'\in \Expset(\zconst)}{G_1\in \BipGSet_n(d)}} \Psimple(G_1)\Psimple(G') \big(\tilde f(G_1)-f(G')\big)^2\\
&\geq c\,\Pconfig\big(\BipGSet_n(d)\big)\sum_{\underset{G'\in \Expset(\zconst)}{G_1\in \BipGSet_n(d)\setminus  \Expset(\zconst)}} \Psimple(G_1)\Psimple(G') \big(\tilde f(G_1)-f(G')\big)^2,
\end{align*}
for some universal constant $c>0$. Expanding the square and using the definition of $\tilde f$, we get that with probability at least $c$ 
\begin{align*}
\Var_{\Pconfig}(\tilde f)&\geq c\,\Psimple\big(\Expset(\zconst)\big)\sum_{G_1\in \BipGSet_n(d)\setminus  \Expset(\zconst)}\Psimple(G_1) f(G_1)^2 \\
&\quad -2c\,\sqrt{\Pconfig\big(\BipGSet_n(d)\big)}\sum_{G_1\in \BipGSet_n(d)\setminus  \Expset(\zconst)}\Psimple(G_1) f(G_1) \sum_{G'\in \Expset(\zconst)} \Psimple(G')f(G')\\
&\geq c\,\frac{2^6-1}{2^6} \Psimple\big(\Expset(\zconst)\big)\Var_{\Psimple}(f) -\frac{c}{4} \Var_{\Psimple}(f),
\end{align*}
where the last inequality follows after using Cauchy-Schwarz inequality together with  \eqref{eq1: proof-lower bound variance} and \eqref{eq2: proof-lower bound variance}. 
The proof is finished by noting that $\Psimple\big(\Expset(\zconst)\big)\geq 1-2^{-8}$ for $n$ sufficiently large. 
\end{proof}

\medskip

Note that Theorem~\ref{th: poincare} can be obtained as a simple combination of Theorems~\ref{th: main comparison dirichlet} and~\ref{th: variance comparison}.
Indeed, applying Markov's inequality, we get that with probability at least $c/2$,
$$
\Var_{\Pconfig}(\tilde f)\geq c\,\Var_{\Psimple}(f)\quad\mbox{ and }\quad
\Dir_{\Pconfig}(\tilde f, \tilde f)\leq \tilde C\,\Dir_{\Psimple}(f,f)+\tilde C\,\frac{d}{n^2} \Var_{\Psimple}(f)
$$
for some universal constant $\tilde C>0$, whereas, according to Proposition~\ref{prop: poincare-logsob-config}, we have (deterministically)
$$
\Var_{\Pconfig}(\tilde f)\leq C'\,nd\,\Dir_{\Pconfig}(\tilde f, \tilde f).
$$
This immediately implies
$$
\Var_{\Psimple}(f)\leq \frac{C'}{c}\,nd\,\bigg(\tilde C\,\Dir_{\Psimple}(f,f)+\tilde C\,\frac{d}{n^2} \Var_{\Psimple}(f)\bigg),
$$
and Theorem~\ref{th: poincare} follows.

\section{Logarithmic Sobolev inequality}\label{sec: log-sob}

The goal of this section is to prove Theorem~\ref{th: log-sob}. 
As discussed in the introduction, when $d$ is constant, the two probability measures $\Psimple$ and $\Pconfig$ are comparable 
and one can potentially use the techniques developed by Diaconis and Saloff-Coste \cite{DS-Comparison,DS-Comparison2}. 
More precisely, for any $G\in \BipGSet_n(d)$, we have 
\begin{equation}\label{eq: compare-proba-last}
\Pconfig(G)=\Pconfig\big(\BipGSet_n(d)\big)\Psimple(G)\geq \frac12 e^{-\frac{(d-1)^2}{2}} \Psimple(G),
\end{equation}
where we made use of \eqref{eq: prob-simple}. 
With this comparison in hand, one can readily apply \cite[Theorem~4.1.1]{saloff} provided we associate to any $f:\, \BipGSet_n(d)\to \R$ 
an extension to $\ConfBipGSet_n(d)$ such that the corresponding Dirichlet forms are comparable
(in this case, we have to work with {\it actual} function extensions, so we cannot directly use our function $\tilde f$ since
it may differ from $f$ on a small subset of simple graphs). 
Given $f:\, \BipGSet_n(d)\to \R$, we will define $\hat{f}:\, \ConfBipGSet_n(d)\to \R$ as follows:
$$
\hat{f}(G)=\begin{cases} 
f(G) & \text{ if $G\in \BipGSet_n(d)$};\\
\tilde f(G)& \text{ if $G\in \ConfBipGSet_n(d)\setminus \BipGSet_n(d)$},
\end{cases}
$$
where $\tilde f$ is the function from Definition~\ref{def: extension}. Thus, the only difference of $\hat f$ from $\tilde f$
is that we make sure the function $f$ and its extension agree on $\BipGSet_n(d)$.
The statement of \cite[Theorem~4.1.1]{saloff} asserts that if $\Dir_{\Pconfig}(\hat{f},\hat{f})\leq A\Dir_{\Psimple}(f,f)$ (for some realization of $\hat{f}$), then combined with \eqref{eq: compare-proba-last} this implies that 
$$
\alpha_{LS}(Q_u)\leq 2Ae^{\frac{(d-1)^2}{2}}\alpha_{LS}(Q_c), 
$$
where $\alpha_{LS}(Q_u)$ and $\alpha_{LS}(Q_c)$ denote the log-Sobolev constants of the switch chain on $\BipGSet_n(d)$ and $\ConfBipGSet_n(d)$ respectively. 
Thus, Theorem~\ref{th: log-sob} follows from Proposition~\ref{prop: poincare-logsob-config} once we determine the constant $A$ in the above comparison. 
The details on obtaining the latter are indicated in the following lemma which completes the proof. 

\begin{lemma}
There exists a universal constant $C>0$ such that the following holds. 
Let $f:\, \BipGSet_n(d)\to \R$ and let $\hat{f}$ be its (random) extension defined above. 
Then, we have $\Exp\, \Dir_{\Pconfig}(\hat{f},\hat{f})\leq  C\Dir_{\Psimple}(f,f)$. 
\end{lemma}
\begin{proof}
The proof is actually identical to the one in Section~\ref{sec: Dirichlet}. One starts similarly by partitioning the pairs of adjacent multigraphs into 
the sets $\Gamma_i$, $i=1,\ldots,6$ (see beginning of Section~\ref{sec: Dirichlet}). Since $\hat{f}$ and $\tilde f$ coincide on $\ConfBipGSet_n(d)\setminus \BipGSet_n(d)$, then one only needs to modify the treatment of the sets $\Gamma_1$ and $\Gamma_2$ which is done in Lemma~\ref{lem: Dirichlet-gamma1}.  
Since $\hat{f}$ and $f$ coincide on $\BipGSet_n(d)$, then one only needs to bound the corresponding sum over $\Gamma_2$. A quick look at Lemma~\ref{lem: Dirichlet-gamma1} reveals that the corresponding bound for the sum over $\Gamma_2$ follows from \eqref{eq: dirichlet-gamma2} since $\hat{f}$ coincides with $f$ on $\BipGSet_n(d)$. 
Therefore, we deduce similarly to Theorem~\ref{th: main comparison dirichlet} that 
$$
\Exp\, \Dir_{\Pconfig}(\hat{f},\hat{f})\leq C\Dir_{\Psimple}(f,f)+ C\frac{d}{n^2} \Var_{\Psimple}(f),
$$
for some universal constant $C$. Applying Theorem~\ref{th: poincare} and using that $d$ is constant, the result follows. 
\end{proof}

\end{document}